\documentclass[12pt]{amsart}
\usepackage{amssymb}
\newcommand{\C}{\mathbb{C}}
\newcommand{\Hcal}{\mathcal{H}}
\newcommand{\param}{\mathfrak{c}}
\newcommand{\paramq}{\mathfrak{c}}
\newcommand{\Ocat}{\mathcal{O}}
\newcommand{\Str}{O}
\newcommand{\A}{\mathcal{A}}
\newcommand{\Pro}{\mathcal{P}}
\newcommand{\Loc}{\operatorname{Loc}}
\newcommand{\codim}{\operatorname{codim}}
\newcommand{\Coh}{\operatorname{Coh}}
\newcommand{\Hom}{\operatorname{Hom}}
\newcommand{\Ext}{\operatorname{Ext}}
\newcommand{\Z}{\mathbb{Z}}
\newcommand{\Ca}{\mathsf{C}}
\newcommand{\Q}{\mathbb{Q}}

\newcommand{\g}{\mathfrak{g}}
\newcommand{\h}{\mathfrak{h}}
\newcommand{\bor}{\mathfrak{b}}
\newcommand{\lf}{\mathfrak{l}}
\newcommand{\Orb}{\mathbb{O}}
\newcommand{\quo}{/\!/}
\newcommand{\ring}{\mathsf{R}}
\newcommand{\Weyl}{\mathbf{A}}
\newcommand{\HC}{\operatorname{HC}}
\newcommand{\B}{\mathcal{B}}
\newcommand{\F}{\mathbb{F}}
\newcommand{\End}{\operatorname{End}}
\newcommand{\Ecal}{\mathcal{E}}
\newcommand{\gr}{\operatorname{gr}}

\newcommand{\U}{\mathcal{U}}
\newcommand{\M}{\mathcal{M}}
\newcommand{\GL}{\operatorname{GL}}
\newcommand{\VA}{\operatorname{V}}

\newcommand{\Pic}{\operatorname{Pic}}
\newcommand{\Fr}{\operatorname{Fr}}
\newcommand{\R}{\mathbb{R}}
\newtheorem{Thm}{Theorem}[subsection]
\newtheorem{Prop}[Thm]{Proposition}
\newtheorem{Cor}[Thm]{Corollary}
\newtheorem{Lem}[Thm]{Lemma}
\theoremstyle{definition}
\newtheorem{Ex}[Thm]{Example}
\newtheorem{defi}[Thm]{Definition}
\newtheorem{Rem}[Thm]{Remark}
\newtheorem{Conj}[Thm]{Conjecture}
\numberwithin{equation}{section}
\title{Localization theorems for quantized symplectic resolutions}
\author{Ivan Losev}
\thanks{Yale University, New Haven, CT, and IAS, Princeton, NJ}
\thanks{e-mail: ivan.loseu@gmail.com}
\begin{document}
\begin{abstract}
The goal of this paper is to establish Beilinson-Bernstein type localization theorems
for quantizations of some conical symplectic resolutions. We prove the full
localization theorems for finite and affine type A Nakajima quiver varieties.
The proof is based on two partial results that hold in  more general situations.
First, we establish an exactness result for global section functor if there is
a tilting generator that has a rank 1 summand. Second, we examine when the global
section functor restricts to an equivalence between categories $\mathcal{O}$.
\end{abstract}
\maketitle
\tableofcontents
\section{Introduction}
\subsection{Beilinson-Bernstein localization theorems}\label{SS_BB_review}
In this section we recall classical results of Beilinson and Bernstein,
\cite{BB_abel,BB_derived} on the localization theorems for semisimple
Lie algebras. These theorems are of crucial importance for Lie representation
theory.

Let $G$ be a semisimple algebraic group over $\C$, $\g$ be its Lie algebra
and $\U:=U(\g)$ be the universal enveloping algebra. We pick Cartan and
Borel subalgebras $\h\subset \bor\subset \g$ and let $W$ denote the
Weyl group. The center of $\U$ is identified with $\C[\h^*]^W$
by means of the Harish-Chandra isomorphism (we use the usual Weyl group
action on $\mathfrak{h}^*$ so that the central character of the trivial $\g$-module
is $W\rho$, where $\rho\in \h^*$ denotes half the sum of positive roots). So, for $\lambda\in \h^*$,
we can consider the central reduction $\U_\lambda$ of $\U$ by the maximal
ideal in $\C[\h^*]^W$ consisting of all elements that vanish at $\lambda$.

On the other hand, consider the flag variety $\mathcal{B}$ for $\g$.
 Let
$\mathcal{D}_\lambda$ denote the sheaf of $(\lambda-\rho)$-twisted differential
operators on $\mathcal{B}$. The algebra of global sections $\Gamma(\mathcal{D}_\lambda)$
is identified with $\U_\lambda$, while the higher cohomology spaces of $\mathcal{D}_\lambda$ vanish.
So we can consider the global section functor $\Gamma: \Coh(\mathcal{D}_\lambda)
\rightarrow \U_\lambda\operatorname{-mod}$ and its derived functor
$R\Gamma: D^b(\Coh(\mathcal{D}_\lambda))
\rightarrow D^b(\U_\lambda\operatorname{-mod})$. The functor $\Gamma$
has left adjoint, the localization functor $\Loc:=\mathcal{D}_\lambda\otimes_{\U_\lambda}\bullet$.
Its derived functor $L\Loc: D^-(\U_\lambda\operatorname{-mod})\rightarrow
D^-(\Coh(\mathcal{D}_\lambda))$ is left adjoint of $R\Gamma$ viewed as a functor
between the $D^-$ categories.

The following theorem is the main result of \cite{BB_abel}.

\begin{Thm}\label{Thm:BB_abel}
Suppose that $\langle\lambda,\alpha^\vee\rangle\not\in \Z_{< 0}$ for all positive
coroots $\alpha^\vee$. Then $\Gamma$ is an exact functor. Moreover, if
$\langle\lambda,\alpha^\vee\rangle\not\in \Z_{\leqslant 0}$ for all positive
coroots $\alpha^\vee$, then $\Gamma$ is an equivalence.
\end{Thm}

The derived analog of this theorem was obtained in \cite{BB_derived}.

\begin{Thm}\label{Thm:BB_derived}
Suppose that $\langle\lambda,\alpha^\vee\rangle\neq 0$ for all
coroots $\alpha^\vee$. Then $R\Gamma$ is an equivalence.
\end{Thm}

\subsection{Generalizations to quantized symplectic resolutions}
The two theorems quoted above fit into a more general framework of
{\it quantized conical symplectic resolutions}. Here we start with a pair
of varieties $X$ and $Y$, where $Y$ is affine and a morphism $X\rightarrow Y$ that is
a resolution of singularities. These data is supposed to satisfy
a number of assumptions, most importantly, that $X$ is symplectic,
i.e., carries an algebraic symplectic form, and that the algebra
$\C[Y]$ is positively graded. For example, for
$Y$ one can take the nilpotent cone in $\g$ and for $X$
its Springer resolution, that is, $X=T^*\mathcal{B}$.

One can talk about filtered quantizations of $X$. Roughly speaking,
those are sheaves of filtered algebras on $X$ that deform the structure
sheaf $\Str_X$ viewed as a sheaf of Poisson algebras. For example,
we can view $\mathcal{D}_\lambda$ (or, more precisely, its
{\it microlocalizaton}) as a filtered quantization of $T^*\mathcal{B}$.
In the general case, the filtered quantizations of $X$ are parameterized
by the space $\paramq:=H^2(X,\C)$, see \cite[Section 2.3]{quant_iso}.

For $\lambda\in \paramq$, we will write $\A_\lambda^\theta$ for the filtered
quantization of $X$ corresponding to $\lambda$ (here  $\theta$ is an element of
$H^2(X,\mathbb{R})$ that encodes a choice of $X$, we will elaborate on this below,
Section \ref{SSS_classif_resol}).
The algebra $\A_\lambda:=\Gamma(\A_\lambda^\theta)$ is a
quantization of a filtered algebra $\C[Y]$ and $R^i\Gamma(\A_\lambda^\theta)=0$
for $i>0$. Further, it makes sense to speak about the category
$\Coh(\A_\lambda^\theta)$ of coherent $\A_\lambda^\theta$-modules,
see, for example, the introduction to \cite[Section 4]{BPW}
or \cite[Section 2.3]{BL}. So we can still consider
the global section functor $\Gamma: \Coh(\A_\lambda^\theta)
\rightarrow \A_\lambda\operatorname{-mod}$, its left adjoint,  the localization
functor $\Loc:\A_\lambda\operatorname{-mod}\rightarrow
\Coh(\A_\lambda^\theta)$ as well as their derived analogs.

\begin{defi}\label{defi:localization}
We say that {\it abelian localization holds} for $\A_\lambda^\theta$ if $\Gamma,\Loc$
are  mutually inverse equivalences between $\Coh(\A_\lambda^\theta)$ and $\A_\lambda\operatorname{-mod}$.
We say that {\it derived localization holds} for $\A_\lambda^\theta$ if $R\Gamma$
and $L\Loc$ are mutually inverse equivalences between $D^b(\Coh(\A_\lambda^\theta))$
and $D^b(\A_\lambda\operatorname{-mod})$.
\end{defi}

Questions to be investigated in this paper are when (i.e., for which $\lambda,\theta$)
\begin{enumerate}
\item does abelian localization hold?
\item is $\Gamma$ exact?
\item does derived localization hold?
\end{enumerate}
Note that the affirmative answer to (1) is equivalent to the affirmative answers to
both (2) and (3).

We will explain standard by now conjectures in Section \ref{SS_conjectures}
and known results in Section \ref{SS_known_results}. Then we will proceed to
explaining results of this paper.

\subsection{Conjectures}\label{SS_conjectures}
Now we state three conjectures related to questions (1)-(3). The first one characterizes
the locus of parameters where derived localization holds in terms of the algebraic properties
of $\A_\lambda$.

\begin{Conj}\label{Conj:fin_hom_dim}
The following two conditions are equivalent:
\begin{enumerate}
\item The algebra $\A_\lambda$ has finite homological dimension.
\item Derived localization holds for $\A_\lambda^\theta$.
\end{enumerate}
\end{Conj}

Note also that (2) easily implies (1). Conjecture \ref{Conj:fin_hom_dim}
implies, in particular, that whether derived localization holds depends only on $\A_\lambda$ and not
on the choice of $X$. In particular, if abelian localization holds for $\A_\lambda^\theta$, then derived localization
holds for $\A_\lambda^{\theta'}$ with all $\theta'$.
In all examples that we know the converse is also true:
if $\A_\lambda$ has finite homological dimension, then abelian localization
holds for  $\A_\lambda^\theta$ with some $\theta$.

The next two conjectures give a description of the locus of parameters $\lambda$
where abelian and derived localizations hold.  In order to state them we need some
terminology.

In \cite{Namikawa_birational} Namikawa  constructed a finite collection of rational
hyperplanes in $H^2(X,\C)$ that control  the deformations of both $X$ and $Y$
and also the classification of (partial) symplectic resolutions of $Y$. We call these
hyperplanes the {\it classical walls}. We will elaborate on them below
in Section \ref{SS_resolution_generalities}. For example,
when $X=T^*\mathcal{B}$, then the classical walls are exactly  the root hyperplanes.
We will also need an integral lattice  $H^2(X,\C)$: the image of
$\operatorname{Pic}(X)$ under $c_1$. For $X=T^*\mathcal{B}$ this lattice is the
weight lattice.

Starting from the classical walls we can define two collections of affine hyperplanes
in $H^2(X,\C)$ viewed as a parameter space for quantizations: {\it essential} hyperplanes and {\it singular} hyperplanes. We will give precise definitions later, Section
\ref{SS_essent_sing}, but will highlight certain features here:
\begin{enumerate}
\item Each singular hyperplane is essential. A hyperplane is essential
if and only if it is obtained from a singular one via a shift by an
integral element.
\item There are finitely many singular hyperplanes and each of them
is parallel to a classical wall.
\end{enumerate}

 For example,
when $X=T^*\mathcal{B}$, then the singular hyperplanes are again exactly the
root hyperplanes.

The essential hyperplanes can be characterized as follows: the complement to their
union consists precisely of parameters $\lambda$ such that $\A_\lambda$ is simple
and abelian localization holds for $\A_\lambda^\theta$ for all choices of $\theta$.
It is expected that a hyperplane is singular if for a ``very generic'' element $\lambda$
in the hyperplane, the algebra $\A_\lambda$ has infinite homological dimension.
A precise definition is more technical and will be given later.

Now we explain how to determine the singular hyperplanes. If $\dim H^2(X,\C)>1$, then
we can reduce the computation to that for ``smaller'' symplectic resolutions
(that are expected to have 1-dimensional second cohomology spaces). For example,
for $X=T^*\mathcal{B}$ we will get a reduction to $T^*\mathbb{P}^1$, which is a very
easy case. It is known how to completely determine the singular hyperplanes in a
number of examples.

The following conjecture concerns the locus of $\lambda$ where derived localization holds
(equivalently, modulo Conjecture \ref{Conj:fin_hom_dim}, $\A_\lambda$ has infinite homological
dimension).

\begin{Conj}\label{Conj:der_loc}
The algebra $\A_\lambda$ has infinite homological dimension if and only if $\lambda$
lies in a singular hyperplane.
\end{Conj}

\begin{defi}\label{defi:sing_parameters}
Below we call $\lambda\in \paramq$ {\it singular} if it lies in a singular hyperplane
and {\it regular} else.
\end{defi}

Now we proceed to abelian localization. We assume that
\begin{itemize}
\item[(*)] if an essential hyperplane $\Upsilon'$ {\it lies between} two singular hyperplanes,
$\Upsilon^1, \Upsilon^2$, then it is singular. A rigorous meaning of  ``lying between'' is as follows:
for every (equivalently, for some) $\lambda\in \Upsilon'$, the cone spanned by the subsets $\{\mu\in H^2(X,\R)| \lambda+\mu\in \Upsilon^i\}$  span $H^2(X,\R)$.
\end{itemize}
This condition holds in all examples we know, it will be a part of a formal definition of a singular
hyperplane.

Now take a regular parameter $\lambda$. Let $\Upsilon_1,\ldots,\Upsilon_k$ be all classical
walls that are parallel to essential hyperplanes containing $\lambda$. For example, if $\lambda$
is ``very generic'' on an essential wall, then $k=1$. In the general case, the hyperplanes $\Upsilon_1,\ldots,\Upsilon_k$
split $H^2(X,\mathbb{R})$ into chambers. Thanks to (*), $\lambda$ lies in the same  half space
with respect to all singular hyperplanes parallel to $\Upsilon_i$, and so specifies
a chamber.

\begin{Conj}\label{Conj:abelian}
Suppose that (*) holds.
Then the following two conditions are equivalent:
\begin{enumerate}
\item Abelian localization holds for the quantization $\A_\lambda^\theta$ of $X$.
\item The ample cone of $X$ lies in the chamber determined by $\lambda$ as in the previous paragraph.
\end{enumerate}
\end{Conj}

In fact, we will see below, Section \ref{SSS_conj_equiv}, that, modulo Conjecture \ref{Conj:fin_hom_dim},
Conjecture \ref{Conj:abelian} follows
from Conjecture \ref{Conj:der_loc}. On the
other hand, Conjecture \ref{Conj:abelian} implies the hard part of
Conjecture \ref{Conj:der_loc} ($\A_\lambda$ has infinite homological dimension
implies that $\lambda$ is in a singular hyperplane).

We cannot prove either of the conjectures in the general situation. We will prove
some weaker or related statements under additional assumptions that hold in many
examples. This will allow us to recover most of the known cases of the conjectures
as well as to establish them in several new situations.

\subsection{Known results}\label{SS_known_results}
Now we want to describe what is known about Conjectures \ref{Conj:fin_hom_dim},
\ref{Conj:der_loc}, \ref{Conj:abelian}.

Conjecture \ref{Conj:fin_hom_dim} is known to hold in the following important
examples:
\begin{enumerate}
\item $X=T^*\mathcal{B}$, this is classical.
\item $X$ is a symplectic resolution of a symplectic quotient singularity, \cite{GL}.
Here the condition that $\A_\lambda$ has finite homological dimension is equivalent to
the claim that $\lambda$ is ``spherical'' thanks to a result of
Bezrukavnikov, \cite[Theorem 5.5]{Etingof_affine}.
Below we will elaborate on this in  a
more general situation.
\item $X$ is obtained as a Hamiltonian reduction of a symplectic vector space
and several technical conditions hold, see \cite[Assumptions 3.2]{MN_derived}.
In particular, these assumptions are satisfied when $X$ is  a smooth Nakajima
quiver variety (at least, for a finite or affine type quiver) or a smooth
hypertoric variety. This covers most of (2).
\end{enumerate}

Together with these results,  the present paper establishes
Conjecture \ref{Conj:fin_hom_dim} in almost all known examples.

Now we list several cases where abelian and derived localization theorems
has been established. These results are known to confirm Conjectures \ref{Conj:der_loc}
and \ref{Conj:abelian}, or these conjectures can be deduced from them
with a bit of work.

\begin{enumerate}
\item  For $X=T^*\mathcal{B}$, we have the original Beilinson-Bernstein theorems.
\item (1) is easily generalized to the case when we replace $\mathcal{B}$ with a
partial flag variety $G/P$. The original proof of Theorem \ref{Thm:BB_abel}
works in this case as well.
\item (1) also generalizes in a different way: to the case when $X$ is a Slodowy
variety, i.e. the preimage of  a Slodowy slice in $T^*\mathcal{B}$, see
\cite{DK,Ginzburg_HC}. This is a special case of a general phenomenon: if abelian
localization holds for $\A_\lambda^\theta$, it also holds for all ``slice quantizations'',
as we will see below, Section \ref{SSS_res_fun}.
The converse is not true, so for parabolic Slodowy varieties
(=preimages of Slodowy slices inside of $T^*(G/P)$) the exact locus where abelian
localization holds is not known, in general.
\item The next case that was studied was $X=\operatorname{Hilb}_n(\C^2)$.
Here abelian and derived localization theorems follow from
\cite{GS}, see also \cite{KR}.
\item The previous example can be generalized in two ways. First, we can replace
$Y=(\C^2)^n/S_n$ with $Y=(\C^2)^n/\Gamma_n$, where $\Gamma_n=S_n\ltimes (\Z/\ell\Z)$.
The algebras $\A_\lambda$ are precisely the spherical rational Cherednik algebras.
The aspherical locus (=the locus, where $\A_\lambda$ has infinite homological
dimension) was determined in \cite{DG}. And \cite{cycl_ab_loc} proves a version of
an abelian localization theorem.
\item Another generalization of (5) is the case when $X$ is the Gieseker moduli
space, the smooth Nakajima quiver variety associated to the Jordan quiver
 with arbitrary dimension and framing.
Here abelian and derived localization theorems were established in
\cite[Section 5]{Gies}. We will recall these results below as we will need them.
\item Finally, when $X$ is a hypertoric variety, then  abelian localization
was proved in \cite{Bellamy_Kuwabara}.
\end{enumerate}

We should also mention results of \cite{MN_exact} that, in particular, give sufficient conditions
on $\lambda$ for the global section functor to be exact in the case when $X,Y$ are Hamiltonian reductions
(of symplectic vector space by a reductive group action). In general, these conditions are not necessary.
However, it looks like they are necessary in the case when we deal with Nakajima quiver varieties
for a quiver with a single vertex.

We would like to also mention a general result from \cite{catO_charp}. It says that weak versions
of Conjectures \ref{Conj:der_loc}, \ref{Conj:abelian} hold: instead of singular hyperplanes
that we can explicitly describe in many cases, we need to remove an unspecified (but finite)
collection of essential hyperplanes that are obtained from singular ones by integral shifts. The method
of proof does not allow to even get an estimate on how many hyperplanes we need to remove.

\subsection{Tilting bundles and exactness of global section functor}
We have two general results on Conjectures \ref{Conj:der_loc} and
\ref{Conj:abelian}. The first one shows that
under a suitable assumption on $X$ that we will explain below in this section the global section
functor is exact if $\lambda,\theta$ are as in Corollary \ref{Conj:abelian}.

Now we introduce our assumption on $X$. Kaledin, \cite{Kaledin}, constructed a
{\it tilting bundle} on $X$ to be denoted by $E^\theta$. Recall that this means
that $E^\theta$ is a vector bundle on $X$ without higher self-extensions and such that
$H:=\End(E^\theta)$ has finite homological dimension.  By
\cite[Proposition 2.2]{BK} the functor $R\Gamma(E^\theta\otimes\bullet)$ is an
equivalence $D^b(\Coh (X))\xrightarrow{\sim}D^b(H\operatorname{-mod})$.

The assumption we need is as follows.

\begin{itemize}
\item[($\heartsuit$)] The bundle $E^\theta$ has a direct summand that is a line
bundle.
\end{itemize}

We can actually assume that this summand is the structure sheaf $\Str_X$.

($\heartsuit$) is a pretty subtle property, which is difficult to check. A priori,
it depends on a choice of $E^\theta$. It is known
to hold in three families of examples.

\begin{itemize}
\item $X=T^*(G/P)$ for an arbitrary semisimple algebraic group $G$ and its parabolic
subgroup $P$. This follows from results of \cite{BMR,BMR_sing} (for $P=B$)
and \cite{BM} (in the general case).
In fact, this gives rise to a tilting generator satisfying ($\heartsuit$)
on any parabolic Slodowy variety.
\item $X$ is a symplectic resolution of a symplectic quotient singularity.
Here a tilting bundle satisfying ($\heartsuit$) (a Procesi bundle)
was constructed in \cite{BK}.
\item Finally, $X$ is a smooth Coulomb branch of a gauge theory, \cite{Webster_tilting}.
This includes hypertoric varieties, finite and affine type A Nakajima quiver varieties
as special cases.
\end{itemize}

We can quantize $E^\theta$ to a right $\A_\lambda^\theta$-module $\Ecal^\theta_\lambda$.
This gives a sheaf of algebras $\Hcal_\lambda^\theta:=\mathcal{E}nd(\Ecal^\theta_\lambda)$
that quantizes $\mathcal{E}nd(E^\theta)$. We set $\Hcal_\lambda:=\Gamma(\Hcal_\lambda^\theta)$,
this algebra quantizes $H$. We have a functor $\tilde{\Gamma}:D^b(\Coh(\A_\lambda^\theta))
\rightarrow D^b(\Hcal_\lambda\operatorname{-mod})$ given by $\Gamma(\Ecal^\theta_\lambda\otimes_{\A_\lambda^\theta}\bullet)$. It is easy to
see, compare to \cite[Section 5]{GL}, that $R\tilde{\Gamma}$ is always an equivalence.

In the case when $X$ is a resolution of a symplectic quotient singularity and
$E^\theta$ is a Procesi bundle, $\Hcal_\lambda$ is a symplectic reflection algebra from \cite{EG}.
We do not know other cases where $\Hcal_\lambda$ has a similarly explicit description.

Thanks to ($\heartsuit$) we get an idempotent $e\in \Hcal_\lambda$ such that $e\Hcal_\lambda e=\A_\lambda$. From \cite[Theorem 5.5]{Etingof_affine} one can deduce that Conjecture \ref{Conj:fin_hom_dim} holds. Note that $\Gamma=e\tilde{\Gamma}$.

This is the first of our main results.

\begin{Thm}\label{Thm:exactness}
We suppose that ($\heartsuit$) (and (*)) hold.
Let $X$ and $\lambda$ be as in Conjecture \ref{Conj:abelian}.
Then $\tilde{\Gamma}: \Coh(\A_\lambda^\theta)\rightarrow
\Hcal_\lambda\operatorname{-mod}$ is an equivalence of abelian categories and, therefore,
$\Gamma$ is exact.
\end{Thm}

In fact, we expect that the locus where $\tilde{\Gamma}$ is an equivalence is larger than
what we have in the theorem. We will comment on this below, see Section
\ref{SSS_enhanced_localization}.

The proof of Theorem \ref{Thm:exactness} is based on analyzing compositions of  {\it translation
functors},  right t-exact derived equivalences between the categories $D^b(\Hcal_\lambda\operatorname{-mod})$
and $D^b(\Hcal_{\lambda+\chi}\operatorname{-mod})$ for integral $\chi$.

The same argument in the general setting implies that Conjectures \ref{Conj:fin_hom_dim},
\ref{Conj:der_loc} imply Conjecture \ref{Conj:abelian}.

\subsection{$\mathcal{O}$-regular parameters}\label{SS_O_reg_intro}
Conjecture \ref{Conj:der_loc} currently seems to be out of reach. However, we are able to
control a related regularity condition on $\lambda$ that makes sense under an additional
assumption on $X$. Throughout this section we assume that  there is a Hamiltonian  action of a torus $T$ on $X$ with finitely many fixed points.

Here are some examples when this happens.

\begin{enumerate}
\item We can take $X=T^*(G/P)$ or, more generally, parabolic Slodowy varieties
for principal Levi nilpotent elements.
\item Nakajima quiver varieties for finite or affine type A quivers.
\item Hypertoric varieties.
\item Note that the varieties in (2) and (4) are special cases of smooth Coulomb
branches for gauge theories. Other examples here include the symplectic duals
to some quiver varieties associated to Dynkin quivers. Those are slices in affine Grassmanians of types A,D,E that admit symplectic resolutions.
\end{enumerate}

There are also examples among Nakajima quiver varieties for type $D,E$ finite quivers.

When $T$ acts on $X$ with finitely many fixed points,
it makes sense to speak about category $\mathcal{O}$
of $\A_\lambda$-modules generalizing the classical
BGG category $O$, see \cite[Section 3.2]{BLPW}. We denote this category by $\Ocat_\nu(\A_\lambda)$,
it depends on the choice of a generic one-parameter subgroup $\nu:\C^\times\rightarrow T$.
Namely, $\nu$ gives rise to a grading $\A_\lambda=\bigoplus_i \A_\lambda^i$. We can form
the subalgebra $\A_\lambda^{\geqslant 0}:=\bigoplus_{i\geqslant 0}\A_\lambda^i$ and
its two-sided dideal $\A_\lambda^{>0}$.

\begin{defi}\label{defi:O}
The category $\Ocat_\nu(\A_\lambda)$, by definition, consists of all finitely generated
$\A_\lambda$-modules, where $\A_\lambda^{>0}$ acts locally nilpotently.
\end{defi}

The subspace $\A_\lambda^{\geqslant 0}\cap \A_\lambda \A_\lambda^{>0}$
is also a two-sided ideal in $\A_\lambda^{\geqslant 0}$. So we can form the quotient
\begin{equation}\label{eq:Cartan_sub}
\Ca_\nu(\A_\lambda)=\A_{\lambda}^{\geqslant 0}/(\A_{\lambda}^{\geqslant 0}\cap \A_\lambda\A_\lambda^{>0}).
\end{equation}
This algebra was called the {\it B-algebra} in \cite{BLPW} and the {\it Cartan subquotient}
in \cite{catO_R}, we will use the latter name. This is a finite dimensional algebra.
A principal reason why $\Ca_\nu(\A_\lambda)$ is important in the study of $\Ocat_\nu(\A_\lambda)$ is that the simples in $\Ocat_\nu(\A_\lambda)$ are in a natural one-to-one correspondence with the irreducible
$\Ca_\nu(\A_\lambda)$-modules.

We will need the following notion.

\begin{defi}\label{defi:O_regular}
We say that $\lambda\in \paramq$ is {\it $\Ocat$-regular} if $\Ca_\nu(\A_\lambda)$ is a semisimple
commutative algebra of dimension equal to the number of $T$-fixed points in $X$. Otherwise,
we say that $\lambda$ is {\it $\Ocat$-singular}.
\end{defi}

One can show that the locus of $\Ocat$-singular parameters in $\paramq$ is Zariski closed,
we will provide a proof below. We denote this locus by $\paramq^{\Ocat-sing}$. It follows from
\cite[Section 4.4]{catO_charp} that $\paramq^{\Ocat-sing}$ is contained in a finite union of
essential hyperplanes. But, again, in general, we have no control on how many essential
hyperplanes one needs to take. One should expect the following to hold.

\begin{Conj}\label{Conj:O_singular}
The locus $\paramq^{\Ocat-sing}$ coincides with the union of singular hyperplanes.
\end{Conj}

The next theorem is the second main result of this paper.

\begin{Thm}\label{Thm:O_regular}
As before we assume that $T$ acts on $X$ via a Hamiltonian action with finitely
many fixed points.
Suppose $\Gamma: \Coh(\A_\lambda^\theta)
\rightarrow \A_\lambda\operatorname{-mod}$ is exact. Then the intersection
of $\paramq^{O-sing}$ with a suitable neighborhood (in the usual complex topology)
of $\lambda$ has codimension $1$.
\end{Thm}

Let us explain a connection of this result with abelian localization. Following
\cite[Section 3.3]{BLPW}, it makes sense to speak about the category $\Ocat_\nu(\A_\lambda^\theta)$,
this is a full subcategory in $\Coh(\A_\lambda^\theta)$. The functor $\Gamma$
restricts to  $\Ocat_\nu(\A^\theta_\lambda)\rightarrow \Ocat_\nu(\A_\lambda)$.
If it is exact, then it is a Serre quotient functor. So if $\lambda$ is $\Ocat$-regular,
then the categories   $\Ocat_\nu(\A^\theta_\lambda),\Ocat_\nu(\A_\lambda)$
have the same number of simple objects. Hence $\Gamma$ gives an equivalence between
these categories.


Now assume that  $X^T$ is finite, and ($\heartsuit$)  holds (and so does (*) from
Section \ref{SS_conjectures}). In all examples that we know the singular hyperplanes
lie in $\paramq^{\Ocat-sing}$, so we assume that as well. Theorem \ref{Thm:exactness}
implies that outside of the union of singular hyperplanes, the functor $\Gamma:
\Coh(\A_\lambda^\theta)\rightarrow \A_\lambda\operatorname{-mod}$ is exact.
So Theorem \ref{Thm:O_regular} implies that $\paramq^{\Ocat-sing}$ is the union
of hyperplanes. Below we will see that, in the case when
$\dim \paramq>1$, the computation of these hyperplanes can be reduced to ``smaller''
resolutions.

Now assume that $\Gamma$ is exact and $\lambda\in \paramq^{\Ocat-reg}$.
Then $\Gamma^\theta_\lambda$ is a Serre quotient functor between two categories
with the same number of simples. So it is an equivalence.

Under some additional assumptions, if $\Gamma^\theta_\lambda$ is an equivalence
$\Ocat_\nu(\A^\theta_\lambda)\rightarrow \Ocat_\nu(\A_\lambda)$, then
$\Gamma^\theta_\lambda$ is an equivalence $\Coh(\A^\theta_\lambda)
\rightarrow \A_\lambda\operatorname{-mod}$. We will establish this implication
for finite and affine type A quivers.

\subsection{Special cases}
Now we would like to discuss special cases when we can establish the exactness of the global
section functor and/or the derived localization theorem using Theorems \ref{Thm:exactness}
and \ref{Thm:O_regular}.

\subsubsection{Springer resolution and generalization}
We get new proofs of Theorems \ref{Thm:BB_abel} and \ref{Thm:BB_derived}. However,
this is, perhaps, the hardest way to prove these theorems. In the case of $X=T^*(G/P)$
we can prove analogs of Theorems \ref{Thm:BB_abel}, \ref{Thm:BB_derived} modulo
determining the singular hyperplanes, we do not have an independent way of doing this.

\subsubsection{Nakajima quiver varieties of finite and affine type}
Here we are able to completely determine the singular hyperplanes, this is easy to do for all quivers of finite and affine type and was essentially done in \cite{BL,Gies}.
We then prove  Conjectures \ref{Conj:der_loc} and \ref{Conj:abelian}
for finite and affine type A quivers. We  also show that for affine type quivers
of types D and E with 1-dimensional framing at the extending vertex
(this gives resolutions of symplectic quotient singularities) the functor
$\tilde{\Gamma}$ is an equivalence (and hence $\Gamma$ is exact) on the
locus where we expect $\Gamma$ to be an equivalence.

We would like to point out an application of these results.
In \cite{BL,perv} the author (partly joint with Bezrukavnikov)
computed the number of finite dimensional irreducible representations
for the algebras $\A_\lambda$ associated to finite and affine type Nakajima
quiver varieties in the case when derived localization
holds for $\lambda$. Once we know that Conjecture \ref{Conj:der_loc},
we know an exact locus where the number of finite dimensional irreducible
representations is as found in \cite{BL,perv}.

\subsubsection{Other cases} Techniques of this paper should allow to reprove
Conjectures \ref{Conj:der_loc} and \ref{Conj:abelian} for hypertoric varieties
and also prove them for slices in affine Grassmanians of type D and E that admit symplectic
resolutions. We are not going to elaborate on these  families of varieties.

{\bf Acknowledgements}: I would like to thank Roman Bezrukavnikov, Yoshinori
Namikawa, and Ben Wesbter for stimulating discussions.  I am partially supported by the NSF under grant
DMS-2001139.

\section{Preliminaries}

\subsection{Symplectic resolutions}\label{SS_resolution_generalities}
\subsubsection{Definition}
Let $Y$ be a normal Poisson affine variety over $\C$.
We assume that the algebra $\C[Y]$ is positively graded, i.e., we have an algebra grading
$\C[Y]=\bigoplus_{i\geqslant 0}\C[Y]_i$ with $\C[Y]_0=\C$.
Further, we assume  there is a positive integer $d$
such that $\deg \{\cdot,\cdot\}=-d$, meaning that $\{\C[Y]_i,\C[Y]_j\}\subset \C[Y]_{i+j-d}$ for all $i,j$.

\begin{defi}\label{defi:sympl_resol}
By a {\it symplectic resolution
of singularities} of $Y$ one means a pair $(X,\pi)$ of
\begin{itemize}
\item a smooth symplectic algebraic variety $X$
\item and a morphism $\pi:X\rightarrow Y$ of Poisson varieties that is a projective resolution of singularities.
\end{itemize}
\end{defi}

The $\C^\times$-action lifts from $Y$ to $X$ making $\rho$ equivariant, see
Step 1 of the proof of \cite[Proposition A.7]{Namikawa_flop}. Clearly, such a lift is unique.
This action of $\C^\times$ on $X$ will be called {\it contracting}.

An example, already mentioned in the introduction is the Springer resolution: $Y$
is the nilpotent cone in a semisimple Lie algebra $\g$, and $X=T^*\mathcal{B}$,
where $\mathcal{B}$ is the flag variety for $\g$. The morphism $\pi$ is the
Springer morphism, i.e., the moment map for the action of $G$ on $X$. We consider the
fiberwise action of  $\C^\times$ on $T^*\mathcal{B}$. Here $d=1$.

\subsubsection{Structural results}
\begin{Lem}\label{Lem:cohom_vanishing}
We have the following results
\begin{enumerate}
\item   $\pi^*:\C[Y]\rightarrow \C[X]$  is an isomorphism,
\item  $H^i(X,\Str_X)=0$ for $i>0$,
\item  $H^i(X,\C)=0$ for $i$ odd.
\end{enumerate}
\end{Lem}

(1) holds because $\pi$ is birational and $Y$ is normal. (2) is a consequence of
the Grauert-Riemenschneider theorem. (3) is \cite[Proposition 2.5]{BPW}. The following
is a classical consequence of (3).

\begin{Cor}\label{Lem:Pic_Hom}
The Chern character map $c_1:\operatorname{Pic}(X)\rightarrow H^2(X,\Z)$ induces an isomorphism
$\mathbb{Q}\otimes_{\Z}\operatorname{Pic}(X)\xrightarrow{\sim} H^2(X,\mathbb{Q})$.
\end{Cor}

\begin{defi}\label{defi:param}
Below we will write $\param$ for $H^2(X,\C)$,  $\param_{\Q}$ for $H^2(X,\mathbb{Q})$,
and $\param_{\Z}$ for the image of $\operatorname{Pic}(X)$ in $\param_{\Q}$.
\end{defi}

The following is a consequence of \cite[Theorem 2.3]{Kaledin_symplectic}.

\begin{Lem}\label{Lem:fin_leaves}
The variety $Y$ has finitely many symplectic leaves.
\end{Lem}

\subsubsection{Deformations}\label{SSS_deformations}
We will be interested in deformations $X_{\param'}/\param'$ of $X$, where $\param'$ is a finite dimensional vector space,
and $X_{\param'}$ is a symplectic scheme over $\param'$ with a symplectic form $\hat{\omega}\in
\Omega^2(X_{\param'}/\param')$  and also with a $\C^\times$-action
on $X_{\param'}$ having the following properties:
\begin{itemize}
\item $\hat{\omega}|_X=\omega$,
\item the morphism $X_{\param'}\rightarrow \param'$ is $\C^\times$-equivariant,
\item the restriction of the $\C^\times$-action to $X$ coincides with the contracting action,
\item $t.\hat{\omega}:=t^{d}\hat{\omega}$.
\end{itemize}
It turns out that there is a {\it universal} such deformation $X_{\param}$ over $\param(=H^2(X,\C))$
(``universal'' means that any other deformation is obtained via the pull-back with respect to
a unique linear map $\param'\rightarrow \param$).  This result for formal deformations is
due to Kaledin-Verbitsky, \cite{KV}, but then it carries over to the algebraic setting
thanks to the contracting $\C^\times$-action on $X$, see \cite[Lemmas 12,22, Proposition 13]{Namikawa_flop}.

Let $Y_{\param}$ stand for $\operatorname{Spec}(\C[X_{\param}])$. Then the natural morphism
$\tilde{\pi}:X_{\param}\rightarrow Y_{\param}$ is projective and birational. The scheme $Y_{\param}$
is a deformation of $Y$ over $\param$ meaning that $Y_\param$ is flat over $\param$ and $\{0\}\times_{\param}Y_{\param}$ is identified with $Y$.

For $\zeta\in \param$, let $X_\zeta,Y_\zeta$
denote the fibers of $X_{\param},Y_{\param}$ over $\zeta$. Let $\param^{sing}$ denote the
locus of $\zeta\in \param$ such that the following equivalent conditions hold
\begin{itemize}
\item
$\pi_\lambda:X_\zeta\rightarrow Y_\zeta$
is not an isomorphism
\item  equivalently, $Y_\zeta$ is not smooth,
\item equivalently,
$X_\zeta$ is not affine.
\end{itemize}
Then, according to the main theorem in \cite{Namikawa_birational},
$\param^{sing}$ is the union of
rational codimension $1$ vector subspaces in $\param$.

\begin{defi}\label{defi:walls} These hyperplanes (as well as their real loci)
will be called {\it classical walls}. An element  $\theta\in \param_{\R}$ lying away from
the classical walls will be called {\it generic}.
\end{defi}

\begin{Ex}
Let $X=T^*\mathcal{B}$. Let $\h\subset\bor\subset \g$
denote the Cartan and Borel subalgebras inside $\g$. Let $\mathfrak{n}$ denote the
maximal $\operatorname{ad}$-nilpotent subalgebra of $\bor$ so that $\bor=\h\ltimes \mathfrak{n}$.
Let $W$ denote the Weyl group.

The space $\param$ is naturally identified with $\h^*$.  We can identify $X$ with $G\times^B \bor^\perp$, where we write $\bullet^\perp$ for the annihilator in the dual space. Then $X_\param$ is identified
with $G\times^B \mathfrak{n}^{\perp}$ with $X_\param\rightarrow \param$
is induced from the projection $\mathfrak{n}^\perp\rightarrow \h^*$.
Further, $Y_{\param}\cong \g^*\times_{\h^*/W}\h^*$ and
$\tilde{\pi}:X_{\param}\rightarrow Y_{\param}$ is Grothendieck's simultaneous resolution.
In particular, the classical walls are precisely the root hyperplanes.
\end{Ex}

\subsubsection{Classification of resolutions}\label{SSS_classif_resol}
We proceed to classifying the  possible symplectic resolutions of $Y$ following
\cite{Namikawa_birational}.

Suppose $X^1,X^2$ are two  symplectic resolutions of $Y$. There are
open subvarieties $\breve{X}^i\subset X^i, i=1,2,$ with $\operatorname{codim}_X^i X^i\setminus \breve{X}^i
\geqslant 2$ such that the rational map $X^1\dashrightarrow X^2$ restricts to an isomorphism $\breve{X}^1\xrightarrow{\sim} \breve{X}^2$,
see, e.g., \cite[Proposition 2.19]{BPW}. This allows to identify the Picard groups $\operatorname{Pic}(X^1),\operatorname{Pic}(X^2)$. Denote the resulting abelian group by $\Lambda$.

Let $C_{mov}$  denote the movable cone of $X$. It does not depend on the choice of a resolution by the previous paragraph. According to the main theorem of \cite{Namikawa_birational},
there is a finite group $W_X$ acting on $\param_{\mathbb{R}}$ as a reflection group, such that $C_{mov}$ is a fundamental chamber for $W_X$. Further, the set of classical walls (Definition
\ref{defi:walls}) is $W_X$-stable  and the walls for $C_{mov}$ are among the classical walls. So the
classical walls
further split $C_{mov}$ into chambers and the set of (isomorphism classes of) conical symplectic resolutions of $Y$ is in one-to-one
correspondence with the set of chambers inside $C_{mov}$. Under this correspondence,
a symplectic resolution goes to the closure of its ample cone.

Now we want to define a $W_X$-action on $\Lambda$ so that the 1st Chern class  map
$c_1:\Lambda\rightarrow \param$ is $W_X$-equivariant.
We note that $W_X$ acts on $Y_{\param}$ by $\C^\times$-equivariant Poisson automorphisms making the morphism $Y_{\param}\rightarrow \param$
$W_X$-equivariant. The quotient $Y_\param/W_X$ is a universal $\C^\times$-equivariant
deformation of $Y$, \cite{Namikawa2}. In particular, it is independent of $X$. It follows
that $Y_\param$ is independent of the choice of $X$ as well.

Note that
the locus in $X_{\param}$, where $X_{\param}\rightarrow Y_{\param}$ is not an isomorphism
has codimension $\geqslant 2$. Let $Y_\param^0$ denote the complement to this locus. This locus is easily seen to coincide with the union of the smooth loci in $Y_\zeta$ for $\zeta\in \param$. In
particular, $Y_\param^0\subset Y_{\param}$ is independent of the choice of $X$.

Thanks to $\operatorname{codim}_{X_{\param}}(X_\param \setminus Y_\param^0)\geqslant 2$, we can  identify  $\Pic(X_{\param})$
with $\Pic(Y_\param^0)$. On the other hand, every line bundle on $X$ uniquely deforms
to $X_{\param}$ thanks to $H^i(X,O_X)=0$ for $i>0$.
This gives an identification $\Pic(X)\cong \Pic(X_{\param})$. So for two different
resolutions $X^1,X^2$ we get identifications
\begin{equation}\label{eq:Pic_ident}
\Pic(X^1)\xrightarrow{\sim} \Pic(X^1_\param)\xrightarrow{\sim} \Pic(Y^0_\param)
\xrightarrow{\sim} \Pic(X^2_\param)\xrightarrow{\sim} \Pic(X^2).
\end{equation}
Let $\breve{X}^1_{\param},\breve{X}^2_\param$ denote the loci in $X_\param^1,X_\param^2$
where the rational map $X_\param^1\dashrightarrow X_\param^2$ is an isomorphism.
Note that $Y^0_\param\subset \breve{X}^i_\param$ and $\breve{X}^i=\breve{X}^i_\param\cap X^i$.
It follows that (\ref{eq:Pic_ident}) coincides with the identification constructed in
the beginning of the section.

For a regular element $\theta\in C_{mov}$ we define $X^\theta$ to be the symplectic resolution corresponding to the chamber of $\theta$.
For $w\in W_X$, we set $X^{w\theta}:=X^\theta$. But we twist the identification of $\param$ with
$H^2(X,\C)$ by $w$ so that the ample cone for $X$ in $\param$ now contains $w\theta$.
So $X^\theta$ now makes sense for all regular $\theta\in \param_{\mathbb{R}}$.
Note that $W_X$ acts on $Y_\param^0$, which gives a $W_X$-action on $\Pic(X^\theta)$. We identify
$\Pic(X^{\theta})$ with $\Pic(X^{w\theta})$ by  $w$. In particular, for $\chi\in \Lambda$
it makes sense to speak about the line bundle $\Str(\chi)$ on  $X^\theta$ for every
regular $\theta$.

Now we explain what this means for $X=T^*\mathcal{B}$. In this case, the movable cone
coincides with the ample cone of $X$ and is the positive Weyl chamber. So all $X^\theta$
are just $T^*\mathcal{B}$ with different identifications between $H^2(X,\C)$ and $\h^*$.

\subsubsection{Partial resolutions}\label{SSS_partial_resolutions}
Let $X=X^\theta$ be a symplectic resolution of $Y$ and let $C$ denote the closure of the ample
cone for $X$. Pick a face $\underline{C}\subset C$. By standard results in the MMP, the choice of
$\underline{C}$ gives rise to a factorization
$$X\xrightarrow{\pi_1} X^{\underline{C}}\xrightarrow{\pi_2}Y, $$
where
\begin{itemize}
\item $X^{\underline{C}}$ is normal,
\item $\pi_1,\pi_2$ are projective and birational,
\item and for a rational point $\chi$ in the relative interior of $\underline{C}$,
the point $\chi$ is a positive multiple of $c_1(\pi_1^*\mathcal{L})$ for an
ample line bundle $\mathcal{L}$ on $X^{\underline{C}}$. Conversely, for any
ample line bundle $\mathcal{L}$ on $X^{\underline{C}}$, the element
$c_1(\pi_1^*\mathcal{L})$ lies in the relative interior of $ \underline{C}$.
\end{itemize}

See, for example, \cite[Section 3-2]{KMM}. Note that in \cite{KMM} the result is stated
in terms of the nef cone. By a theorem of Kleiman, in our situation, the nef cone
is dual to the ample cone, and so we can talk about the ample cone instead.

Again, we explain what we get in the case of $X=T^*\mathcal{B}$. The choice of
$\underline{C}$ gives rise to that of a standard parabolic subgroup $P\supset B$. Let
$L\subset P$ be the standard Levi subgroup and $U$ be the unipotent radical of
$P$. The choice of $P$ is uniquely characterized by the condition that $\underline{C}$
coincides with the subset of strictly dominant elements in
$\mathfrak{X}(L)\otimes_{\mathbb{Z}}\mathbb{R}$. Let $Y_L$ denote the nilpotent cone for $L$. Then $X^{\underline{C}}=
G\times^P (Y_L\times \mathfrak{p}^\perp)$.

\subsubsection{Further examples}
\begin{Ex}\label{Ex:partial_flag}
Let $P$ be a parabolic subgroup in  a semisimple algebraic group $G$. Fix a Levi decomposition
$P=L\ltimes U$. Consider the variety $X=T^*(G/P)$. The group $G$ acts on $X$ with an open
orbit, to be denoted by $\tilde{\Orb}$. The orbit $\tilde{\Orb}$ is an equivariant cover of
a nilpotent orbit in $\g^*$. Then $X$ is a symplectic resolution of $Y:=\operatorname{Spec}(\C[\tilde{
\Orb}])$. One can show that $\tilde{\Orb}$ only depends on $L$, so, in general, $Y$ has several
non-isomorphic symplectic resolutions. We have $\param=(\lf/[\lf,\lf])^*$ and $X_{\param}=
G\times^P ([\lf,\lf]+\mathfrak{u})^\perp$. The classical walls are the hyperplanes in
$(\lf/[\lf,\lf])^*$ defined by coroot hyperplanes $\alpha^\vee=0$, where $\alpha$ is
a root of $\g$ but not of $\lf$.
\end{Ex}

\begin{Ex}\label{Ex:Coulomb}
In this paper we will also need another class of symplectic resolutions that is more recent and much less
understood: smooth Coulomb branches of gauge theories. We start with a connected reductive group
$G$ and its representation $V$. Starting from $(G,V)$ Braverman, Finkelberg and Nakajima, \cite[Section 6]{BFN1},  produce a normal affine variety $Y$ defined as the spectrum of a convolution algebra of a suitable infinite dimensional analog of the Steinberg variety.

Let $T$ denote a maximal torus in $N_{\operatorname{GL}(V)}(G)/G$.
The variety $Y$ admits a deformation $Y_{\mathfrak{t}}$ over $\mathfrak{t}$, where
$\mathfrak{t}=\operatorname{Lie}(T)$. Moreover, for a one-parameter subgroup
$\nu:\C^\times\rightarrow T$ there is a partial resolution $X^\nu\rightarrow Y$,
\cite[Introduction]{BFN2}. In some cases it is known that $X^\nu$ is smooth. These cases include
the following:
\begin{itemize}
\item the case when $T$ is a torus. Then $X^\nu$ is a smooth hypertoric variety.
Moreover, every smooth hypertoric variety arises in this way.
\item some cases when $V$ is a framed representation of a quiver and $G$
is the corresponding group. We will review these cases below,
Section \ref{SSS_quiver_examples}.
\end{itemize}
\end{Ex}

Finally, we will need another class of symplectic resolutions: Nakajima quiver varieties.
This is the most important class of symplectic resolutions for the purposes of this paper
and we will consider it in Section \ref{SS_Nakajima}.

\subsubsection{Conical neighborhoods and slices}\label{SSS_conical_slices}
Take $y\in Y_{\param}$. Let $\zeta$ denote the image of $y$ in $\param$. Consider the
completion $\C[Y_\param]^{\wedge_\zeta}$. This is a Poisson algebra over the completion $\C[\param]^{\wedge_y}$. Further, set $Y_{\param}^{\wedge_y}:=\operatorname{Spec}(\C[Y_{\param}^{\wedge_y}])$
and $X_\param^{\wedge_y}:=Y_{\param}^{\wedge_y}\times_{Y_\param}X_\param$.
Note that $\C[Y_\param^{\wedge_y}]$ is naturally identified with $\C[X_\param^{\wedge_y}]$.

The space $\param^*$ embeds into $\C[\param]^{\wedge_\zeta}$ via $\alpha\mapsto \alpha-\langle\zeta,\alpha\rangle$.  We have a unique $\C^\times$-action on $\C[\param]^{\wedge_\zeta}$ by topological algebra automorphisms characterized by the property
that $\param^*$ has degree $d$.

Below we always impose the following assumption that holds in all examples we know.

\begin{itemize}
\item[($\diamondsuit$)] There is a $\C^\times$-action on $X_\param^{\wedge_y}$ that
\begin{itemize}
\item makes the morphism $X_\param^{\wedge_y}\rightarrow \param^{\wedge_\zeta}$
$\C^\times$-equivariant,
\item rescales the fiberwise symplectic form on $X_\param^{\wedge_y}$ by $t\mapsto t^d$,
\item induces a contracting action on $Y_\param^{\wedge_y}$.
\end{itemize}
\end{itemize}

We would like to deduce some corollaries. We can find an embedding
$\C[[T_y\mathcal{L}]]\rightarrow \C[Y_{\param}]^{\wedge_y}$. Such embeddings are conjugate under
Hamiltonian automorphisms  of $\C[Y_\param]^{\wedge_y}$.
We can decompose $\C[Y_{\param}]^{\wedge_y}$
into the completed tensor product of complete local Poisson algebras $\C[[T_y\mathcal{L}]]
\widehat{\otimes} \underline{A}'_{\param}$, where $\mathcal{L}$ is the symplectic leaf
through $y$ in $Y_\zeta$ and $\underline{A}'_\param$ is the Poisson centralizer of
$\C[[T_y\mathcal{L}]]$ under an embedding $\C[[T_y\mathcal{L}]]
\hookrightarrow \C[Y_\param]^{\wedge_y}$.

Under assumption ($\diamondsuit$), we can choose the decomposition
$$\C[Y_\param]^{\wedge_y}\cong \C[[T_y\mathcal{L}]]
\widehat{\otimes} \underline{A}'_{\param}$$
to be $\C^\times$-stable. Define $\underline{A}_\param$ to be
the $\C^\times$-finite part of $\underline{A}'_\param$. Then
$\underline{A}_\param$ is a finitely generated graded Poisson algebra.
We set $\underline{Y}_\param:=\operatorname{Spec}(\underline{A}_\param)$
so that $Y_{\param}^{\wedge_y}$ is identified with
$(\underline{Y}_\param\times T_y \mathcal{L})^{\wedge_0}$.

The variety $\underline{Y}$ admits a conical symplectic resolution $\underline{X}$
with a deformation $\underline{X}_\param$ over $\param$, see, e.g., \cite[Section 2.1.6]{catO_charp}.
By the construction, we have
$X_\param^{\wedge_y}\cong (\underline{X}_\param\times T_y\mathcal{L})^{\wedge_0}$.

\begin{Ex}\label{Ex:Springer_conical_slice}
Let $X=T^*\mathcal{B}$ so that $X_{\param}=G\times^B \mathfrak{n}^\perp$.
We can rescale the $\C^\times$-action to assume that $d=2$. We claim that
($\diamondsuit$) holds.

Let $L$ be the stabilizer of $\lambda\in \mathfrak{h}^*\hookrightarrow \mathfrak{g}^*$
in $G$ and $\mathcal{N}_L$ be the nilpotent cone in $\mathfrak{l}^*$. Then
$Y_\zeta=G\times^L \mathcal{N}_L$. So the symplectic leaves in $Y_\zeta$
are indexed by the nilpotent orbits of $L$. Let $e\in \lf^*$ be a nilpotent element.
We can consider the Slodowy slice $S_L$ to $e$ in $\mathfrak{l}^*$, this is a transverse
slice to $Le$ in $\lf^*$. This slice comes with a contracting $\C^\times$-action
with $d=2$. The point $e$ defines $y\in Y_\zeta$. We have $T_y\mathcal{L}=\lf^\perp\oplus T_e(Le)$.

Recall that $Y_\param=\h^*\times_{\h^*/W}\g^*$. For $\underline{Y}_{\param}$
we can take $\h^*\times_{\h^*/W_L}S_L$. Then $\underline{X}_\param$
is the preimage of $S_L$ in $L\times^{B\cap L}(\lf^*\cap\mathfrak{n}^\perp)$.
The varieties $\underline{Y}_{\param},\underline{X}_\param$ come with a contracting
$\C^\times$-action with  $d=2$. It is easy to see that $(T_y\mathcal{L}\times \underline{Y}_\param)^{\wedge_0}\cong Y_\param^{\wedge_y}$.

This shows, in particular, that ($\diamondsuit$) holds for $T^*\mathcal{B}$.
A similar construction shows that ($\diamondsuit$) holds for $T^*(G/P)$.
\end{Ex}

Now we return to the general situation. Set $\underline{\param}:=H^2(\underline{X},\C)$.
Note that $\underline{\param}=H^2(\underline{\pi}^{-1}(0),\C)$ and
$\underline{\pi}^{-1}(0)\hookrightarrow X_\zeta$. The spaces $H^2(X_\param,\C)$ and
$H^2(X,\C)=\param$ are naturally identified. This gives rise to the pullback map
$\param=H^2(X_{\param},\C)\rightarrow H^2(\underline{\pi}^{-1}(0),\C)=\underline{\param}$
to be denoted by $\eta$. We have a similar map between the Picard groups:
for this one needs to notice that $\operatorname{Pic}(\underline{X})=
\operatorname{Pic}(\underline{X}^{\wedge_0})$ thanks to the contracting $\C^\times$-action.
The maps $\eta:\operatorname{Pic}(X)\rightarrow \operatorname{Pic}(\underline{X})$
and $\eta: H^2(X,\C)\rightarrow H^2(\underline{X},\C)$ are intertwined by the
maps $c_1$.

By the construction, we have $\underline{X}_{\param}=\param\times_{\underline{\param}} \underline{X}_{\underline{\param}}$.

Now we consider the case when $\zeta$ is generic (in the sense explained below)
in a classical wall. The following is \cite[Proposition 2.3]{catO_charp}.

\begin{Prop}\label{Prop:rank1}
Suppose that $\zeta\in \param$ is such that
the only rational hyperplane in $\param$ that contains $\lambda$ is a classical
wall, say $\Upsilon$. Let $y\in Y_\zeta$ and form the symplectic resolution  $\underline{X}$
as above in this section. Then the kernel of $\eta: \param\rightarrow
\underline{\param}$ coincides with $\Upsilon$.
\end{Prop}

\begin{Ex}\label{Ex:rank1_cotangent}
Let $X=T^*\mathcal{B}$. For $\lambda$ as above, we have $\underline{X}=T^*\mathbb{P}^1$
hence $\underline{\param}$ is 1-dimensional. More generally for $X=T^*(G/P)$, all possible
$\underline{X}$ in the previous proposition are going to be Slodowy subvarieties of
 $T^*(G_1/P_1)$, where $G_1$ is a Levi subgroup
in $G$ and $P_1=P\cap G_1$ is a maximal parabolic subgroup of $G_1$.
\end{Ex}

\subsection{Quantizations}
\subsubsection{Definitions}
We will study filtered quantizations of $Y,Y_{\param}, X,X_{\param}$.
By a quantization of $Y$, we mean a pair $(\A,\iota)$ of
\begin{itemize}
\item a filtered
algebra $\A=\bigcup_i \A_{\leqslant i}$ such that $\deg [\cdot,\cdot]\leqslant -d$ meaning that
$[\A_{\leqslant i},\A_{\leqslant j}]\subset
\A_{\leqslant i+j-d}$ for all $i,j$,
\item and an isomorphism $\iota:\gr\A\xrightarrow{\sim} \C[Y]$ of graded Poisson algebras.
\end{itemize}
Similarly, a quantization of  $Y_{\param}$ is a pair $(\tilde{\A},\iota)$, where
$\tilde{\A}$ is a filtered $\C[\param]$-algebra (with $\param^*$ in degree $d$
and $\deg [\cdot,\cdot]\leqslant -d$) together with an isomorphism
$\iota: \gr\tilde{\A}\xrightarrow{\sim} \C[Y_{\param}]$ of graded Poisson $\C[\param]$-algebras. For $\lambda\in \param$,
we set $\A_\lambda:=\C_\lambda\otimes_{\C[\param]}\tilde{\A}$, this is a filtered quantization of
$\C[Y]$.

By a quantization of $X=X^\theta$, we mean a pair $(\A^\theta,\iota)$ of
\begin{itemize}
\item  a sheaf $\A^\theta$ of filtered algebras in the conical
topology on $X$  that is complete
and separated with respect to the filtration and satisfies $\deg [\cdot,\cdot]\leqslant -d$
\item together with an isomorphism $\iota: \gr\A^\theta\xrightarrow{\sim} \Str_{X^\theta}$ (of sheaves of graded Poisson algebras).
\end{itemize}

Similarly, we can talk about quantizations of $X_{\param}^\theta$.

\subsubsection{Classification of quantizations}
\cite[Theorem 1.8]{BK} (with ramifications given in \cite[Section 2.3]{quant_iso}) shows that the quantizations
$\A^\theta$ of $X$ are parameterized (up to an isomorphism) by the points in $\param=H^2(X,\C)$.
More precisely, there is a {\it canonical} quantization $\A_{\paramq}^\theta$ of $X_{\param}^\theta$ such that the quantization of $X^\theta$ corresponding to $\lambda\in \paramq$ is the specialization of $\A_{\paramq}^\theta$ to $\lambda$. We will write $\A_\lambda^\theta$ for that specialization.

Set $\A_{\paramq}:=\Gamma(\A_{\paramq}^\theta)$. This is a quantization of $Y_{\param}$ because
$R\Gamma(\Str_X)=\C[Y]$. As our notation suggests, the algebra $\A_{\paramq}$ is independent of
$\theta$, this follows from  \cite[Proposition 3.8]{BPW}. We set $\A_\lambda=\C_\lambda\otimes_{\C[\paramq]}\A_{\paramq}$.

It was observed in the proof of  \cite[Proposition 3.10]{BPW} that the Namikawa-Weyl group $W_X$ acts on $\A_{\paramq}$ by filtered algebra
automorphisms lifting the action on $\paramq$. In particular, $\A_\lambda\cong \A_{w\lambda}$
for all $\lambda\in \paramq$ and all $w\in W_X$.

In fact, every quantization of $\C[Y]$ has the form $\A_\lambda$ for some $\lambda\in \paramq$,
\cite[Theorem 3.4]{orbit}.

\begin{Ex}\label{Ex:Springer}
Let $X=T^*\mathcal{B}$, recall that $\paramq=\h^*$. Then $\A_\lambda^\theta$ is the microlocalization of
the sheaf $\mathcal{D}^\lambda_{\mathcal{B}}$ of $\lambda-\rho$-twisted differential
operators. Similarly, $\A_\paramq^\theta$ is the microlocalization of $(\upsilon_*D_{G/U})^T$,
where $\upsilon$ is the projection $G/U\rightarrow G/B=\mathcal{B}$. The algebra $\A_\lambda$
is the central reduction $\U_\lambda$ of $U(\g)$. And $\A_{\paramq}$ is
$\C[\h^*]\otimes_{\C[\h^*]^W}U(\g)$.
\end{Ex}

\subsubsection{Slice quantizations}\label{SSS_slice_quant}
We use the notation and conventions of Section \ref{SSS_conical_slices}.
Consider the algebra $\underline{\A}_{\underline{\paramq}}$, the global sections
of the canonical quantization $\underline{X}_{\underline{\param}}$, and set
$\underline{\A}_\paramq:=\C[\paramq]\otimes_{\C[\underline{\paramq}]}\underline{\A}_{\underline{\paramq}}$.
Further let $\mathbb{A}$ denote the Weyl algebra of the symplectic vector space
$T_y\mathcal{L}$. Form the filtered algebra $\mathbb{A}\otimes \underline{\A}_{\paramq}$,
take its Rees algebra $R_\hbar(\mathbb{A}\otimes \underline{\A}_{\paramq})$
and complete it at zero. Denote this completion by $R_\hbar(\mathbb{A}\otimes \underline{\A}_{\paramq})^{\wedge_0}$.

On the other hand, we can take the Rees algebra $R_\hbar(\A_{\paramq})$ and complete it
at $y$ getting the algebra $R_\hbar(\A_{\paramq})^{\wedge_y}$. It was shown in
\cite[Remark 3.4]{perv} that we have a $\C[[\paramq,\hbar]]$-linear isomorphism
$$R_\hbar(\A_{\paramq})^{\wedge_y}\cong R_\hbar(\mathbb{A}\otimes \underline{\A}_{\paramq})^{\wedge_0}$$
that lifts the isomorphism $\C[Y_\param]^{\wedge_y}\cong \C[T_y \mathcal{L}\times \underline{Y}_\param]^{\wedge_0}$.

\begin{Ex}\label{Ex:W-alg}
Let $Y$ be the nilpotent cone in $\g^*$. For simplicity, take $y\in Y$.
From $y$ we can form the finite W-algebra $\mathcal{W}$ whose center is still
$\C[\h^*]^W$. Then, by a construction of the W-algebra,
$\underline{\A}_{\paramq}=\C[\h^*]\otimes_{\C[\h^*]^W}\mathcal{W}$.
\end{Ex}

\subsection{Special case: Nakajima quiver varieties}\label{SS_Nakajima}
\subsubsection{Definition of quiver varieties}
Let $Q$ be a quiver (=oriented graph, we allow loops and multiple edges). We can formally represent $Q$ as a quadruple $(Q_0,Q_1,t,h)$, where $Q_0$ is a finite set of vertices, $Q_1$ is a finite set of arrows,
$t,h:Q_1\rightarrow Q_0$ are maps that to an arrow $a$ assign its tail and head.

Pick vectors $v,w\in \Z_{\geqslant 0}^{Q_0}$ and vector spaces $V_i,W_i$ with
$\dim V_i=v_i, \dim W_i=w_i$. Consider the (co)framed representation space
$$R=R(Q,v,w):=\bigoplus_{a\in Q_1}\Hom(V_{t(a)},V_{h(a)})\oplus \bigoplus_{i\in Q_0} \Hom(V_i,W_i).$$
We will also consider the cotangent bundle $T^* R=R\oplus R^*$ that can be identified with
$$\bigoplus_{a\in Q_1}\left(\Hom(V_{t(a)},V_{h(a)})\oplus \Hom(V_{h(a)}, V_{t(a)})\right)\oplus \bigoplus_{i\in Q_0} \left(\Hom(V_i,W_i)\oplus \Hom(W_i,V_i)\right).$$
The space $T^*R$ carries  a natural symplectic form, denote it by $\omega$.
On $R$ we have a natural action of the group $G:=\prod_{i\in Q_0} \GL(v_i)$. This action extends to
an action on $T^*R$ by linear symplectomorphisms.
As any action by linear symplectomorphisms, the $G$-action
on $T^*R$ admits a moment map, i.e., a $G$-equivariant morphism $\mu:T^*R\rightarrow \g^*$ with the
property that $\{\mu^*(x),\bullet\}=x_{T^*R}$ for any $x\in \g$. Here $\mu^*:\g\rightarrow \C[T^*R]$
denotes the dual map to $\mu$, $\{\bullet,\bullet\}$ is the Poisson bracket on $\C[T^*R]$ induced by $\omega$,
and $x_{T^*R}$ is the vector field on $T^*R$ induced by $x$ via the $G$-action. The moment map
$\mu$ is recovered from $\mu^*(x)=x_{R}$.
Also we consider the dilation action of the one-dimensional torus $\C^\times$ on $T^*R$ given by $t.r=t^{-1}r$.

The character lattice of $G$ is identified with $\Z^{Q_0}$ via $(\theta_i)_{i\in Q_0}\mapsto \theta$,
where $\theta((g_i))=\prod_{i\in Q_0} \det(g_i)^{\theta_i}$. Similarly, the space $\g^{*G}$
is identified with $\C^{Q_0}$. We set $\tilde{\param}:=\C^{Q_0}$.

We say that $\theta\in \Z^{Q_0}$ is generic if
$G$ acts on $\mu^{-1}(0)^{\theta-ss}$ freely.
According to
\cite{Nakajima94} $\theta$ is {\it generic} provided $\theta\cdot v'\neq 0$ for all roots
$v'$ of $Q$ such that $v'\leqslant v$. Note, however, that
$\mu^{-1}(0)^{\theta-ss}$ may be empty.

Define the following Hamiltonian reductions.

\begin{equation}\label{eq:quiver_def}
\M_0^0:=\mu^{-1}(0)\quo G, \M_{\tilde{\param}}^0:=\mu^{-1}(\tilde{\param})\quo G,
\M_0^\theta:=\mu^{-1}(0)\quo^\theta G, \M_{\tilde{\param}}^\theta:=\mu^{-1}(\tilde{\param})\quo^\theta G.
\end{equation}
Here $\quo$ stands for the categorical quotient and $\quo^\theta$ stands for the GIT quotient with
respect to $\theta$.
We will write $\M_0^0(v,w)$ instead of $\M_0^0$ when we want to indicate the dependence on $v,w$.
We remark that, by the construction, we have projective morphisms $\M_0^\theta\rightarrow \M_0^0,
\M^\theta_{\tilde{\param}}\rightarrow \M^0_{\tilde{\param}}$.

Note that since $\theta$ is generic, we have a linear map $\Z^{Q_0}\rightarrow \operatorname{Pic}(\M^\theta_0), \chi\mapsto \Str(\chi),$ by equivariant descent.
Note that $\Str(\theta)$ is ample
on $\M_0^\theta$.

We will need the following result that is a consequence, for example, of
\cite[Theorem 1.2]{MN_Kirwan}.

\begin{Lem}\label{Lem:2nd_surjectivity}
The induced map $\tilde{\param}\rightarrow H^2(\M_0^\theta,\C)$ is surjective.
\end{Lem}

Note that we have a contracting action of $\C^\times$ on $\M_0^\theta,\M_{\param}^\theta$
coming from the action of $\C^\times$ on $T^*R$ introduced above.

\subsubsection{Flatness}
In what follows we impose the following condition
\begin{itemize}
\item[($\sharp$)] The moment map $\mu: T^*R\rightarrow \g^*$ is flat.
\end{itemize}

Here is a consequence of ($\sharp$) obtained, for example, in \cite[Proposition 2.5]{BL}.

\begin{Lem}\label{Lem:resolution}
The morphism $\M_0^\theta\rightarrow \M_0^0$ is a resolution of singularities.
\end{Lem}

Let $\M^\theta_\param$ denote the universal deformation of $\M^\theta_0$.
Then we have $\M^\theta_{\tilde{\param}}=\tilde{\param}\times_{\param}\M^\theta_\param$,
this follows from \cite[Section 3.2]{quant_iso}.

Crawley-Boevey in \cite[Theorem 1.1]{CB} found a necessary and sufficient combinatorial condition
for ($\sharp$) to hold. For a vector $v\in \Z^{Q_0}$ we define $(v,v)$
as $\sum_{i\in Q_0}v_i^2-\sum_{a\in Q_1}v_{t(a)}v_{h(a)}$ (so that for a real
root we have $(v,v)=1$). According to \cite{CB},
($\sharp$) is equivalent to the following condition ($\sharp'$):

\begin{itemize}
\item[($\sharp'$)] For all decompositions $v=\sum_{i=0}^k v^i$, where all $v^i$
are in $\Z_{\geqslant 0}^{Q_0}$ and $v^1,\ldots,v^k$ are roots of $Q$, we have
\begin{equation}\label{eq:CB_condition} w\cdot v-(v,v)\geqslant w\cdot v^0-(v^0,v^0)+k-\sum_{i=1}^k (v^i,v^i).
\end{equation}
\end{itemize}

We will only need this in the  case when $Q$ is of finite
or affine type.
Assume $Q$ has no loops (the only affine quiver without this property is the Jordan quiver, which
needs to be analyzed separately, but this is easy).
Let $\alpha_i, i\in Q_0,$ denote the simple roots for $\g(Q)$ and
$\varpi_i, i\in Q_0,$ denote the fundamental weights. We record the framing
vector $w$ as the dominant weight $\omega:=\sum_{i\in Q_0} w_i\varpi_i$
and the dimension vector $v$ as a weight $\nu:=\omega-\sum_{i\in Q_0}v_i \alpha_i$.

The following results was obtained in \cite[Lemma 2.1]{BL}.

\begin{Lem}\label{Lem:dominant}
Suppose $Q$ is of finite or affine type. If $\nu$ is dominant, then
($\sharp'$), equivalently, ($\sharp$) holds.
\end{Lem}

The reasons why this is useful is because every $\M^\theta_0(v,w)$ is isomorphic
to $\M^\theta_0(v',w)$, where $\nu'$ is dominant. For general $Q$, $\M_0^\theta$
is still a symplectic resolution of singularities.

\subsubsection{Quantizations}
In this section we construct quantizations of $\M^0_0,\M^0_{\tilde{\param}},\M^\theta_0,
\M^\theta_{\tilde{\param}}$. Recall that $\theta$ is generic. We assume ($\sharp$)
holds.

Quantizations of $\M_0^0$ will be constructed via quantum Hamiltonian reduction.
This reduction will depend on a parameter $\lambda\in \tilde{\param}=\C^{Q_0}$.

Consider the algebra $D(R)$ of differential operators on $R$. We equip this algebra
with the Bernstein filtration. The group $G$ acts on
$D(R)$ in a Hamiltonian way with the symmetrized
quantum comoment map $\xi\mapsto \Phi(\xi):=\frac{1}{2}(\xi_R+\xi_{R^*})$.
Fix $\lambda\in \tilde{\paramq}\cong\g^{*G}$. Then we can define a filtered algebra
$$\A^0_\lambda:=[D(R)/D(R)\{\Phi(\xi)-\lambda(x)| x\in \g\}]^G.$$
Thanks to ($\sharp$) this algebra is a filtered quantization of $\M^0_0$.

Similarly, we can construct a quantization
$$\A^0_{\tilde{\paramq}}:=[D(R)/D(R)\{\Phi(\xi)| \xi\in [\g,\g]\}]$$
of $\M^0_{\tilde{\param}}$. Note that $\A^0_\lambda$ is the specialization
of $\A^0_{\tilde{\paramq}}$ to $\lambda\in \tilde{\paramq}$.

Now we proceed to quantizations of $\M^\theta_0, \M^\theta_{\tilde{\param}}$.
Let $D_R$ denote the microlocalization of $D(R)$ to $T^*R$. We write
$D_R^{\theta-ss}$ for the restriction of $D_R$ to $(T^*R)^{\theta-ss}$.
Let $\varrho: \mu^{-1}(0)^{\theta-ss}\rightarrow \M^\theta_0$ be the
quotient morphism. Then we set
$$\A^\theta_\lambda=\varrho_*[D_R^{\theta-ss}/D_R^{\theta-ss}\{\Phi(\xi)-\lambda(\xi)| \xi\in \g\}]^G.$$
This is a filtered quantization of $\M_0^\theta$. Since $\pi:\M_0^\theta\rightarrow \M_0^0$
is a resolution of singularities, we see that $\Gamma(\A^\theta_\lambda)=\A^0_\lambda$,
compare to \cite[Lemma 4.2.4]{quant_iso}.

Similarly we can define the quantization $\A^\theta_{\tilde{\paramq}}$ of $\M^\theta_{\tilde{\param}}$.
This quantization satisfies $\Gamma(\A^\theta_{\tilde{\paramq}})=\A^0_{\tilde{\paramq}}$.

It was shown in \cite[Section 5.4]{quant_iso} that  $\A^\theta_{\tilde{\paramq}}$ is obtained from
the canonical quantization $\A^\theta_\param$ via the pullback under the natural
map $\tilde{\param}\rightarrow \param$.

\subsubsection{Examples and special cases}\label{SSS_quiver_examples}
\begin{Ex}\label{Ex:Grassmanian}
Consider the case when we have quiver of type $A_1$: one vertex and no loops.
Assume $\theta>0$.
Here $\M^\theta_0=T^*\operatorname{Gr}(v,w)$ and $\A^\theta_\lambda$
is the microlocalization of the sheaf $D^\lambda_{\operatorname{Gr}(v,w)}$
of $(\lambda-w/2)$-twisted differential operators on $\operatorname{Gr}(v,w)$.
For $\theta<0$ we get $\M^\theta_0=T^*\operatorname{Gr}(w-v,w)$.
\end{Ex}

\begin{Ex}\label{Ex:Slodowy_A}
More generally, for type $A$ Dynkin quivers we recover parabolic Slodowy
slices of type $A$, see \cite{Maffei_Slodowy}.
\end{Ex}

\begin{Ex}\label{Ex:Gieseker}
Consider the case of Jordan quiver.  We write $n$ for $v$ and $r$ for $w$.
The variety $\M^\theta(n,r)$ is the Gieseker moduli space of rank $r$
degree $n$ torsion free sheaves on $\mathbb{P}^2$ framed at $\infty$.
We will also consider a closely related variety $\bar{\M}^\theta(n,r)$,
where in the definition we replace $R=\mathfrak{gl}_n\oplus \Hom(\C^n,\C^r)$
with $\bar{R}:=\mathfrak{sl}_n\oplus \Hom(\C^n,\C^r)$. Note that
$\M^\theta(n,r)=\C^2\times \bar{\M}^\theta(n,r)$.
\end{Ex}

\begin{Ex}\label{Ex:SRA}
Now consider the situation when $Q$ is an affine quiver, $v=n\delta$, where we write
$\delta$ for the indecomposable imaginary root, and $w=\epsilon_0$, the one-dimensional
framing at the extending vertex. Let $\Gamma_1$ be the finite subgroup of
$\operatorname{SL}_2(\C)$ of the same type as $Q$. Set $\Gamma_n=S_n\ltimes \Gamma_1^n$,
this group naturally acts on $(\C^2)^n$. Then $\M^0(n\delta,\epsilon_0)\cong (\C^2)^n/\Gamma_n$.
Moreover, the algebras $\A^0_{\lambda}$ are precisely the spherical reflection algebras
from \cite{EG}, see \cite{EGGO} or \cite{quant_iso}.
\end{Ex}

\begin{Rem}\label{Rem:Coulomb}
Some quiver varieties arise also as smooth Coulomb branches. Namely, if $Q$
is a finite or affine type $A$ quiver and $\nu$ is dominant, then
$\M_0^\theta$ is a smooth Coulomb branch  of a (finite/ affine type A)
quiver gauge theory, \cite[Sections 4,5]{BFN2}. Every finite and affine
type quiver variety $\M_0^\theta$ arises in this way.
\end{Rem}

\subsubsection{Symplectic leaves}\label{SSS_sympl_leaves}
We assume that ($\sharp$) holds. We would like to understand the symplectic leaves of
the specialization $\M_\zeta^0$ of $\M^0_{\tilde{\param}}$ to $\zeta$. Using this
we will obtain a (well-known) description of the classical walls. For simplicity,
below we assume that $Q$ is of finite or affine type.

It is a classical fact that follows,  for example, from \cite{CB}, that
the stratification by symplectic leaves coincides with the stratification
by representation types. Namely, one can view a point $y\in \M^0_\zeta$
as an isomorphism class of a semisimple representation of a suitable deformed preprojective algebra $\Pi^w_\zeta$. The quiver used to construct this algebra is as follows: we consider
the quiver $Q^w$ that is obtained from $Q$ by adding a new vertex $\infty$ with $w_i$
arrows from $i$ to $\infty$. The algebra $\Pi^w_\zeta$ is associated to this quiver
and the parameter $(-\sum_{i}\zeta_i v_i,\zeta)$, where the first coordinate corresponds
to the vertex $\infty$.

Denote the representation of $\Pi^w_\zeta$ corresponding to $y$ by $r$. We can decompose $r$ into the sum of irreducibles with multiplicities. The collection of the dimensions of irreducible
summands together with their multiplicities is what we mean by the representation
type of $y$.

In more detail we can write $r$ as $r^0\oplus r_1^{\oplus n_1}\oplus\ldots \oplus r_k^{\oplus n_k}$,
where $r^0$ is an irreducible representation of $\Pi^w_\zeta$ with dimension
$(1,v^0)$ (the first component corresponds to the new vertex $\infty$ which encodes
the framing) and $r^1,\ldots,r^k$ are pairwise non-isomorphic irreducible representations with dimensions
$v^1,\ldots,v^k$. Formally, the representation type of $y$ is defined to be $(v^0; v^1,n_1;\ldots; v^k,n_k)$.
We need to understand the possible values of $v^0$ and of $v^1,\ldots,v^k$.

We start  with $v^1,\ldots,v^k$. The following lemma follows easily from results of
\cite[Theorem 1.2]{CB}. Let $\Delta^+(\zeta)$ denote the set of positive roots for $\g(Q)$ that
vanish on $\zeta$.

\begin{Lem}\label{Lem:dim_description}
Recall that we assume that $Q$ is of finite or affine type.
The following are exactly the possible values of $v^1,\ldots,v^k$:
\begin{itemize}
\item the minimal roots in $\Delta^+(\zeta)$,
\item or, in the case of affine $Q$, the indecomposable imaginary root $\delta$,
assuming $\zeta\cdot \delta=0$.
\end{itemize}
\end{Lem}

Now we proceed  to $v^0$.

The following lemma is a straightforward consequence of \cite[Theorem 1.2]{CB}.

\begin{Lem}\label{Lem:simple}
We continue to assume that $Q$ is of finite or affine type. Then the following
two conditions are equivalent:
\begin{itemize}
\item $(1,v^0)$ is the dimension of an irreducible representation,
\item $(1,v^0)$ is a root for $Q^w$. Moreover, for any decomposition $v^0=\sum_{j=0}^\ell \underline{v}^j$ with $\ell>0$, where all $\underline{v}^j$
are nonzero elements of $\Z^{Q_0}_{\geqslant 0}$, and the elements
$\underline{v}^1,\ldots, \underline{v}^\ell$ satisfy $\zeta\cdot \underline{v}^j=0$
and  are as in Lemma \ref{Lem:dim_description},
we have the following inequality:
\begin{equation}\label{eq:inequality_irreducible}
w\cdot v^0-(v^0,v^0)> w\cdot \underline{v}^0-(\underline{v}^0,\underline{v}^0)+s,
\end{equation}
where $s$ is the number of  summands among $\underline{v}^1,\ldots,\underline{v}^\ell$
equal to $\delta$.
\end{itemize}
\end{Lem}

The following is a consequence of Lemma \ref{Lem:simple}, see
\cite[Lemma 2.1]{BL}.

\begin{Cor}\label{Cor:simple0}
Let $\zeta=0$. Then $(1,v^0)$ is the dimension of an irreducible representation
of $\bar{\Pi}^0$ if and only if
\begin{itemize}
\item for $Q$ of  finite type: $\nu^0$ is dominant, where $\nu^0$ is the weight formed from $v^0$;
\item for $Q$ of affine type: either $w\cdot \delta>1$ and $\nu^0$ is dominant, or $w\cdot\delta=1$
and $v^0=0$.
\end{itemize}
\end{Cor}

\subsubsection{Classical walls}\label{SSS_class_wall_quiver}
We will also use Lemmas \ref{Lem:dim_description} and \ref{Lem:simple} to recover the
classical walls. We still assume that $Q$ is of finite or affine type. We also assume that $\nu$
is dominant. Note that if $Q$ is of affine type and $(\omega,\delta)=1$, this
means that $v=n\delta$ for some $n$.

The condition that $\M^0_\zeta$ is singular is equivalent to the claim that
there is more than one representation type in $\M^0_\zeta$.

By \cite[Theorem 2.8]{Nakajima94}, each classical wall has the
form $\{\zeta| \zeta\cdot v'=0\}$ for some root  $v'$ of $Q$ with $v'\leqslant v$.
However, not every hyperplane of this form is a wall.

First we need to understand the case when $v'$ is a real root. Let $\beta$ denote the
root of $\g(Q)$ corresponding to $v'$.  Further, let $L^\omega$ denote the irreducible
representation of $\g(Q)$ with highest weight $\omega$. We take $\zeta$ to be
a generic element in the orthogonal complement of $v'$.

\begin{Lem}\label{Lem:leaves_generic_real}
Let $\beta$ be a real root.
Let $m\geqslant 0$ be maximal such that $\nu+m\beta$ is a weight of $L^\omega$.
Then there are $m+1$ symplectic leaves in $\M^0_\zeta$, they correspond to
the representation types $(v-iv';v',i)$ with $i=0,\ldots,m$. In particular,
$(v')^\perp$ is a classical wall if and only if $\nu+\beta$ is a weight of $L^\omega$.
\end{Lem}
\begin{proof}
Let $v^0=v-jv'$ hence $\nu^0=\nu+j\beta$. We have $(\nu+j\beta,\nu+j\beta)=(\nu,\nu)+2j(\nu,\beta)+2j^2$.
Since $\nu$ is dominant, the function $(\nu+j\beta,\nu+j\beta)$ is increasing in $j$. Inequality (\ref{eq:inequality_irreducible}) is therefore always satisfied. Because of this,
$v^0=v-jv'$ with maximal $j$ such that  $\M^0_\zeta(v^0,w)\neq \varnothing$ is a root.
If $\M_\zeta^\theta(v^0,w)\neq \varnothing$, then $\M^0_\zeta(v^0,w)\neq \varnothing$. The former is equivalent to the condition that
$\nu^0$ is a weight of $L_\omega$. Since $j$ is maximal such that
$\M^0_\zeta(v^0,w)\neq \varnothing$, we conclude that
$\M^0_\zeta(v^0,w)$ coincides with its regular locus in the sense
of \cite[Section 2]{Nakajima94}. It follows, in particular, that
$\M_\zeta^\theta(v^0,w)$ is nonempty.
\end{proof}

In particular, the lemma shows that there is a minimal (with respect to inclusion of closures)
symplectic leaf in $\M^0(v,w)$.

Now consider the case when $v'=\delta$, the indecomposable affine root. Assume
first that $w\cdot \delta=1$. Then $v=n\delta$. For a generic parameter $\zeta\in \ker\delta$,
the leaves in $\M^0_\zeta(n\delta,w)$ are parameterized by the partitions of $n$. In particular,
there is again a unique minimal leaf, it corresponds to the one part partition. And $\delta=0$ is a classical wall if and only if
$n>1$.

The following lemma is proved similarly to Lemma \ref{Lem:leaves_generic_real}.

\begin{Lem}\label{Lem:leaves_generic_imaginary}
Assume that $Q$ is affine and $w\cdot \delta>1$. Let $m$ be maximal such that $\nu+m\delta$
is a weight of $L^\omega$. The hyperplane $\ker\delta$ is a classical  wall if and only
if $m>0$. Moreover, there is a unique minimal symplectic leaf, it corresponds to
the representation type $(v-m\delta; \delta,m)$.
\end{Lem}


\subsubsection{Local structure}\label{SSS_slices_quiver}
Now we are going to describe the local behaviour of the variety $\M_{\tilde{\param}}^0$ and
the morphism $\M_{\tilde{\param}}^\theta\rightarrow \M_{\tilde{\param}}^0$ following \cite[Proposition 6.2]{Nakajima94}
and also \cite[Section 2.1.6]{BL}.

Pick $y\in \M^0_\zeta$ and let $(v^0;v^1,n_1;\ldots; v^k;n_k)$ be the corresponding
representation type. We define the new quiver $\underline{Q}$ and new dimension and framing
vectors $\underline{v}$ and $\underline{w}$ as follows. For the set of vertices of $\underline{Q}$ we take $\{1,\ldots,k\}$.  The number of arrows between $i$ and  $j$ in $\underline{Q}$ is
$\delta_{ij}-(v^i,v^j)$. We define $\underline{w}_i:=w\cdot v^i-(v^0,v^i)$ and $\underline{v}_i=n_i$
for $i=1,\ldots,k$.  Set $\underline{G}:=\prod_{i=1}^k \operatorname{GL}(n_i)$. Note that this group
is the stabilizer of a semisimple representation lying over $y$. This gives an embedding
$\underline{G}\rightarrow G$ and hence the restriction maps $\tilde{\eta}: \g^{*G}\rightarrow
\underline{\g}^{*\underline{G}},\mathfrak{X}(G)\rightarrow \mathfrak{X}(\underline{G})$.

From $\underline{Q},\underline{v},\underline{w}$ and the restriction of $\theta$ to $\underline{G}$ we get the quiver varieties
$\underline{\M}^0_{\underline{\tilde{\param}}}, \underline{\M}^\theta_{\underline{\tilde{\param}}}$.
Consider their base changes $\underline{\M}^0_{\tilde{\param}}:=\tilde{\param}\times_
{\underline{\tilde{\param}}}\underline{\M}^0_{\underline{\tilde{\param}}}$ and
$\underline{\M}^\theta_{\tilde{\param}}$.

For the symplectic vector space $V:=T_y\mathcal{L}$, where $\mathcal{L}$
stands for the symplectic leaf of $y$,  we get isomorphisms
$$(\M^0_{\tilde{\param}})^{\wedge_y}\cong (\underline{\M}^0_{\tilde{\param}}\times V)^{\wedge_0},
(\M^\theta_{\tilde{\param}})^{\wedge_y}\cong (\underline{\M}^\theta_{\tilde{\param}}\times V)^{\wedge_0}.$$
These are isomorphisms of schemes over $\tilde{\param}^{\wedge_0}$, where in the left hand sides we identify
$\tilde{\param}^{\wedge_\zeta}$ with $\tilde{\param}^{\wedge_0}$ by shift by $-\zeta$. Also these isomorphisms
intertwine the projective morphisms
$(\M^\theta_{\tilde{\param}})^{\wedge_y}\rightarrow
(\M^0_{\tilde{\param}})^{\wedge_y}$ and
$(\underline{\M}^\theta_{\tilde{\param}}\times V)^{\wedge_0}
\rightarrow (\underline{\M}^0_{\tilde{\param}}\times V)^{\wedge_0}$.

We can apply this analysis to finite or affine type $A$. We easily arrive at the following
result.

\begin{Lem}\label{Lem:type_A_slices}
If $Q$ is of finite or affine type A, then so is $\underline{Q}$.
\end{Lem}

\subsubsection{Rank 1 varieties}
We will describe the data $(\underline{Q},\underline{v},\underline{w})$ and also the maps
$\tilde{\eta}:\tilde{\param}\rightarrow \tilde{\underline{\param}}$ when $\zeta$
is generic in a classical wall and $y$ is a point in the  unique minimal symplectic leaf in
$\M^0_\zeta$.

The following lemmas are consequences of the discussions in Sections \ref{SSS_class_wall_quiver} and \ref{SSS_slices_quiver}. We assume that the weight $\nu$ is dominant.

\begin{Lem}\label{Lem:rk1_slice_real}
Suppose that the classical wall in question is defined by a real root $\beta$.
Then $\underline{Q}$ is a type $A_1$ quiver and $\tilde{\eta}(?)=\langle\beta,?\rangle$.
In the notation of Lemma \ref{Lem:leaves_generic_real}, we have $\underline{v}=m$
and $\underline{w}=(\nu,\beta)+2m$.
\end{Lem}

\begin{Lem}\label{Lem:rk1_slice_imaginary}
Suppose that the classical wall in question is defined by
the indecomposable imaginary root $\delta$. Then $\underline{Q}$
is the Jordan quiver and $\tilde{\eta}(?)=\langle\delta, ?\rangle$.
We have $\underline{v}=m$ and $\underline{w}=(\nu,\delta)$.
\end{Lem}

\subsection{Harish-Chandra bimodules}
\subsubsection{Definition and basic constructions}
Let $\chi\in \param$. The following definition generalizes the classical
HC bimodules over semisimple Lie algebras, see \cite{HC,BPW,BL}.

\begin{defi}\label{defi:HC}
An $\A_{\paramq}$-bimodule $\B$ is called a {\it Harish-Chandra $(\A_{\paramq},\chi)$-bimodule}
if the following hold:
\begin{itemize}
\item[(i)] For all $\alpha\in \paramq^*$ and $b\in \B$ we have $[\alpha,b]=\langle \alpha,\chi\rangle b$,
\item[(ii)] and there is a bimodule filtration (called {\it good}) $\B=\bigcup \B_{\leqslant j}$ such that
$[\A_{\paramq,\leqslant i},\B_{\leqslant j}]\subset \B_{\leqslant i+j-d}$ for all $i,j$
and $\gr\B$ is a finitely generated $\C[Y_\param]$-bimodule.
\end{itemize}
\end{defi}

For example, $\A_{\paramq}$ is a HC $(\A_\paramq,0)$-bimodule.

The category of HC $(\A_{\paramq},\chi)$-bimodules (with morphisms being bimodule homomorphisms)
will be denoted by $\HC(\A_{\paramq},\chi)$.

Let $\B$ be a HC $(\A_{\paramq},\chi)$-bimodule. If the right action of $\A_{\paramq}$  on $\B$
factors through $\A_\lambda$, then we say that $\B$ is a HC $\A_{\lambda+\chi}$-$\A_\lambda$-module.

Note also that $\operatorname{Tor}^{\A_{\param}}_i(\bullet,\bullet)$ gives a functor
$$\HC(\A_{\paramq},\chi')\times \HC(\A_{\paramq},\chi)\rightarrow \HC(\A_{\paramq},\chi+\chi'),$$
see \cite[Proposition 6.3]{BPW}. Similarly, the functors
$\operatorname{Ext}^{\A_{\param}}_i$ and $\operatorname{Ext}^{\A_{\param}^{opp}}_i$
send HC bimodules to HC bimodules.

For a HC $(\A_{\paramq},\chi)$-bimodule $\B$, we can define its associated variety
$\VA(\B)\subset Y_{\param}$, it is the closed subvariety of $Y_\param$ defined by
the annihilator of $\gr \B$.

For a HC bimodule $\B\in \HC(\A_{\paramq},\chi)$ we can define its {\it right $\param$-support}
$$\operatorname{Supp}^r_{\paramq}(\B):=\{\lambda\in \paramq| \B\otimes_{\C[\paramq]}\C_\lambda
\neq \{0\}\}.$$ The following result is \cite[Proposition 2.6]{catO_R}.

\begin{Lem}\label{Lem:support_closed}
$\operatorname{Supp}^r_{\paramq}(\B)$  is closed and its asymptotic cone
coincides with the support of $\gr\B$, where the associated graded bimodule is taken
with respect to any good filtration.
\end{Lem}

\subsubsection{Translation bimodules}
Translation bimodules introduced in the generality we need in \cite{BPW}
is the main family of HC bimodules we will need in the present paper.

Take $\chi\in \Lambda$, let
$\Str_\param(\chi)$ denote the corresponding line bundle on $X_{\param}$.
By \cite[Section 5.1]{BPW}, this bundle
admits a unique filtered quantization to a right $\A_{\paramq}^\theta$-module,
to be denoted by $\A_{\paramq,\chi}^\theta$. There is a commuting
left $\A_{\paramq}^\theta$-action on $\A_{\paramq,\chi}^\theta$
satisfying condition (i) of Definition \ref{defi:HC}.

\begin{defi}\label{defi:translation}
By the {\it translation bimodule} $\A_{\paramq,\chi}$ we mean $\Gamma(\A_{\paramq,\chi}^\theta)$.
\end{defi}

It was shown in \cite[Propositions 6.23,6.24]{BPW}
that $\A_{\paramq,\chi}$ is independent of the choice of
$\theta$ and that it is a HC $(\A_{\paramq},\chi)$-bimodule.

For $\lambda\in \paramq$, we set
$$\A_{\lambda,\chi}:=\A_{\paramq,\chi}\otimes_{\C[\paramq]}\C_\lambda,
\A_{\lambda,\chi}^\theta:=\A_{\paramq,\chi}\otimes_{\C[\paramq]}\C_\lambda.$$

We will need the following properties of translation bimodules.

\begin{Lem}\label{Lem:specialization}
Suppose that $H^i(X^\theta, \Str(\chi))=0$ for all $i>0$. Then the following claims hold:
\begin{enumerate}
\item $R\Gamma(\A_{\paramq,\chi}^\theta)=\A_{\paramq,\chi}$.
\item The bimodule $\A_{\paramq,\chi}$ is a flat right $\C[\paramq]$-module.
\item $R\Gamma(\A_{\lambda,\chi}^\theta)=\A_{\lambda,\chi}$.
\end{enumerate}
\end{Lem}
\begin{proof}
Since $H^i(X^\theta, \Str(\chi))=0$ for $i>0$, we see that $H^i(X_\param^\theta, \Str_\param(\chi))=0$
for $i>0$. Moreover, $\Gamma(\Str_\param(\chi))$ is a flat deformation of $\Gamma(\Str(\chi))$.
Apply the standard long exact sequence in cohomology for $R\Gamma$ to
$$0\rightarrow R_\hbar(\A_{\paramq,\chi}^\theta)\xrightarrow{\hbar\cdot}
R_\hbar(\A_{\paramq,\chi}^\theta)\rightarrow \Str_{\param}(\chi)\rightarrow 0,$$
where  $R_\hbar(\A_{\paramq,\chi}^\theta)$ is the Rees sheaf of $\A_{\paramq,\chi}^\theta$.
We see that $R\Gamma(\A_{\param,\chi}^\theta)=\A_{\paramq,\chi}$ (which is part (1)) and the right hand side  is a filtered deformation of $\Gamma(\Str_\param(\chi))$. The latter is flat, hence free over
$\C[\param]$, hence so is  $\A_{\paramq,\chi}$. This proves (2). (3) easily follows from
here and also follows from \cite[Proposition 6.26]{BPW}.
\end{proof}

\begin{Rem}\label{Rem:translation_coincidence}
Note that if $\Gamma_\lambda^\theta$ is exact, then the conclusion of (3) holds. This also
follows from \cite[Proposition 6.26]{BPW}.
 \end{Rem}

\subsubsection{Connection to abelian localization}
Here we discuss a connection between the translation bimodules and abelian localization.

We write $\Gamma_\lambda^\theta$ for the global
section functor $\Coh(\A_\lambda^\theta)\rightarrow \A_\lambda\operatorname{-mod}$
and $\Loc_\lambda^\theta$ for its left adjoint functor. The following lemma is
elementary.

\begin{Lem}\label{Lem:fun_iso1}
We have a functor isomorphism
$$R\Gamma(\A_{\lambda,\chi}^\theta)\otimes^L_{\A_\lambda}\bullet
\cong R\Gamma_{\lambda+\chi}^\theta(\A_{\lambda,\chi}^\theta\otimes_{\A_\lambda^\theta}
L\Loc_\lambda^\theta(\bullet)).$$
\end{Lem}


\begin{Cor}\label{Cor:abel_loc_bimod}
Suppose that abelian localization holds for $\A_{\lambda+\chi}^\theta$. Then the equivalence $$\Gamma_{\lambda+\chi}^\theta(\A_{\lambda,\chi}^\theta\otimes_{\A_\lambda^\theta}\bullet):
\Coh(\A^\theta_\lambda)\xrightarrow{\sim} \A_{\lambda+\chi}\operatorname{-mod}.
    $$
    identifies $\Gamma(\A^\theta_{\lambda,\chi})\otimes^L_{\A_\lambda}\bullet$ with $L\Loc_\lambda^\theta$
    and $R\Hom_{\A_{\lambda+\chi}}(\Gamma(\A^\theta_{\lambda,\chi}),\bullet)$ with $R\Gamma_\lambda^\theta$.
\end{Cor}

We also have the  following results, \cite[Proposition 2.11]{catO_charp}, a related
result was previously obtained in \cite[Section 5.3]{BPW}.

\begin{Prop}\label{Prop:localization_translation}
For $\lambda\in \paramq$ the following two conditions are equivalent:
\begin{enumerate}
\item Abelian localization holds for $\lambda$,
\item There is $\chi\in \Lambda$ satisfying the following conditions
\begin{itemize}
\item $\chi$ is ample for $X^\theta$,
\item $H^i(X^\theta,\Str(m\chi))=0$ for all $i,m>0$,
\item
and for all $m\geqslant 0$, the bimodules $\A_{\lambda+m\chi,\chi}$
and $\A_{\lambda+(m+1)\chi,-\chi}$ are mutually inverse Morita equivalences.
\end{itemize}
\end{enumerate}
\end{Prop}

\subsubsection{Restriction functors}\label{SSS_res_fun}
We use the notation of Section \ref{SSS_slice_quant}. We have the restriction
functor $\bullet_{\dagger,y}:\HC(\A_{\paramq},\chi)\rightarrow \HC(\underline{\A}_{\paramq},\chi)$.

In this generality the functor was constructed in \cite[Section 2.4.4]{catO_charp}.
It has the following properties:
\begin{itemize}
\item[(i)] It is exact and $\C[\paramq]$-linear.
\item[(ii)] We have $\B_{\dagger,y}=0$ if and only if $y\not\in \VA(\B)$.
\item[(iii)] For $\B\in \HC(\A_\lambda,\chi)$ and $y\in Y$, the space $\B_{\dagger,y}$ is finite dimensional and nonzero
if and only if $\mathcal{L}$ is open in $\VA(\B)$.
\item[(iv)] The functor $\bullet_{\dagger,y}$ is monoidal.
\item[(v)] Assume that $H^i(X^\theta, \Str(\chi))=0$ for all $i>0$.
Then $(\A_{\paramq,\chi})_{\dagger,y}=\underline{\A}_{\paramq,\chi}$.
\end{itemize}

We note that the bimodule $\underline{\A}_{\paramq,\chi}$ is obtained
from $\C[\paramq]\otimes_{\C[\underline{\paramq}]}\underline{\A}_{\underline{\paramq},\eta(\chi)}$,
where $\eta$ here is the  map $\Lambda\rightarrow \underline{\Lambda}$
described in Section \ref{SSS_conical_slices},
by shifting the left $\C[\paramq]$-action so that (i) from Definition \ref{defi:HC} holds.

From (iv),(v) and Proposition  \ref{Prop:localization_translation}
we deduce the following.

\begin{Cor}\label{Cor:loc_slice}
If abelian localization holds for $\A_\lambda^\theta$, then it also
holds for $\underline{\A}_\lambda^\theta$.
\end{Cor}

\subsection{Essential and singular hyperplanes}\label{SS_essent_sing}
\subsubsection{Construction of singular hyperplanes}
Here we will recall the notion of an essential hyperplane and introduce the notion of
a singular hyperplane. Pick a classical wall $\Upsilon$. Fix $\zeta\in \Upsilon$ as in Proposition \ref{Prop:rank1} so that, for all $y\in Y_\zeta$,
the rank of the map $\eta:\paramq\rightarrow \underline{\paramq}$ is equal to $1$.
Let $\underline{\paramq}'$ denote the image of $\eta$, it is 1-dimensional by Proposition
\ref{Prop:rank1}. A priori, as a subspace of $\underline{\paramq}$, it depends on the choice of
$y$. We consider the lattice $\underline{\param}'_\Z$ that is the image of
$\Lambda$ in $\underline{\param}$, it is isomorphic to $\Z$.

Consider the algebra $\underline{\A}_{\underline{\paramq}'}$. Pick an ample line bundle
$\Str(\chi)$ on $X$. Let $\Str(\underline{\chi})$ denote the induced line bundle on
$\underline{X}$. Consider the bimodules $\underline{\A}_{\underline{\paramq}',\underline{\chi}},
\underline{\A}_{\underline{\paramq}',-\underline{\chi}}$. Note that $\underline{\chi}$ is  ample.
It follows that the locus in $\underline{\paramq}'$,
where these bimodules fail to be mutually inverse Morita equivalence bimodules
is a finite subset, compare to, e.g., \cite[Lemma 2.12]{catO_charp}.
This finite subset will be denoted by $\Sigma_y$.

We define $\Sigma_\Upsilon$ as follows. We pick points $y_1,\ldots,y_k$, one in
each minimal symplectic leaf in $Y_\zeta$. Consider the union $\bigcup_i \Sigma_{y_i}$.
By definition,
$\Sigma_\Upsilon$  consists of all $\underline{\lambda}$ in $\underline{\paramq}'$
such that either $\underline{\lambda}\in \bigcup\Sigma_{y_i}$ or there are $\underline{\chi}_1,\underline{\chi}_2\in \underline{\param}'_{\Z}$
of different signs such that
 $\underline{\lambda}+\underline{\chi}_j\in \bigcup\Sigma_{y_i}, j=1,2$.

\begin{defi}\label{defi:singular_hyperplane}
By a {\it singular hyperplane} in $\paramq$, we mean a hyperplane of the form
$\eta^{-1}(\sigma)$, where $\eta:\paramq\twoheadrightarrow \underline{\paramq}'$
is associated to an arbitrary classical wall $\Upsilon$ and $\sigma\in \Sigma_\Upsilon$.
\end{defi}

\begin{Ex}\label{Ex:Springer_singular}
Let $X=T^*\mathcal{B}$. Then $\underline{X}=T^*\mathbb{P}^1$ so that $\underline{\paramq}\cong
\underline{\paramq}'\cong \C$. There is only one minimal symplectic leaf in $Y_\zeta$
(equivalently, in $\underline{Y}$). We have $\Sigma_\Upsilon=\{0\}$. It follows that
the singular hyperplanes are precisely the classical walls (=root hyperplanes).
\end{Ex}

As was mentioned in the introduction, the essential hyperplanes are defined as follows.

\begin{defi}\label{defi:essential_hyperplanes}
By an {\it essential hyperplane} in $\paramq$ we mean an affine hyperplane, which is obtained from
a singular hyperplane by a translation by an element of $\param_{\Z}$. We say that a classical wall
is {\it relevant} for $\lambda$ if $\lambda$ lies in the essential hyperplane parallel to that
classical wall.
\end{defi}

Note that this definition is the same as in \cite[Section 4.1]{catO_charp}, but
our notation here is different.

\subsubsection{Properties of essential and singular hyperplanes}
Here is a result that follows from \cite[Lemma 4.6]{catO_charp}.

\begin{Prop}\label{Prop:outside_essential}
Let $\lambda\in \paramq$ lie outside of the union of all essential hyperplanes.
Then $\A_\lambda$ is simple and has finite homological dimension. Moreover,
abelian localization holds for $\A_\lambda^\theta$ for all generic $\theta$.
\end{Prop}

The following claim follows from the proof of \cite[Lemma 4.7]{catO_charp}.

\begin{Lem}\label{Lem:inverse_Morita}
Let $\chi\in \Lambda$. Then the locus in $\paramq$ where
$\A_{\paramq,\chi}\otimes_{\A_{\paramq}} \A_{\paramq,-\chi}\rightarrow \A_{\paramq}$ is an isomorphism is Zariski open and contains the complement to a finite union
of essential hyperplanes.
\end{Lem}

We also have the following two results on abelian localization. The following
claim follows from Lemma \cite[Lemma 4.7]{catO_charp}.

\begin{Lem}\label{Cor:abel_loc}
Let $\theta$ be generic and $\chi\in \Lambda$
correspond to an ample line bundle on $X^\theta$.
There is a finite collection of essential hyperplanes
such that abelian localization holds for $(\lambda,\theta)$ if,
for, the set $\{\lambda+n\chi| n\geqslant 0\}$ does not intersect the
union of these essential hyperplanes.
\end{Lem}

 Note that
a choice of a non-singular essential hyperplane $\tilde{\Upsilon}$ gives a preferred halfspace
for $\Upsilon$: we pick the unique halfspace $\Theta$ such that $(\tilde{\Upsilon}+\Theta)\cap(\tilde{\Upsilon}+\param_\Z)$ does not intersect singular
hyperplanes. Recall that an element of an irreducible algebraic variety over $\C$
is called {\it Weil generic} if it does not lie in the union of countably many
subvarieties.

\begin{Lem}\label{Lem:ab_loc_Weil_generic}
Let $\lambda$ be a Weil generic point in a non-singular essential hyperplane $\tilde{\Upsilon}$
and let $\theta$ be generic and lying in the halfspace discussed in the previous
paragraph. Then   abelian localization holds for $\A_\lambda^\theta$.
\end{Lem}
\begin{proof}
Fix  a generic element $\chi\in \Lambda$ lying in the chamber of $\theta$
such that
\begin{itemize}
\item $H^i(X^\theta, \Str(m\chi))=0$ for all $m,i>0$,
\item $H^i(X^{-\theta}, \Str(-m\chi))=0$ for all $m,i>0$,
\item $\tilde{\Upsilon}+n\chi$ is not one of the hyperplanes in
Lemma \ref{Lem:inverse_Morita} for all $n>0$.
\end{itemize}
This choice of $\chi$ guarantees that abelian localization holds for $(\lambda+\chi,\theta)$
for a Weil generic element $\lambda\in \tilde{\Upsilon}$, thanks to Lemma \ref{Cor:abel_loc}.
 Thanks to Proposition
\ref{Prop:localization_translation}, it is enough to show that
$\A_{\lambda,\chi}$ and $\A_{\lambda+\chi,-\chi}$ are mutually inverse
Morita equivalences. Consider the bimodules $$\A_{\tilde{\Upsilon},\chi}:=
\A_{\paramq,\chi}\otimes_{\C[\paramq]}\C[\tilde{\Upsilon}],
\A_{\tilde{\Upsilon}+\chi,-\chi}:=\A_{\paramq,\chi}\otimes_{\C[\paramq]}\C[\tilde{\Upsilon}+\chi].$$ Let $\B$ be the direct sum of the kernel and cokernel
of $$\A_{\tilde{\Upsilon},\chi}\otimes_{\A_{\tilde{\Upsilon}}}
\A_{\tilde{\Upsilon}+\chi,-\chi}\rightarrow \A_{\tilde{\Upsilon}+\chi}.$$
Pick a point $y$ in a minimal symplectic leaf in $Y_\zeta$. Form the corresponding
varieties $\underline{X},\underline{Y}$, and the quantization $\underline{\A}_\paramq$.
Recall that the restrictions of
$\A_{\tilde{\Upsilon},\chi},\A_{\tilde{\Upsilon}+\chi,-\chi}$
associated with $y$ are the translation bimodules for $\underline{\A}$, see
(v) in Section \ref{SSS_res_fun}. By the choice of $\tilde{\Upsilon}$ and
$\chi$, these restrictions are mutually inverse Morita equivalences.
So $\B_{\dagger,y}=0$.  It follows from here (and the same argument
for $\lambda$ and $\lambda+\chi$ swapped) that the locus in $\tilde{\Upsilon}$,
where $\A_{\lambda,\chi},\A_{\lambda+\chi,-\chi}$ fail to be mutually
inverse Morita equivalences is proper (and Zariski closed, by
Lemma \ref{Lem:support_closed}). So for a Weil generic $\lambda\in \tilde{\Upsilon}$,
the bimodules $\A_{\lambda,\chi}$ and $\A_{\lambda+\chi,-\chi}$ are mutually
inverse Morita equivalence bimodules and abelian localization holds for $\A_{\lambda+\chi}^\theta$.
As mentioned above in the proof, it follows that abelian localization holds for
$\A_\lambda^\theta$.
\end{proof}

\subsubsection{The case of quantized quiver varieties}\label{SSS_quant_quiver_walls}
Here we describe the singular hyperplanes for quantized quiver varieties of finite and
affine type. So let $X=\M^\theta(v,w)$ be a quiver variety.
We assume that the weight $\nu$ is dominant.
The following two propositions describe the result.

\begin{Prop}\label{Prop:singular_real}
Suppose that $Q$ is of finite or affine type.
Let $\beta$ be a real root such that $\ker\beta$ is a classical wall.
In the notation of Lemma \ref{Lem:leaves_generic_real},
the singular hyperplanes parallel to the wall $\ker\beta$
are given by $$\langle\beta, ?\rangle=\frac{1}{2}\underline{w}-i$$
for $i=1,\ldots, \underline{w}-1$.
\end{Prop}
\begin{proof}
The proof reduces to the case when the quiver has type $A_1$. Here it follows
from the abelian Beilinson-Bernstein localization theorem for twisted differential operators
on the Grassmanians -- note that we have precisely the parameters such that abelian localization
fails for both choices of $\theta$.
\end{proof}

%

\begin{Prop}\label{Prop:singular_imaginary}
Suppose that $Q$ is of affine type.
Let $\delta$ be the indecomposable imaginary root and $\ker\delta$ be a classical wall.
Then the singular hyperplanes parallel to $\ker\delta$ are of the form
$\underline{w}/2-\alpha$, where $\alpha\in (0,\underline{w})$ is an element
with denominator between $1$ and $\underline{v}$.
\end{Prop}
\begin{proof}
We again reduce the proof to the case of a quiver with a single vertex. Here
the claim follows from results of \cite[Section 5]{Gies}.
\end{proof}


\subsection{Categories $\mathcal{O}$}\label{SS_cat_O}
\subsubsection{Hamiltonian tori actions and gradings}\label{SSS_O_basics}
Let $T$ be a torus. Suppose that we have a Hamiltonian action of $T$ on $X$
that commutes with contracting $\C^\times$. Then $X^T$ is a smooth symplectic variety.
Moreover, we have seen in \cite[Proposition 5.1]{catO_R} that every irreducible component of
$X^T$ is a symplectic resolution.

The action of $T$ lifts to a Hamiltonian action on $\A_{\paramq}^\theta$
and hence gives a Hamiltonian action on $\A_{\paramq}$, see, e.g.,
\cite[Proposition 3.11]{BPW}. Let $\Phi:\mathfrak{t}
\rightarrow \A_{\paramq}$ denote the quantum comoment map.

Choose a  one-parameter subgroup $\nu$ of $T$. Let $h\in \A_{\paramq}$ denote
$\Phi(d_1\nu)$. The choice of $\nu$ equips $\A_{\paramq}$ with a grading,
$\A_{\paramq}=\bigoplus_{i\in \Z}\A_{\paramq}^i$. We set
$\A_{\paramq}^{\geqslant 0}=\bigoplus_{i\geqslant 0}\A_{\paramq}^i$
and define $\A_{\paramq}^{>0}$ similarly.

We will also use this construction with $\A_\lambda$ instead of $\A_{\paramq}$.

\subsubsection{Cartan subquotients}\label{SSS_Cartan_subquotient}
We define the subquotient
$$\Ca_\nu(\A_{\paramq}):=\A_{\paramq}^{\geqslant 0}/(\A^{\geqslant 0}_\paramq\cap \A_{\paramq}\A^{>0}_{\paramq})$$
to be called the {\it Cartan subquotient} of $\A_{\paramq}$. We note that if
$X^{\nu(\C^\times)}$ is finite, then $\Ca_\nu(\A_{\paramq})$ is a finitely
generated $\C[\paramq]$-module.

We can also define a microlocal sheaf $\Ca_{\nu}(\A_{\paramq}^\theta)$ on $X_{\paramq}^{\nu(\C^\times)}$
that is a quantization of $X_{\paramq}^{\nu(\C^\times)}$, see \cite[Section 5.2]{catO_R}.
We have an algebra
homomorphism $\Ca_\nu(\A_{\paramq})\rightarrow \Gamma(\Ca_\nu(\A_{\paramq}^\theta))$.

We will mostly need the special case when $X^{\nu(\C^\times)}$ is finite. In this case
$\Ca_\nu(\A^\theta_{\paramq})$  is the direct sum of copies of $\C[\paramq]$
labelled by $X^T$. In this situation we have the following result
from \cite[Proposition 4.13]{catO_charp}.

\begin{Lem}\label{Lem:cartan_iso}
Suppose $X^{\nu(\C^\times)}$ is finite. Then the locus in $\paramq$ where
$\Ca_\nu(\A_{\paramq})\rightarrow \Ca_\nu(\A^\theta_{\paramq})$ is an isomorphism
is Zariski open, and moreover, contains the complement to a finite union of
essential hyperplanes.
\end{Lem}

%
%

\subsubsection{Category $\mathcal{O}$ and Verma modules}
From now on and until the end of Section \ref{SS_cat_O}
we assume that $X^T$ is finite and $\nu:\C^\times\rightarrow T$
is generic meaning that $X^{\nu(\C^\times)}=X^T$.

We define $\Ocat_\nu(\A_{\lambda})$ as the category of all finitely generated $\A_\lambda$-modules,
where $\A_\lambda^{>0}$ acts locally nilpotently. We have the Verma module functor
$\Delta_\lambda:\Ca_\nu(\A_\lambda)\operatorname{-mod}\rightarrow \Ocat_\nu(\A_\lambda)$ given by tensoring with the $\A_\lambda$-$\Ca_\nu(\A_\lambda)$-bimodule $\A_\lambda/\A_{\lambda}\A_{\lambda}^{>0}$.

Fix $p\in X^T$. This gives rise to the 1-dimensional $\Ca_\nu(\A^\theta_\lambda)$-module to be denoted
by $\C_p$. Thanks to the homomorphism $\Ca_\nu(\A_\lambda)\rightarrow \Ca_\nu(\A_\lambda^\theta)$
we can view $\C_p$ as a $\Ca_\nu(\A_\lambda)$-module. Set $$\Delta_\lambda(p):=
[\A_\lambda/\A_{\lambda}\A_{\lambda}^{>0}]_{\otimes_{\Ca_\nu(\A_\lambda)}}\C_p.$$

Similarly, we get the $\Ca_\nu(\A_{\paramq})$-module $\C[\paramq]_p$ and the
$\A_{\paramq}$-module $\Delta_\paramq(p)$.

The following result was established in \cite[Section 3.4.2]{catO_charp}.

\begin{Lem}\label{Lem:Verma_flatness}
The locus in $\paramq$ where  $\Delta_{\paramq}(p)$ is flat is Zariski open and contains
the complement to the union of finitely many essential hyperplanes.
\end{Lem}

Let $N$ be an irreducible $\Ca_\nu(\A_\lambda)$-module. Then $\Delta_\lambda(N)$
has a unique irreducible quotient to be denoted by $L_\lambda(N)$. The map
$N\mapsto L_\lambda(N)$ defines a bijection between $\operatorname{Irr}(\Ca_\nu(\A_\lambda))$
and $\operatorname{Irr}(\Ocat_\nu(\A_\lambda))$.

Note that $\Delta_\nu(N)$ is graded with $N$ in degree $0$. On the $i$th
graded component the element $h$ acts by $h_N-i$, where $h_N$ is the scalar
of action of $h$ on $N$. Since every object in $\Ocat_\nu(\A_\lambda)$
admits a finite filtration by quotients of $\Delta_\lambda(N)$'s,
the element $h$ acts locally finitely on every object in $\Ocat_\nu(\A_\lambda)$.
And since $\operatorname{Irr}(\Ocat_\nu(\A_\lambda))$ is finite, we see
that every $\Delta_\nu(\A_\lambda)$ has finite length. And therefore
every object in $\Ocat_\nu(\A_\lambda)$ has finite length, see
also \cite[Lemma 5.7]{BLPW}.




Finally, let us discuss derived translation functors between the categories $\mathcal{O}$.
We write $D^b_{\Ocat}(\A_\lambda)$ for the full subcategory in $D^b(\A_\lambda\operatorname{-mod})$
consisting of all complexes with homology in $\Ocat_\nu(\A_\lambda)$. Note that a finitely
generated $\A_\lambda$-module lies in $\Ocat_\nu(\A_\lambda)$ if it admits a weakly
$\nu(\C^\times)$-equivariant structure and is supported on the contracting locus of $\nu$ in $Y$.
The functor  $\A_{\lambda,\chi}\otimes^L_{\A_\lambda}\bullet$ preserves both conditions and so maps $D^b_{\Ocat}(\A_\lambda)$
to $D^b_{\Ocat}(\A_{\lambda+\chi})$.

\subsubsection{Highest weight structure}\label{SSS_HW_struct}
The algebra $\A_{\lambda}^{opp}$ is identified with $\A_{-\lambda}$, this follows
from \cite[Section 2.3]{quant_iso}.
Moreover, $\Ca_{-\nu}(\A_\lambda^{opp})$ is naturally identified with
$\Ca_\nu(\A_\lambda)^{opp}$. So we can consider the Verma modules
$\Delta_\lambda^{opp}(p)$ as well as $\Delta_{\paramq}^{opp}(p)$. The following claim
has been established in \cite[Proposition 4.13]{catO_charp}.

\begin{Lem}\label{Lem:Tor_locus}
There is a finite union of essential hyperplanes such that for all $\lambda$
outside of this union we have $\Ca_\nu(\A_\lambda)\xrightarrow{\sim}\Ca_\nu(\A_\lambda^\theta)$
and $\operatorname{Tor}_i^{\A_\lambda}(\Delta_{\lambda}^{opp}(p), \Delta_{\lambda}(p'))=0$
for all $p,p'\in X^T$ and $i>0$.
\end{Lem}

The next result follows in the same way as \cite[Theorem 5.12]{BLPW}.

\begin{Cor}\label{Cor:hw}
In the locus described in Lemma \ref{Lem:Tor_locus}, the category $\Ocat_\nu(\A_\lambda)$
is highest weight and the standard objects are the Verma modules $\Delta_\lambda(p)$.
\end{Cor}

An order making $\Ocat_\nu(\A_\lambda)$ into a highest weight category is as follows.
Let $h_p(\lambda)$ be the image of $h$ in $\C_p$ at the parameter $\lambda$.
We set $p<_\lambda p'$ if $h_{p'}(\lambda)-h_p(\lambda)\in \Z_{>0}$. This
is a highest weight order for $\Ocat_\nu(\A_\lambda)$. We also note that, as every
highest weight category, the category $\Ocat_\nu(\A_\lambda)$
admits the minimal (with respect to inclusion) highest
weight order: it is generated by the relation $\rightsquigarrow$ given by
$p\rightsquigarrow p'$ if
\begin{itemize}
\item[(i)]
the simple $L_\lambda(p')$ appears in $\Delta_\lambda(p)$
\item[(ii)]
or the standard $\Delta_\lambda(p)$ appears in the projective cover $P_\lambda(p')$
of $L_\lambda(p')$.
\end{itemize}

Here is another important corollary, \cite[Corollary 5.13]{BLPW}.

\begin{Cor}\label{Cor:full_embedding}
Let $\lambda$ be in the locus described in Lemma \ref{Lem:Tor_locus}.
Further, suppose abelian localization holds for $\A_\lambda^\theta$
with some generic $\theta$.
Then the natural functor $D^b(\Ocat_\nu(\A_\lambda))\rightarrow
D^b_{\Ocat}(\A_\lambda)$ is an equivalence.
\end{Cor}

By applying a contravariant duality functor $\Ocat_{-\nu}(\A^{opp}_{\lambda})
\xrightarrow{\sim} \Ocat_\nu(\A_\lambda)$, see, e.g., \cite[Section 3.4.3]{catO_charp},
to the objects $\Delta^{opp}_{\lambda}(p)$ we get the objects denoted by
$\nabla_\lambda(p)$. If $\lambda$ is in the locus described in Lemma
\ref{Lem:Tor_locus}, then the objects $\nabla_\lambda(p)$ are costandard
in the highest weight category $\Ocat_\nu(\A_\lambda)$.

\begin{Lem}\label{Lem:stand_costand}
There is a finite union of essential hyperplanes such that, for $\lambda$
outside of that locus, the $K_0$-classes of $\Delta_\lambda(p)$ and
$\nabla_\lambda(p)$ coincide. In particular, the conditions
(i) and (ii) above are equivalent.
\end{Lem}
\begin{proof}
This is proved as \cite[Corollary 6.4]{BLPW} but we use
\cite[Proposition 4.10]{catO_charp} to prove that the classes
of $\Delta_\lambda(p)$ form an orthonormal basis with respect to
the Ext pairing, which is a claim of \cite[Proposition 6.3]{BLPW}.

The last claim of the lemma follows from the general form of
the BGG reciprocity.
\end{proof}

%

\subsection{Tilting bundles from quantizations in positive characteristic}
\subsubsection{Tilting generators}
Let $X$ be a smooth CY algebraic variety that is projective over an affine variety $Y$.
By a {\it tilting generator} on $X$ we mean a vector bundle $E$ on $X$ subject
to the following two conditions.

\begin{enumerate}
\item $\Ext^i_{\Str_X}(E)=0$ for all $i>0$,
\item and the algebra $H:=\End(E)$ has finite homological dimension.
\end{enumerate}

We can consider the functor $R\Gamma(\mathcal{T}\otimes \bullet):
D^b(\Coh(X))\rightarrow D^b(H\operatorname{-mod})$. This functor
is a category equivalence, \cite[Proposition 2.2]{BK}. Here is another useful
general property, \cite[Lemma 4.2]{BL_modular}.

\begin{Lem}\label{Lem:tilting_CM}
The algebra $H$ is Cohen-Macaulay and Gorenstein.
\end{Lem}

\subsubsection{Construction}\label{SSS_tilting_construction}
Let $X\rightarrow Y$ be a conical symplectic resolution. In many cases one can construct a
tilting generator on $X$ using quantizations in characteristic $p\gg 0$\footnote{There is
a general construction due to Kaledin, \cite{Kaledin}, which, however, has some gaps. These gaps
are supposed to be easy to fix. Still, we will only consider special cases as the property
($\heartsuit$) that is crucial for our purposes only holds in special cases}. The construction
below is known to hold in the following cases:
\begin{itemize}
\item[(i)] $X=T^*\mathcal{B}$, see \cite[Section 1.5]{BM}, and, more generally, for $X=T^*(G/P)$, \cite[Section 4]{BM}.
It also carries over to parabolic Slodowy varieties.
\item[(ii)] $X=\mathcal{M}^\theta_0$ for a generic $\theta$, \cite[Section 7]{BL}.  This includes almost all of the symplectic resolutions of  symplectic quotient singularities, that case was fully treated in
\cite{BK}.
\item[(iii)] $X$ is a smooth Coulomb branch of a gauge theory, \cite[Section 3]{Webster_tilting}.
\end{itemize}

We will recall the construction of $E$ as we will need it below.

In all of these cases, $X$
is defined over a finite localization of $\Z$, to be denoted by $\ring$. This is what
we are going to assume to slightly simplify the exposition. Let $X_{\ring}$
be a form of $X$ over $\ring$. Replacing $\ring$ with a finite localization, we can
ensure that $X_{\ring}$ behaves nicely, for example, that $X_{\ring}\rightarrow Y_{\ring}$
is a symplectic resolution, $R\Gamma(\Str_{X_\ring})=\ring[Y]$, etc.

Choose $p\gg 0$ and let $\F:=\bar{\F}_p$. Consider the Frobenius twist $X_{\F}^{(1)}$. Note that $X_{\F}^{(1)}$ is isomorphic to $X_{\F}$. We have an $\F$-scheme
morphism $X_{\F}\rightarrow X_{\F}^{(1)}$ called the {\it Frobenius
morphism} and denoted by $\Fr$.

The following definition is essentially due to \cite{BK}.

\begin{defi}\label{defi_Frob_constant}
Let $\mathfrak{A}$ be a sheaf of Azumaya algebras on $X_\F^{(1)}$. We say that $\mathfrak{A}$ is a
{\it Frobenius constant quantization} if its restriction to the conical topology on $X$ is equipped with a filtration such that we have a graded Poisson  $\Str_{X_\F^{(1)}}$-linear isomorphism
$\gr\mathfrak{A}\cong \Fr_*\Str_{X_\F}$.
\end{defi}

We note that the $\hbar$-adic completion $R_\hbar(\mathfrak{A})^{\wedge_\hbar}$
may be viewed as a formal
quantization of $X_\F$. This formal quantization is a $\mathbb{G}_m$-equivariant
sheaf of $\Str_{X_\F^{(1)}}[[\hbar]]$-algebras. One can recover $\mathfrak{A}$
from $R_\hbar(\mathfrak{A})^{\wedge_\hbar}$ by taking $\mathbb{G}_m$-finite sections
and then setting $\hbar=1$. This has the following corollary.

\begin{Lem}\label{Lem:cohomology_vanishing}
Let $\varphi: X_\F\rightarrow X'_\F$ be a projective birational morphism  such that
$R^i\varphi_*\Str_{X_\F}=0$ for $i>0$.  Then, for $\varphi$ viewed as a morphism
$X_\F^{(1)}\rightarrow X_\F'^{(1)}$, we have $R^i\varphi_* \mathfrak{A}=0$.
\end{Lem}

We can produce a tilting generator out of $\mathfrak{A}$ provided $\Gamma(\mathfrak{A})$
has finite homological dimension, which is equivalent, by \cite[Section 2.2]{BK} to
$R\Gamma: D^b(\Coh(\mathfrak{A}))\xrightarrow{\sim} D^b(\Gamma(\mathfrak{A}))$
being an equivalence. In the examples (i)-(iii) above we can construct
$\mathfrak{A}$ by reducing from characteristic $0$. Namely, we have $\lambda\in \paramq_{\mathbb{Q}}$
such that $\A_\lambda$ has finite homological dimension. Then we get $\mathfrak{A}$ out of
$\mathcal{A}_\lambda^\theta$ (we reduce to $\F$, then pass from $\mathcal{A}_{\lambda,\F}^\theta$ to
a formal quantization of $X_\F$, in all the cases of interest this quantization
comes from a Frobenius constant quantization as explained in the paragraph preceding
Lemma \ref{Lem:cohomology_vanishing}).

Once we have $\mathfrak{A}$ as in the previous paragraph, to construct a tilting generator
we proceed as follows.

\begin{itemize}
\item[(a)] In all of our examples, the Azumaya algebra
$\mathfrak{A}^{\wedge_0}:=\mathfrak{A}|_{X_\F^{(1)\wedge_0}}$ splits.
Let $\hat{E}_\F^{(1)}$ denote a splitting bundle. Recall that it is
defined up to a twist with a line bundle.
\item[(b)] The bundle $\hat{E}_\F^{(1)}$ has a $\mathbb{G}_m$-equivariant
structure by \cite{Vologodsky} and we can use it to extend $E_\F^{(1)\wedge_0}$
to a unique equivariant bundle  $E_\F^{(1)}$ on $X_\F^{(1)}$. Then we can view the extension as a bundle $E_\F$ on $X_\F$.
\item[(c)] The bundle $E_\F$ is defined over a finite field $\F_q$.
Let $R^{\wedge_q}$ denote a complete local ring with residue field $\F_q$
and fraction field that can be embedded into $\C$. We  deform $E_{\F_q}$
to a vector bundle $E^{\wedge_q}_{R}$ on the formal neighborhood of $X_{\F_q}$
in $X_{R^{\wedge_q}}$.
\item[(d)] Thanks to the $\mathbb{G}_m$-equivariance, from $E^{\wedge_q}_{R}$
we get a vector bundle $E_{R^{\wedge_q}}$ on $X_{R^{\wedge_q}}$. Then we
change the base to $\C$, finally getting $E$.
\end{itemize}

\subsubsection{Consequences}
Note that we can deform $E$ to a vector bundle $E_\param$ on $X_{\param}$.
Namely, we first deform $E$ to the formal neighborhood of $X$ in $X_{\param}$,
which is possible because $E$ has no higher self-extensions. Then we use the $\mathbb{G}_m$-equivariance
to extend from the formal neighborhood to $X$ itself.

The following proposition summarizes properties of $E_\param$ that we will need.

\begin{Prop}\label{Prop:tilting_properties1}
The following claims are true:
\begin{itemize}
\item
$\operatorname{Ext}^i(E_\param,E_\param)=0$ for $i>0$.
In particular, $H_{\param}:=\End(E_\param)$ is a flat deformation of $H$.
\item The functor $R\Gamma(E_\param\otimes\bullet)$ is a derived
equivalence $D^b(\Coh(\Str_{X_\param}))\xrightarrow{\sim} D^b(H_\param\operatorname{-mod})$.
\item The algebra $H_{\param}$ is independent of the choice of
the symplectic resolution $X$.
\end{itemize}
\end{Prop}
\begin{proof}
(1) and (2) are standard consequences of the construction of $E_\param$
and (3) is proved in the same way as the analogous result in \cite[Section 5.1]{Kaledin}.
Namely, let $Y^2$ denote the complement in $Y$ of the  symplectic leaves
of codimension $\geqslant  4$. Let $X^2$ denote the preimage of $Y^2$ in $X$.
For another symplectic resolution $X'\rightarrow Y$, we have a unique
isomorphism $X^2\xrightarrow{\sim} X'^2$ of schemes over $Y$. Let $E,E'$
be tilting generators on $X,X'$ produced from the quantizations with the same
parameter. As in \cite[Section 5.1]{Kaledin} we prove that the pullbacks of $H,H'$
to $Y^2$ are isomorphic, hence $H\cong H'$. Since $\operatorname{codim}_{Y}Y\setminus Y^2
\geqslant 4$ and $H$ is a Cohen-Macaulay $\C[Y]$-module, see Lemma
\ref{Lem:tilting_CM}, we see that $H^i(Y^2, H|_{Y_2})=0$ for $i=1,2$.
It follows that $H^1(X^2, \mathcal{E}nd(E))=H^1(X'^2, \mathcal{E}nd(E'))=0$.
As in \cite[Section 5.1]{Kaledin}, the restrictions on $\mathcal{E}nd(E)$
and $\mathcal{E}nd(E')$ to $X^2$ and $X'^2$, respectively,  coincide.
Since this sheaf has no 1st cohomology, the deformations of
$$\operatorname{End}(E)=\Gamma(X^2, \mathcal{E}nd(E))=\Gamma(X'^2,
\mathcal{E}nd(E'))$$ coming from $X_\param, X'_\param$ must coincide as well,
by an argument similar to \cite[Section 5.1]{Kaledin}. This finishes the proof.
\end{proof}

\subsubsection{Property ($\heartsuit$)}
We recall our most important assumption on $E$.

\begin{itemize}
\item[($\heartsuit$)] Assume that $E$ has a direct summand of rank $1$.
Then we can twist $E$ with a line bundle and assume that $\Str_X$ is a direct
summand of $E$. Also assume that the indecomposable constituents of $E$ are
independent of $p$ (the prime used in the construction of $E$).
\end{itemize}

Checking ($\heartsuit$) in concrete examples is
quite nontrivial. There are three families of examples, where ($\heartsuit$) is known to
hold.

\begin{Ex}\label{Ex:Springer_rk1}
Let $X$ be the Springer resolution of the nilpotent cone in $\g$.
($\heartsuit$) is true in this case. We take the tilting bundle $\tilde{E}$
for $T^*\mathcal{B}$ that comes from the construction explained
in Section \ref{SSS_tilting_construction} with $\lambda=\rho$ (so
that the corresponding Azumaya algebra $\mathfrak{A}$ is
$D_{\B_\F}$).
The presence of a rank 1 summand follows from \cite[Section 2.2.5]{BMR_sing}
(compare to the proof of \cite[Lemma 4.7]{BL_modular}). By the construction of \cite{BM}, the indecomposable summands of $\tilde{E}$ are independent of the choice of $p$.

More generally, let $X=T^*(G/P)$.  Let $Z:=G\times^B \mathfrak{p}^\perp$
denote the standard correspondence between $T^*\mathcal{B}$ and $T^*(G/P)$
and let $\iota:Z\hookrightarrow T^*\mathcal{B}$ and $\varpi: Z\twoheadrightarrow
T^*(G/P)$ be the natural inclusion and the natural projection, respectively. Then the tilting bundle $E$ obtained from the Azumaya algebra
$D_{G_\F/P_\F}$ is related to $\tilde{E}$ via $E=\varpi_*\iota^* \tilde{E}$. This is by
the construction in \cite[Section 4]{BM}.
Clearly, $\varpi_* \iota^* \Str_{T^*\mathcal{B}}=\Str_{T^*(G/P)}$. So $T^*(G/P)$
satisfies ($\heartsuit$) as well.
\end{Ex}

\begin{Ex}\label{Ex:sing_resol}
Let $V$ be a symplectic vector space, $\Gamma$ a finite subgroup of linear
symplectomorphisms of $V$. We set $Y=V/\Gamma$ and assume there is a projective conical
symplectic resolution to be denoted by $X$. Then ($\heartsuit$) holds thanks to \cite{BK}. More precisely,
for a suitable choice of $\lambda$ above we can replace $E$ with
a direct sum of its direct summands with suitable grading shifts
and achieve that $H\cong \C[V]\#\Gamma$, an  isomorphism of graded
$\C[Y]=\C[V]^\Gamma$ algebras. So, applying the trivial idempotent  $e\in \C\Gamma$
to $E$ we get  a line bundle. The direct summands of $E$ are independent
of the choice of $p$ by the classification result of \cite{Losev_Procesi}.

The resulting tilting generator $E$ with $eE=\Str_X$ will be denoted
by $\Pro$ and called a {\it Procesi} bundle. The algebra  $H_{\param}$
is the symplectic reflection algebra from \cite{EG} at $t=0$, see
\cite[Section 6]{quant_iso}.
\end{Ex}

\begin{Ex}
Let $X$ be a smooth Coulomb branch of a gauge theory. Then ($\heartsuit$) holds for all
choices of $\lambda$, \cite{Webster_tilting}. Indeed, by the paragraph preceding
\cite[Lemma 3.10]{Webster_tilting}, the idempotent $e_{\tau}$ there splits
into the  sum of $|W|$ idempotents. Then we can take $\mu=\tau$ in
\cite[Lemma 3.10]{Webster_tilting} getting a tilting generator
$\hat{\mathcal{Q}}_\tau$ of rank $|W|$. Hitting it with one of those
$|W|$ idempotents we get a line bundle summand.
The independence of $p$ is discussed after \cite[Lemma 3.24]{Webster_tilting}.
So ($\heartsuit$) indeed holds.

We will
be interested in the case when $X$ is a Nakajima quiver variety of finite
or affine type $A$, see Remark \ref{Rem:Coulomb}.
\end{Ex}

Now we will discuss some structural features.
The direct summand $\Str_X$ in $E$ gives rise to an idempotent in $H$
to be denoted by $e$, compare to the proof of Proposition
\ref{Prop:tilting_properties1}. This idempotent is independent of the choice of $X$.
Note that $eHe\cong \C[Y]$ and we have $\Gamma(E)=He$
and $\Gamma(E^*)=eH$. The idempotent $e$ deforms to $H_{\param}$,
denote the extension by the same letter.

We now record some properties of the algebra $H$ and the modules $He$ and $eH$.

\begin{Lem}\label{Lem:He_properties}
The following claims are true:
\begin{enumerate}
\item $eH$ and $He$ are maximal Cohen-Macaulay $\C[Y]$-modules.
\item The natural homomorphism $H\rightarrow \End_{eHe}(He)$ is
an isomorphism.
\end{enumerate}
\end{Lem}
\begin{proof}
(1) is a direct corollary of Lemma \ref{Lem:tilting_CM}.

Now we prove (2). Note that the kernel and the cokernel of  $H\rightarrow \End_{eHe}(He)$
are supported on $Y^{sing}$. Since $H$ is a maximal Cohen-Macaulay $\C[Y]$-module we
see that it is torsion free and $H=\Gamma(H|_{Y^{reg}})$. The former implies that the
homomorphism is injective and the latter implies that it is surjective.
\end{proof}

\section{Exactness of global section functor}

\subsection{Cohomology vanishing for tilting bundles}
Let $E=E^\theta$  be a tilting bundle on $X^\theta$ coming from a quantization in  characteristic $p$ as explained in Section \ref{SSS_tilting_construction}. We write $H^\theta$ for $\mathcal{E}nd(E^\theta)$.

We write $C$ be for the closure of
the chamber of $\theta$.
Pick $\chi\in C$ and let $X^\chi$ be
the partial resolution of $Y$ corresponding to $\chi$. We write $\pi_1$ for the
morphism $X\rightarrow X^\chi$.

\begin{Lem}\label{Lem:direct_image}
We have $R^i \pi_{1*} H^\theta=0$ for $i>0$.
\end{Lem}
\begin{proof}
We use the notation of Section \ref{SSS_tilting_construction}.
Tracking the construction of $E$ in Section \ref{SSS_tilting_construction},
we see that it is enough to
prove that
\begin{equation}\label{eq:another_direct_image}
R^i \pi^{(1)}_{1*} \mathfrak{A}=0, \forall i>0,
\end{equation}
for $p\gg 0$.
Indeed, for each of the following sheaves it is straightforward
to see that the cohomology vanishing of its endomorphism sheaf implies that for the next:
$\hat{E}_\F^{(1)}\dashrightarrow E_\F^{(1)}\dashrightarrow
E_\F\dashrightarrow E_{\F_q}\dashrightarrow E_{R}^{\wedge_q}
\dashrightarrow E_{R^{\wedge_q}}\dashrightarrow E$.
And the cohomology vanishing for $\mathfrak{A}$ implies that for
$\mathfrak{A}^{\wedge_0}=\mathcal{E}nd(E_\F^{(1)\wedge_0})$.

Next note that (\ref{eq:another_direct_image}) will follow once we know that
$R^i \pi^{(1)}_{1*} R_\hbar(\mathfrak{A})^{\wedge_\hbar}=0$ for all $i>0$, here
$\bullet^{\wedge_\hbar}$ stands for the $\hbar$-adic completion. Since
$R_\hbar(\mathfrak{A})^{\wedge_\hbar}$ is a formal deformation of $\operatorname{Fr}_* \Str_{X_\F}$
it suffices to show that $R^i\pi^{(1)}_{1*}\operatorname{Fr}_* \Str_{X_\F}=0$ for all $i>0$.
Note that $R^i\pi^{(1)}_{1*}\operatorname{Fr}_*\cong R^i (\pi^{(1)}_1\circ \Fr)$ and
$\pi_1^{(1)}\circ \Fr=\Fr^\chi\circ \pi_1$, where we write $\Fr^\chi$ for the
Frobenius morphism $X^\chi_\F\rightarrow X^{\chi(1)}_\F$. So we reduce to
checking $R^i \pi_{1*}\Str_{X_\F}=0$. Since $p$ is very large, this claim follows
from $R^i\pi_{1*}\Str_X=0$, where it is a consequence of the Grauert-Riemenschneider
theorem, compare to Lemma \ref{Lem:cohom_vanishing}.
\end{proof}

By the discussion in Section \ref{SSS_partial_resolutions}, there is $m_0>0$ such that
\begin{itemize}
\item
$\Str(m_0\chi)\cong \bar{\rho}^* \mathcal{L}$ for a line bundle
$\mathcal{L}$ on $X^\chi$,
\item and  $\mathcal{L}$ is automatically ample on $X^\chi$.
\end{itemize}

The main result of this section is as follows.

\begin{Prop}\label{Prop:cohomology_vanishing}
The  bundle $H^\theta\otimes\Str(m_0n\chi)$ has no higher cohomology for all $n>0$.
\end{Prop}
\begin{proof}
In the proof we can assume that $n=1$.
Arguing as in the proof of Lemma \ref{Lem:direct_image} and using its claim, we see that we can
reduce the proof to the claim of the current proposition to showing that
\begin{equation}\label{eq:cohom_vanishing}
H^i(X_\F^\chi, \pi^{(1)}_{1*}[\operatorname{Fr}_* \Str_{X_\F}\otimes
\mathcal{L}^{(1)}_\F])=0.
\end{equation}
But $ \pi^{(1)}_{1*}\operatorname{Fr}_* \Str_{X_\F}=\operatorname{Fr}^\chi_* \Str_{X^\chi_\F}$. So the left hand
side of (\ref{eq:cohom_vanishing}) is
$$H^i(X_\F^\chi, \mathcal{L}_\F^{\otimes p}).$$
Since $p$ is sufficiently large and $\mathcal{L}$ is ample, the previous space is zero.
\end{proof}

We remark that if $\chi$ is ample for $X$, then we can take $m_0=1$.

\begin{Cor}\label{Cor:sublattice}
There is a full rank sublattice $\underline{\Lambda}\subset \Lambda$ such that for every $\chi\in \underline{\Lambda}$ there is a generic element $\theta$ such that $H^i(X^\theta, H^\theta(m\chi))=0$
for all $i>0$ and $m>0$. Moreover, for $\theta$ we can choose any element such that $\chi$
lies in the closure of the chamber of $\theta$.
\end{Cor}
\begin{proof}
We note that the image of $\Lambda$ in $\param$ is a lattice. So we can choose a finitely generated
subgroup $\Lambda'\subset \Lambda$ whose image in $\param$ coincides with that of $\Lambda$.
Now the claim of the corollary follows directly from Proposition \ref{Prop:cohomology_vanishing}.
\end{proof}

We note that $\underline{\Lambda}$ embeds into $\param$ under $c_1$.
In what follows we will identify $\underline{\Lambda}$ with its image.

\subsection{Quantizations of tilting bundles and their endomorphisms}
\subsubsection{Construction}
Since $E^\theta$ has no higher self-extensions, we see that it uniquely
quantizes to a right $\A_\lambda^\theta$-module $\Ecal^\theta_\lambda$.
Note that since $\Str_X$ is a direct summand of $E^\theta$,
we see that $\A_\lambda^\theta$ is a direct summand in $\Ecal^\theta_\lambda$.
We set $\Hcal_\lambda^\theta:=\mathcal{E}nd_{\A_\lambda^\theta}(\Ecal^\theta_\lambda)$.

Similarly, we can start with $E_\param^\theta$ and produce a right $\A_\paramq^\theta$-module
$\Ecal^\theta_{\paramq}$ and a sheaf of algebras $\Hcal^\theta_\paramq$. Note that, by the construction,
$\Ecal^\theta_\lambda=\Ecal^\theta_\paramq\otimes_{\C[\paramq]}\C_\lambda$ and
$\Hcal^\theta_\lambda=\Hcal^\theta_{\paramq}\otimes_{\C[\paramq]}\C_\lambda$.

The next proposition should be viewed as a quantum analog of Proposition \ref{Prop:tilting_properties1}.

\begin{Prop}\label{Prop:tilting_properties_quantum}
The following claims are true:
\begin{enumerate}
\item We have $R^i\Gamma(\Hcal_\lambda^\theta)=0, R^i\Gamma(\Hcal_\paramq^\theta)=0$ for $i>0$.
In particular, $\Gamma(\Hcal^\theta_\lambda)$ is a filtered deformation of
$H$ and $\Gamma(\Hcal^\theta_\paramq)$ is a filtered deformation of $H_\param$.
\item The functor $R\Gamma(\Ecal^\theta_\param\otimes\bullet)$ is a derived
equivalence $D^b(\Coh(\A^\theta_\paramq))\xrightarrow{\sim} D^b(\Gamma(\Hcal^\theta_\param)\operatorname{-mod})$.
\item The algebra $\Hcal_{\paramq}:=\Gamma(\Hcal_\paramq^\theta)$ is independent of $\theta$.
\item $\Hcal_\lambda:=\Gamma(\Hcal^\theta_\lambda)$ coincides with $\Hcal_{\paramq}\otimes_{\C[\paramq]}\C_\lambda$.
\end{enumerate}
\end{Prop}
\begin{proof}
(1) follows from the corresponding statement on the classical level, see, for example,
(1) of Proposition \ref{Prop:tilting_properties1}. (2) is deduced from the corresponding
quasi-classical statements (e.g. (2) of Proposition \ref{Prop:tilting_properties1})
as explained in \cite{GL}. And (3) is proved in the same way as (3) of
Proposition \ref{Prop:tilting_properties1}. (4) follows from (1).
\end{proof}

\begin{Ex}\label{Ex:SRA_Procesi}
Suppose $Y=(\C^2)^{\oplus n}/\Gamma_n$, where $\Gamma_n=S_n\ltimes \Gamma_1^n$
and $\Gamma_1\subset \operatorname{SL}_2(\C)$ is a finite subgroup. Then
for $E$ we take the Procesi bundle. By \cite[Section 6]{quant_iso},
$\Hcal_{\paramq}$ is the universal symplectic reflection algebra at $t=1$.
\end{Ex}

\subsubsection{Localization}
We can consider the functor $\tilde{\Gamma}^\theta_\lambda: D^b(\A_\lambda^\theta\operatorname{-mod})
\rightarrow D^b(\Hcal_\lambda\operatorname{-mod})$ given by $\Gamma(\Ecal^\theta_\lambda\otimes_{\A_\lambda^\theta}\bullet)$ and its left adjoint
$\tilde{\Loc}^\theta_\lambda:=(\Ecal^\theta_\lambda)^*\otimes_{\Hcal_\lambda}\bullet$.
\begin{defi}\label{defi:local_H}
When these functors
are mutually quasi-inverse equivalences we say that abelian localization holds for $\Hcal_\lambda^\theta$.
\end{defi}

We point out that, by Proposition \ref{Prop:tilting_properties_quantum}, the derived functors  $R\tilde{\Gamma}^\theta_\lambda$ and $\tilde{\Loc}^\theta_\lambda$ are mutually quasi-inverse
equivalences. Equivalently the functors
\begin{equation}\label{eq:H_functors} R\Gamma:D^b(\Coh(\Hcal^\theta_\lambda))\rightleftarrows D^b(\Hcal_\lambda\operatorname{-mod}):
\Hcal_\lambda^\theta\otimes^L_{\Hcal_\lambda}\bullet
\end{equation}
are mutually inverse equivalences.

Note that the obvious equivalence $\Coh(\A_\lambda^\theta)\cong \Coh(\Hcal^\theta_\lambda)$ intertwines the functor
$\tilde{\Gamma}_\lambda^\theta$ with the global section functor for $\Hcal_\lambda^\theta$-modules.

Now we relate localization for $\Hcal_\lambda^\theta$ to that for $\A_\lambda^\theta$.

Note that $\Hcal_{\paramq}$ has a distinguished idempotent
to be denoted by $e$, this is the projector to the direct summand $\A_{\paramq}^\theta$ in
$\Ecal_{\paramq}^\theta$. Its images in the specializations $\Hcal_\lambda^\theta$ will also
be denoted by $e$.

Observe the following functor isomorphism
\begin{equation}\label{eq:localization_connection}
e\tilde{\Gamma}_\lambda^\theta\cong \Gamma^\theta_\lambda.
\end{equation}

We have the following two consequences of (\ref{eq:localization_connection}).

\begin{Lem}\label{Lem:derived_localization_connection}
The following three conditions are equivalent.
\begin{enumerate}
\item Derived equivalence holds for $\A_\lambda^\theta$.
\item $\Hcal_\lambda e \Hcal_\lambda=\Hcal_\lambda$.
\item $\A_\lambda$ has finite homological dimension.
\end{enumerate}
\end{Lem}
\begin{proof}
(1)$\Rightarrow$(3) follows because $\operatorname{Coh}(\A_\lambda^\theta)$
has finite homological dimension, see, e.g., the proof of \cite[Lemma 9.2]{BL}. (3)$\Rightarrow$(2) follows by combining
Lemma \ref{Lem:He_properties} with \cite[Theorem 5.5]{Etingof_affine}.
Note that the condition in (2) just means that the functor $e\bullet:
\Hcal_\lambda\operatorname{-mod}\rightarrow \A_\lambda\operatorname{-mod}$
is an equivalence. Now (2)$\Rightarrow$(1) follows from (\ref{eq:localization_connection}).
\end{proof}

\begin{Cor}\label{Cor:abelian_localization_connection}
The following two conditions are equivalent:
\begin{enumerate}
\item Abelian localization holds for $\A_\lambda^\theta$.
\item Abelian localization holds for $\Hcal_\lambda^\theta$ and
$\Hcal_\lambda e\Hcal_\lambda=\Hcal_\lambda$.
\end{enumerate}
\end{Cor}

\subsection{Enhanced translation bimodules}
\subsubsection{Construction}
In this section we introduce translation bimodules $\Hcal_{\lambda,\chi}$ and $H_{\paramq,\chi}$.

Consider the line bundle $\Str(\chi)$ on $X^\theta$.
We can define the $\Hcal^\theta_{\paramq}$-bimodule $\Hcal^\theta_{\paramq,\chi}$ by
$$\Hcal^\theta_{\paramq,\chi}:=\Ecal^\theta_{\paramq}\otimes_{\A^\theta_{\paramq}}\A_{\paramq,\chi}^\theta
\otimes_{\A_\paramq^\theta}(\Ecal^\theta_\paramq)^*.$$
Note that $e \Hcal^\theta_{\paramq,\chi} e\cong \A^\theta_{\paramq,\chi}$.

We can also consider the specialization $\Hcal^\theta_{\lambda,\chi}$, this is an
$\Hcal_{\lambda+\chi}^\theta$-$\Hcal_\lambda^\theta$-bimodule.


\begin{Lem}\label{Lem:enhanced_translation_independence}
$\Gamma(\Hcal_{\paramq,\chi}^\theta)$ is independent of the
choice of $\theta$.
\end{Lem}
\begin{proof}
Let $Y^0_\param$ denote the locus in $X_\param$ where the morphism $X_\param\rightarrow
Y_\param$ is an isomorphism. Since $\codim_{X_\param}X_\param\setminus Y_\param^0\geqslant 2$,
we see that $\Hcal_{\paramq,\chi}=\Gamma(\Hcal_{\param,\chi}^\theta|_{Y^0_\param})$.
On the other hand, $\Ecal^\theta_{\paramq}|_{Y^0_\param}$ and $\A^\theta_{\paramq,\chi}|_{Y^0_\param}$
are independent of $\theta$. For the former, this follows from the observation that
$\Ecal^\theta_{\paramq}|_{Y^0_\param}=(\Hcal_{\paramq}e)|_{Y^0_\param}$.
For the latter, the proof is similar to \cite[Proposition 3.3]{BL}.
We conclude that $\Gamma(\Hcal_{\param,\chi}^\theta|_{Y^0_\param})$ is
independent of the choice of $\theta$.
\end{proof}

We write $\Hcal_{\paramq,\chi}$ for $\Gamma(\Hcal_{\paramq,\chi}^\theta)$ and
$\Hcal_{\lambda,\chi}:=\Hcal_{\paramq,\chi}\otimes_{\C[\paramq]}\C_\lambda$.
The latter is a $\Hcal_{\lambda+\chi}$-$\Hcal_{\lambda}$-bimodule.

Note that $\Hcal_{\paramq,\chi}$ inherits a filtration from $\Gamma(\Hcal_{\paramq,\chi}^\theta)$.

\subsubsection{$\Hcal_{\lambda,\chi}$ vs $R\Gamma(\Hcal_{\lambda,\chi}^\theta)$}
\begin{Lem}\label{Lem:translation_coincidence}
Suppose the following conditions holds:
\begin{itemize}
\item[(i)]
 $H^i(X^\theta, H^\theta(\chi))=0$ for all $i>0$
\item[(ii)] or
$\tilde{\Gamma}^\theta_{\lambda+\chi}$ is exact.
\end{itemize}
(i) implies $\Hcal_{\paramq,\chi}$ is flat over $\C[\paramq]$. (i) or (ii) imply  $\Hcal_{\lambda,\chi}=R\Gamma(\Hcal^\theta_{\lambda,\chi})$.
\end{Lem}
\begin{proof}
The sheaf $\Hcal^\theta_{\paramq,\chi}$ is flat over $\C[\paramq]$  by the construction.
It follows that  $\Hcal^\theta_{\lambda,\chi}\cong \Hcal^\theta_{\paramq,\chi}\otimes_{\C[\paramq]}\C_\lambda$.

Let us show that (i) implies that $\Hcal_{\paramq,\chi}$ is flat over $\C[\paramq]$.
Note that (i) implies  $H^i(X^\theta_\param, H^\theta_\param(\chi))=0$ for all $i>0$.
This in turn implies that $\gr \Hcal_{\paramq,\chi}=\Gamma(X_{\param},H^\theta(\chi))$.
It suffices to show that  $\Gamma(X_{\param},H^\theta(\chi))$ is flat. But this module is
graded so it is enough to show that $\Gamma(X_{\param},H^\theta(\chi))\otimes^L_{\C[\param]}\C_0$
has no higher homology. For this we observe that $$\Gamma(X_{\param},H^\theta(\chi))\otimes^L_{\C[\param]}\C_0\cong R\Gamma(H^\theta_\param(\chi)\otimes^L_{\C[\param]}\C_0).$$
The left hand side has no positive homology, while the right hand side has no
positive homology. We have proved that
$\Hcal_{\paramq,\chi}$ is flat over $\C[\paramq]$.

Now we show that (i) or (ii) imply $\Hcal_{\lambda,\chi}=R\Gamma(\Hcal_{\lambda,\chi}^\theta)$.

Note that both conditions (i) and (ii) imply that $H^i(X^\theta,\Hcal^\theta_{\lambda,\chi})=0$ for all
$i>0$. So we reduce to showing that $H^1(X^\theta, \Hcal^\theta_{\lambda,\chi})=0$
implies that the natural homomorphism $\Hcal_{\lambda,\chi}\rightarrow \Gamma(\Hcal_{\lambda,\chi}^\theta)$ is an isomorphism. For this we argue exactly
as in \cite[Proposition 6.26]{BPW}. Namely, we choose a line $\ell:=\lambda+\C\theta\subset \paramq$.
Note that $\Hcal_{\ell,\chi}:=\Hcal_{\paramq,\chi}\otimes_{\C[\paramq]}\C[\ell]$
coincides with $\Gamma(\Hcal_{\ell,\chi}^\theta)$ using the argument of
the proof of Lemma \ref{Lem:enhanced_translation_independence}.
So we need to show that the specialization of  $\Gamma(\Hcal_{\ell,\chi}^\theta)$
to $\lambda\in \ell$ coincides with $\Gamma(\Hcal_{\lambda,\chi}^\theta)$.
For this we prove that $H^1(X^\theta,\Hcal_{\ell,\chi}^\theta)$ has finite support in $\ell$
and $\lambda$ is not contained in that support. Both arguments are as in \cite[Proposition 6.26]{BPW}.
%

\end{proof}

\begin{Cor}\label{Cor:derived_equivalence}
Under (i) or (ii), the functor $\Hcal_{\lambda,\chi}\otimes^L_{\Hcal_\lambda}\bullet$
is an equivalence $D^b(\Hcal_\lambda\operatorname{-mod})\rightarrow
D^b(\Hcal_{\lambda+\chi}\operatorname{-mod})$.
\end{Cor}
\begin{proof}
This follows from the observation that
$$R\Gamma(\Hcal^{\theta}_{\lambda,\chi})\otimes^L_{\Hcal_{\lambda}}\bullet\cong
R\tilde{\Gamma}^{\theta}_{\lambda+\chi} \left(\A^{\theta}_{\lambda,\chi}\otimes_{\A^\theta_{\lambda}}
L\tilde{\Loc}^{\theta}_{\lambda}(\bullet)\right).$$
Note that the all functors in the right hand side are derived equivalences.
\end{proof}

\subsubsection{Morphism $\Hcal_{\lambda+\chi,\chi}\otimes^L_{\Hcal_{\lambda+\chi}}\Hcal_{\lambda,\chi}
\rightarrow \Hcal_{\lambda,\chi+\chi'}$}



Recall the sublattice $\underline{\Lambda}\subset \Lambda$ from Corollary
\ref{Cor:sublattice}. Below we only consider $\chi$ in this lattice. With $\chi$ fixed,  for
$\theta$ we take an element satisfying the conclusion of that corollary.


We want to compare the bimodules
$\Hcal_{\lambda,\chi+\chi'}$ and $\Hcal_{\lambda+\chi,\chi'}\otimes^L_{\Hcal_{\lambda+\chi}}\Hcal_{\lambda,\chi}$.
Note that we have a natural homomorphism
\begin{equation}\label{eq:transl_homom_univ}
\Hcal_{\paramq,\chi'}\otimes^L_{\Hcal_{\paramq}}\Hcal_{\paramq,\chi}
\rightarrow \Hcal_{\paramq,\chi+\chi'}.
\end{equation}
It is induced by the isomorphism
\begin{equation}\label{eq:transl_homom_local}
\Hcal^\theta_{\paramq,\chi'}\otimes_{\Hcal^\theta_{\paramq}}\Hcal^\theta_{\paramq,\chi}
\rightarrow \Hcal^\theta_{\paramq,\chi+\chi'}.
\end{equation}
(\ref{eq:transl_homom_univ}) specializes to
\begin{equation}\label{eq:transl_homom_spec}
\Hcal_{\lambda+\chi,\chi'}\otimes^L_{\Hcal_{\lambda+\chi}}\Hcal_{\lambda,\chi}
\rightarrow \Hcal_{\lambda,\chi+\chi'}.
\end{equation}
We remark that the latter homomorphism automatically factorizes through
$\Hcal_{\lambda+\chi,\chi'}\otimes_{\Hcal_{\lambda+\chi}}\Hcal_{\lambda,\chi}$.
The same holds for (\ref{eq:transl_homom_univ}).

Below we will need a characterization of (\ref{eq:transl_homom_spec}).

\begin{Lem}\label{Lem:homom_spec}
Let $\chi,\chi'\in \underline{\Lambda}$. Then
(\ref{eq:transl_homom_spec}) is the unique homomorphism whose microlocalization
to $X^{reg}$ coincides with the microlocalization to $X^{reg}$ of the specialization of
(\ref{eq:transl_homom_local}) to $\lambda$.
\end{Lem}
\begin{proof}
Compare to the proof of \cite[Lemma 2.10]{catO_charp}.
Clearly, (\ref{eq:transl_homom_spec}) has the required property. We need to show that
the property characterizes the homomorphism uniquely. This will follow if we check that
$\Hcal_{\lambda,\chi+\chi'}\hookrightarrow \Gamma(\Hcal_{\lambda,\chi+\chi'}|_{X^{reg}})$.

Let $\theta$ be such that $H^i(X^\theta, H^\theta(\chi+\chi'))=0$
for all $i>0$. By Lemma \ref{Lem:translation_coincidence}, we have that
$\Hcal_{\lambda,\chi+\chi'}\xrightarrow{\sim} R\Gamma(\Hcal_{\lambda,\chi+\chi'}^\theta)$.
And  $\Gamma(\Hcal_{\lambda,\chi+\chi'}^\theta)
\hookrightarrow \Gamma(\Hcal_{\lambda,\chi+\chi'}^\theta|_{X^{reg}})$
follows because $\Hcal_{\lambda,\chi+\chi'}^\theta$ is a quantization
of a vector bundle.
\end{proof}

\subsubsection{Sufficient conditions for (\ref{eq:transl_homom_spec})
being an isomorphism}
In the proof of Theorem \ref{Thm:exactness} we will need three sufficient conditions for (\ref{eq:transl_homom_spec}) to be an isomorphism. We will state and prove two in
this section and the third one later, after we discuss a connection with
localization.

 The first condition is when one of the bimodules in
the left hand side is a Morita equivalence bimodule. We need the following definition.

\begin{defi}\label{defi:strong_inverse}
We say that $\Hcal_{\lambda+\chi,-\chi}$ is a {\it strong inverse} of $\Hcal_{\lambda,\chi}$
if the homomorphisms
$$\Hcal_{\lambda,\chi}\otimes_{\Hcal_{\lambda}}\Hcal_{\lambda+\chi,-\chi}\rightarrow
\Hcal_{\lambda+\chi},
\Hcal_{\lambda+\chi,-\chi}\otimes_{\Hcal_{\lambda+\chi}}\Hcal_{\lambda,\chi}\rightarrow
\Hcal_{\lambda}$$
coming from (\ref{eq:transl_homom_spec}) are isomorphisms.
\end{defi}

\begin{Lem}\label{Lem:iso_equiv}
Suppose that $\Hcal_{\lambda+\chi,\chi'}$ is a Morita equivalence bimodule
with strong inverse $\Hcal_{\lambda+\chi+\chi',-\chi'}$.
Then (\ref{eq:transl_homom_spec}) is an isomorphism. The same holds if
$\Hcal_{\lambda,\chi}$ is a Morita equivalence bimodule with strong inverse
$\Hcal_{\lambda+\chi,-\chi}$.
\end{Lem}
\begin{proof}
Note that if $\Hcal_{\lambda+\chi,\chi'}$ is a Morita equivalence bimodule,
then the left hand side of (\ref{eq:transl_homom_spec}) has no higher homology.
We need to produce an inverse to
\begin{equation}\label{eq:bimod_homom1}
\Hcal_{\lambda+\chi,\chi'}\otimes_{\Hcal_{\lambda+\chi}}\Hcal_{\lambda,\chi}
\rightarrow \Hcal_{\lambda,\chi+\chi'}.
\end{equation}
Consider the homomorphism
$$
\Hcal_{\lambda+\chi+\chi',-\chi'}\otimes_{\Hcal_{\lambda+\chi+\chi'}}\Hcal_{\lambda,\chi+\chi'}
\rightarrow \Hcal_{\lambda,\chi}.
$$
Apply $\Hcal_{\lambda+\chi,\chi'}\otimes_{\Hcal_{\lambda+\chi}}\bullet$ to this homomorphism. Compose the resulting homomorphism
$$
\Hcal_{\lambda+\chi,\chi'}\otimes_{\Hcal_{\lambda+\chi}}\Hcal_{\lambda+\chi+\chi',-\chi'}
\otimes_{\Hcal_{\lambda+\chi+\chi'}}\Hcal_{\lambda,\chi+\chi'}
\rightarrow \Hcal_{\lambda+\chi,\chi'}\otimes_{\Hcal_{\lambda+\chi}}\Hcal_{\lambda,\chi}
$$
with
$$\Hcal_{\lambda,\chi+\chi'}\xrightarrow{\sim} \left(\Hcal_{\lambda+\chi,\chi'}\otimes_{\Hcal_{\lambda+\chi}}\Hcal_{\lambda+\chi+\chi',-\chi'}\right)
\otimes_{\Hcal_{\lambda+\chi+\chi'}}\Hcal_{\lambda,\chi+\chi'}.$$
We get a homomorphism
\begin{equation}\label{eq:bimod_homom2}\Hcal_{\lambda,\chi+\chi'}\rightarrow
\Hcal_{\lambda+\chi,\chi'}\otimes_{\Hcal_{\lambda+\chi}}\Hcal_{\lambda,\chi}.
\end{equation}
By the construction, it is inverse to (\ref{eq:bimod_homom1}) after microlocalization
to $X^{reg}$. In particular, the composed endomorphism of $\Hcal_{\lambda,\chi+\chi'}$
is the identity after microlocalization to $X^{reg}$. As we have seen in the proof of
Lemma \ref{Lem:homom_spec}, $\Hcal_{\lambda,\chi+\chi'}\hookrightarrow
\Gamma(\Hcal_{\lambda,\chi+\chi'}|_{X^{reg}})$. So
(\ref{eq:bimod_homom2}) is the right inverse of (\ref{eq:bimod_homom1}). It
is also the left inverse: we can tensor (\ref{eq:bimod_homom1}) and
(\ref{eq:bimod_homom2}) with the Morita equivalence bimodule  $\Hcal_{\lambda+\chi+\chi',-\chi'}$
and repeat the argument.
\end{proof}

In Section \ref{SSS_enhanced_abel_transl} we will get a sufficient condition for $\Hcal_{\lambda,\chi}$ and
$\Hcal_{\lambda+\chi,-\chi}$ to be strongly inverse to each other.


Here is the second condition that guarantees that (\ref{eq:transl_homom_spec})
is an isomorphism. For this we need the following definition.

\begin{defi}\label{defi:good_line_bundle}
We say that $\chi\in \underline{\Lambda}$ is good for generic $\theta\in \param_{\R}$ if
$H^i(X^\theta,H^\theta(m\chi))=0$ for all $m,i>0$.
\end{defi}

Note that, by Corollary \ref{Cor:sublattice}, for each $\chi\in \underline{\Lambda}$,
there is a generic element $\theta$ such that $\chi$ is good for $\theta$.

\begin{Lem}\label{Lem:one_direction}
Let $\chi,\chi'$ be such that there is $\theta$ for which
both $\chi$ and $\chi'$ are good.
Then (\ref{eq:transl_homom_spec}) is an isomorphism for all $\lambda$.
\end{Lem}
\begin{proof}
Consider the functor $D^b(\Hcal_{\lambda}\operatorname{-mod})\rightarrow
D^b(\Hcal_{\lambda+\chi+\chi'}\operatorname{-mod})$ given by taking the
derived tensor product with $\Hcal_{\lambda+\chi,\chi'}\otimes^L_{\Hcal_{\lambda+\chi}}
\Hcal_{\lambda,\chi}$. Thanks to Lemma \ref{Lem:translation_coincidence}, we get
$$\Hcal_{\lambda,\chi}\otimes^L_{\Hcal_\lambda}\bullet=R\Gamma(\Hcal_{\lambda,\chi}^\theta
\otimes_{\Hcal_\lambda^\theta}\Hcal^\theta_{\lambda}\otimes^L_{\Hcal_\lambda}\bullet).$$
Since the functors in (\ref{eq:H_functors}) -- for $\lambda+\chi$ instead of $\lambda$ --
are mutually inverse equivalences, we see that
\begin{align*}
&\Hcal_{\lambda+\chi,\chi'}\otimes^L_{\Hcal_{\lambda+\chi}}
\Hcal_{\lambda,\chi}\otimes^L_{\Hcal_\lambda}\bullet
\xrightarrow{\sim}
R\Gamma(\Hcal^\theta_{\lambda+\chi,\chi'}
\otimes_{\Hcal^\theta_{\lambda+\chi}}\Hcal^\theta_{\lambda,\chi}\otimes_{\Hcal^\theta_{\lambda}}
\Hcal_\lambda^\theta\otimes^L_{\Hcal_\lambda}(\bullet))\xrightarrow{\sim}\\&
R\Gamma(\Hcal^\theta_{\lambda,\chi+\chi'}
\otimes_{\Hcal^\theta_{\lambda}}
L\Loc_\lambda^\theta(\bullet))\xrightarrow{\sim}\Hcal_{\lambda,\chi+\chi'}\otimes^L_{\Hcal_\lambda}\bullet.
\end{align*}
The last isomorphism holds by Lemma \ref{Lem:translation_coincidence}.
So we get  an isomorphism $\Hcal_{\lambda+\chi,\chi'}\otimes^L_{\Hcal_{\lambda+\chi}}
\Hcal_{\lambda,\chi}\xrightarrow{\sim} \Hcal_{\lambda,\chi+\chi'}$. Its microlocalization
to $X^{reg}$ is as in Lemma \ref{Lem:homom_spec}, so it coincides with (\ref{eq:transl_homom_spec}).
\end{proof}

\begin{Rem}\label{Rem:A_analogs} Note that the direct analogs of Lemmas \ref{Lem:iso_equiv}
and \ref{Lem:one_direction} hold for the sheaves $\A_?^?$ with the same proofs.
\end{Rem}

\subsubsection{Abelian localization vs $\Hcal_{\lambda,\chi}$}\label{SSS_enhanced_abel_transl}
We discuss a relation between the bimodules $\Hcal_{\lambda,\chi}$
and abelian localization for $\Hcal_{\lambda}^\theta$ and $\Hcal_{\lambda+\chi}^\theta$.
Then we will state our third sufficient condition for (\ref{eq:transl_homom_spec})
to be an isomorphism.

\begin{Lem}\label{Lem:abloc_Morita}
Suppose that abelian localization holds for $\Hcal_{\lambda}^\theta$. Then the following three conditions
are equivalent:
\begin{enumerate}
\item Abelian localization holds for $\Hcal_{\lambda+\chi}^\theta$.
\item The bimodule $\Hcal_{\lambda,\chi}$ is a Morita equivalence bimodule
with strong inverse (Definition \ref{defi:strong_inverse}) $\Hcal_{\lambda+\chi,-\chi}$.
\item The bimodule $\Hcal_{\lambda+\chi,-\chi}$ is a Morita
equivalence bimodule.
\end{enumerate}
\end{Lem}
\begin{proof}
Suppose (1) holds. Then, thanks to Lemma \ref{Lem:translation_coincidence},
$$\Hcal_{\lambda,\chi}\cong R\Gamma(\Hcal_{\lambda,\chi}^\theta), \Hcal_{\lambda+\chi,-\chi}\cong R\Gamma(\Hcal_{\lambda+\chi,-\chi}^\theta).$$
Since abelian localization holds for both $\Hcal_\lambda^\theta$ and $\Hcal_{\lambda+\chi}^\theta$, the functors $R\Gamma(\Hcal_{\lambda,\chi}^\theta)\otimes^L_{\Hcal_\lambda}\bullet$ and
$R\Gamma(\Hcal_{\lambda+\chi,-\chi}^\theta)\otimes^L_{\Hcal_{\lambda+\chi}}\bullet$
are mutually inverse t-exact equivalences. The homomorphism
$\Hcal_{\lambda+\chi,-\chi}\otimes_{\Hcal_{\lambda+\chi}}\Hcal_{\lambda,\chi}\rightarrow
\Hcal_{\lambda+\chi}$ coming from an adjunction counit satisfies the uniqueness property
of Lemma \ref{Lem:homom_spec}. The similar claim holds for the other homomorphism in
Definition \ref{defi:strong_inverse}. We see that $\Hcal_{\lambda+\chi,-\chi}$
is the strong inverse of $\Hcal_{\lambda,\chi}$. This is (2).

Tautologically, (2) implies (3).

Now suppose that (3) holds. Consider the equivalence $$R\Gamma^\theta_\lambda(\Hcal^\theta_{\lambda+\chi,-\chi}\otimes_{\Hcal^\theta_{\lambda+\chi}}
L\Loc^\theta_{\lambda+\chi}(\bullet)):D^b(\Hcal_{\lambda+\chi}\operatorname{-mod})
\xrightarrow{\sim} D^b(\Hcal_{\lambda}\operatorname{-mod}).$$
This equivalence coincides with $R\Gamma(\Hcal_{\lambda+\chi,-\chi}^\theta)\otimes^L_{\Hcal_{\lambda+\chi}}\bullet$,
 i.e., thanks to Lemma \ref{Lem:translation_coincidence}, with $\Hcal_{\lambda+\chi,-\chi}\otimes^L_{\Hcal_{\lambda+\chi}}\bullet$.
 In particular, it is t-exact. Since $R\Gamma^\theta_{\lambda}$
 is also t-exact, we see that $L\Loc^\theta_\lambda$ is t-exact. This implies (1).
\end{proof}

We proceed to the third condition that ensures that (\ref{eq:transl_homom_spec}) is an isomorphism.
Suppose that $\lambda,\chi,\chi'$ and $\theta,\theta',\theta''$ are such that

\begin{itemize}
\item[(a)] Abelian localization holds for $\Hcal_{\lambda}^\theta,
\Hcal_{\lambda+\chi}^{\theta'},\Hcal_{\lambda+\chi+\chi'}^{\theta''}$.
\item[(b)] Every classical wall that is relevant for $\lambda$ (see
Definition \ref{defi:essential_hyperplanes}) and separates
$\theta$ and $\theta'$ also separates $\theta$ and $\theta''$.
\end{itemize}

\begin{Lem}\label{Lem:iso_right_order}
Under the assumptions (a),(b) above, (\ref{eq:transl_homom_spec}) holds.
\end{Lem}
Note that this lemma is similar to \cite[Theorem 6.35]{BPW}.
\begin{proof}
Arguing as in the proof of \cite[Theorem 6.35]{BPW}, we see that we can find
$\tilde{\lambda}\in \lambda+\underline{\Lambda},\tilde{\chi},\tilde{\chi}'\in \underline{\Lambda}$ such that
\begin{itemize}
\item[(i)] Abelian localization holds for $$\A_{\tilde{\lambda}}^\theta,
\A_{\tilde{\lambda}+\tilde{\chi}}^{\theta'}, \A_{\tilde{\lambda}+\tilde{\chi}+\tilde{\chi}'}^{\theta''},$$
\item[(ii)] the elements $\tilde{\chi},\tilde{\chi}'$ satisfy the conditions
of Lemma \ref{Lem:one_direction} for $\theta''$.
\end{itemize}
To simplify the notation below we will write $\lambda_1,\lambda_2,\lambda_3$ for
$\lambda,\lambda+\chi,\lambda+\chi+\chi'$, respectively. The notation
$\tilde{\lambda}_1,\tilde{\lambda}_2,\tilde{\lambda}_3$ has the similar meaning.
We will use the notation $\Hcal_{\lambda_1\leftarrow \lambda_2}$ for $\Hcal_{\lambda_2,-\chi}$,
etc.

Thanks to (i) and  Corollary \ref{Cor:abelian_localization_connection}, we see that abelian
localization holds for
$$\Hcal_{\tilde{\lambda}_1}^\theta,
\Hcal_{\tilde{\lambda}_2}^{\theta'}, \Hcal_{\tilde{\lambda}_3}^{\theta''}.$$


It follows from Lemma \ref{Lem:one_direction} that (\ref{eq:transl_homom_spec})
is an isomorphism for  $\tilde{\lambda},\tilde{\chi},\tilde{\chi}'$. Now we can use
Lemma \ref{Lem:iso_equiv} to prove that (\ref{eq:transl_homom_spec})
holds for $\lambda,\chi,\chi'$. More precisely, we have the three isomorphisms
\begin{equation}\label{eq:triple_composition}
\Hcal_{\lambda_j\leftarrow \tilde{\lambda}_j}\otimes^L_{\Hcal_{\tilde{\lambda}_j}}\Hcal_{\tilde{\lambda}_j\leftarrow\tilde{\lambda}_i}
\otimes^L_{\Hcal_{\tilde{\lambda}_i}}\Hcal_{\tilde{\lambda}_\leftarrow \lambda}
\xrightarrow{\sim}\Hcal_{\lambda_j\leftarrow \lambda_i}
\end{equation}
for $1\leqslant i<j\leqslant 3$
because the first and the third factors in the left hand side are Morita equivalence
bimodules with required strong inverses. Take the derived tensor product of the isomorphisms
for $(i,j)=(1,2)$ and $(i,j)=(2,3)$. Now we use the claims that
\begin{itemize}
\item (\ref{eq:transl_homom_spec})
is an isomorphism for $\tilde{\lambda},\tilde{\chi},\tilde{\chi}'$,
\item  and that
$\Hcal_{\lambda_j\leftarrow \tilde{\lambda}_j}$ and
$\Hcal_{\tilde{\lambda}_j\leftarrow \lambda_j}$ are strong inverses
\end{itemize}
to deduce that   (\ref{eq:transl_homom_spec})
holds for $\lambda,\chi,\chi'$.
%
%
\end{proof}

\begin{Rem}\label{Rem:usual_bimod_another}
Note that an analog of Lemma \ref{Lem:iso_right_order}  for the usual translation bimodules $\A^?_{?}$ with the same proof
(this is essentially a part of the proof of \cite[Theorem 6.35]{BPW}). An analog
of Lemma \ref{Lem:abloc_Morita} also holds for the sheaves $\A_?^?$ -- with the same proof.
\end{Rem}

\subsection{Main result}
\subsubsection{Setting}
Fix $\lambda^\circ\in \paramq$. We are interested in the locus $\lambda^\circ+\underline{\Lambda}$.
Let $\Upsilon_j, j\in J,$ be all classical walls that are relevant for $\lambda^\circ$
(in the sense of Definition \ref{defi:essential_hyperplanes}).
They split $\param_{\R}$ into (closed) chambers. A choice of a (closed) chamber, say $C$, gives rise
to vectors $\alpha_j\in \param_{\mathbb{R}}^*$ such that $\alpha_j|_{C}\geqslant 0$
and $\ker\alpha_j=\Upsilon_i$. Each vector $\alpha_j$ is defined up to a positive multiple.

Each chamber $C_j$ gives rise to a locus in $\lambda^\circ+\underline{\Lambda}$ defined as follows.
Let $m_j$ denote the maximum of $\alpha_j$ on the singular hyperplanes parallel to $\Upsilon_i$
and intersecting $\lambda^\circ+\param_{\Z}$.
Then we consider the locus
\begin{equation}\label{eq:quantum_chamber}
\tilde{C}:=\{\lambda\in \lambda^\circ+\underline{\Lambda}| \langle \alpha_j,\lambda\rangle> m_i\}.
\end{equation}
to be called the {\it quantum chamber} associated with $C$.
Note that the union of all possible quantum chambers is precisely the complement to
the union of all singular hyperplanes in $\lambda^\circ+\underline{\Lambda}$. Note also
that $\tilde{C}+(C\cap \underline{\Lambda})\subset \tilde{C}$.

Now we restate Theorem \ref{Thm:exactness}.

\begin{Thm}\label{Thm:exactness_precise}
Suppose that $\theta$ is a generic element in $C$. Then for all $\lambda\in \tilde{C}$ abelian localization
holds for $\Hcal^\theta_\lambda$. In particular, the functor
$\Gamma^\theta_\lambda:\Coh(\A^\theta_\lambda)\rightarrow \A_\lambda\operatorname{-mod}$
is exact.
\end{Thm}

\subsubsection{Proof}
We introduce some notation. We select chambers $C_0,\ldots,C_m$ in the following way:
\begin{itemize}
\item
$C_m=C, C_0=-C$,
\item The chambers $C_i,C_{i+1}$ have a common codimension 1 face, to be denoted
by $C_{i,i+1}$ for $i=0,\ldots,m-1$. Let $\Upsilon_i$ denote the classical wall
spanned by $C_{i,i+1}$.
\item $C_{j_1}$ and $C_{j_2}$ with $j_1<j_2$ lie on different sides of $\Upsilon_i$
if and only if $j_1\leqslant i$ and $j_2\geqslant i+1$.
\end{itemize}

Here is a picture (with $m=4$).

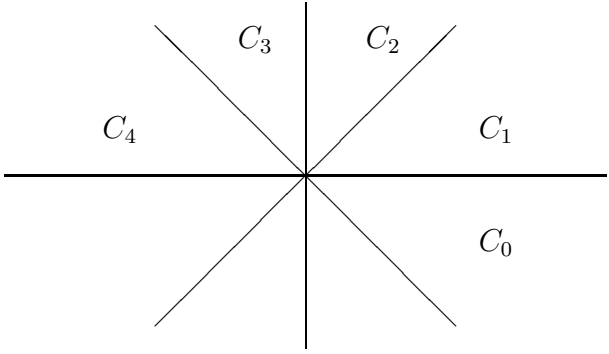
\begin{picture}(100,50)
\put(2,25){\line(1,0){80}}
\put(22,5){\line(1,1){40}}
\put(42,2){\line(0,1){46}}
\put(62,5){\line(-1,1){40}}
\put(15,30){$C_4$}
\put(33,42){$C_3$}
\put(50,42){$C_2$}
\put(65,30){$C_1$}
\put(65,15){$C_0$}
\end{picture}

We pick generic elements $\theta_i\in C_i$. We assume that $\theta_m=\theta$.


\begin{proof}[Proof of Theorem \ref{Thm:exactness_precise}]
The proof is in several steps.

{\it Step 1}.
Pick $\lambda_i\in \tilde{C}_i\cap (\lambda^\circ+\underline{\Lambda})$ such that abelian localization
holds for $\Hcal^{\theta_i}_{\lambda_i}$. We are going to prove the following
series of claims for $i=1,\ldots,m$:
\begin{itemize}
\item[($T_i$)]
The homomorphism in \ref{eq:transl_homom_spec}), $\Hcal_{\lambda,\lambda_i-\lambda}\otimes^L_{\Hcal_\lambda}\Hcal_{\lambda_m,\lambda-\lambda_m}
    \rightarrow\Hcal_{\lambda_m,\lambda_i-\lambda_m}$ is an isomorphism.
\end{itemize}

{\it Step 2}.
There are two important things to observe about the claims ($T_i$).

First, note that $(T_i)$ is independent of the choice of $\lambda_i$ (and $\lambda_m$)
such that abelian localization holds for $\Hcal_{\lambda_i}^{\theta_i}$ (and $\Hcal_{\lambda_m}^{\theta_m}$) thanks to Lemmas \ref{Lem:iso_equiv} and
\ref{Lem:abloc_Morita}, compare to the proof of Lemma \ref{Lem:iso_right_order}.

{\it Step 3}.
Second, we claim that ($T_m$)
is equivalent to the claim that abelian localization holds for $\Hcal_{\lambda}^{\theta}$.
Condition ($T_m$) reads that
\begin{equation}\label{eq:tensor_product_special}
\Hcal_{\lambda,\lambda_m-\lambda}\otimes^L_{\Hcal_{\lambda}}\Hcal_{\lambda_m,\lambda-\lambda_m}
\xrightarrow{\sim} \Hcal_{\lambda_m}.\end{equation}
We claim that the functors $\Hcal_{\lambda_m,\lambda-\lambda_m}\otimes^L_{\Hcal_{\lambda_m}}\bullet:
D^b(\Hcal_{\lambda_m})\rightarrow D^b(\Hcal_{\lambda})$ and
$\Hcal_{\lambda,\lambda_m-\lambda}\otimes^L_{\Hcal_{\lambda}}\bullet:
D^b(\Hcal_{\lambda_m})\rightarrow D^b(\Hcal_{\lambda})$ are equivalences.
Let $\theta'$ be such  that $\lambda-\lambda_m$ is good for $\theta'$
(Definition \ref{defi:good_line_bundle}). Then,
by Lemma \ref{Lem:translation_coincidence}, we have
$$\Hcal_{\lambda_m,\lambda-\lambda_m}
\xrightarrow{\sim} R\Gamma(\Hcal^{\theta'}_{\lambda_m,\lambda-\lambda_m}).$$
Recall (the proof of Corollary \ref{Cor:derived_equivalence}) that
$$R\Gamma(\Hcal^{\theta'}_{\lambda_m,\lambda-\lambda_m})\otimes^L_{\Hcal_{\lambda_m}}\bullet\cong
R\tilde{\Gamma}^{\theta'}_{\lambda_m} \left(\A^{\theta'}_{\lambda_m,\lambda-\lambda_m}\otimes_{\A^{\theta'}_{\lambda_m}}
L\tilde{\Loc}^{\theta'}_{\lambda}(\bullet)\right).$$
It follows that $\Hcal_{\lambda_m,\lambda-\lambda_m}\otimes^L_{\Hcal_{\lambda_m}}\bullet$
is indeed an equivalence. Similarly, the functor $\Hcal_{\lambda,\lambda_m-\lambda}\otimes^L_{\Hcal_{\lambda}}\bullet$
is an equivalence.

Note that these two equivalences are right t-exact. Their composition is a t-exact
equivalence (the identity). So both $\Hcal_{\lambda_m,\lambda-\lambda_m}\otimes^L_{\Hcal_{\lambda_m}}\bullet$
and $\Hcal_{\lambda,\lambda_m-\lambda}\otimes^L_{\Hcal_{\lambda}}\bullet$
are t-exact. Abelian localization holds for $\Hcal_{\lambda_m}^\theta$
by the choice of $\lambda_m$. We use Lemma \ref{Lem:abloc_Morita}
to conclude that abelian localization holds for $\Hcal_{\lambda}^\theta$.

{\it Step 4}. We will prove ($T_i$) by ascending induction on $i$ starting
with $i=1$. To start with, we claim that the following claim holds:

\begin{itemize}
\item[($\natural$)] Let $\tilde{C}'$ be a
quantum chamber in $\lambda^\circ+\underline{\Lambda}$ and $C'$ be the corresponding classical chamber.
Let $\tilde{\Upsilon}$ be an essential hyperplane that is parallel to a wall
of $C'$ and intersects $\tilde{C}'$. Let $\theta'$ be a generic element
of $C'$. Then, for a Zariski generic  element $\lambda'\in \tilde{\Upsilon}\cap
\tilde{C'}$, abelian localization holds for $\A_{\lambda'}^{\theta'}$.
\end{itemize}

Namely, by Lemma \ref{Cor:abel_loc}, there is $\chi'\in \underline{\Lambda}$
such that abelian localization holds for $\A_{\lambda''}^{\theta'}$ for
$\lambda''\in \tilde{C}'+\chi'$.
Consider the locus $\tilde{\Upsilon}^0$ in $\tilde{\Upsilon}$ where the homomorphisms
$$\A_{\tilde{\Upsilon},\chi'}\otimes_{\A_{\tilde{\Upsilon}}}
\A_{\tilde{\Upsilon}+\chi',-\chi'}\rightarrow \A_{\tilde{\Upsilon}+\chi'},
\A_{\tilde{\Upsilon}+\chi',-\chi'}\otimes_{\A_{\tilde{\Upsilon}+\chi'}}
\A_{\tilde{\Upsilon},\chi'}\rightarrow \A_{\tilde{\Upsilon}}$$
are isomorphisms. It is Zariski open. By Lemma \ref{Lem:ab_loc_Weil_generic},
abelian localization holds for $\A_{\lambda'}^{\theta'}$ when $\lambda'$
is Weil generic. Using the direct analog of Lemma \ref{Lem:abloc_Morita}
for the sheaves $\A_?^?$ (see Remark \ref{Rem:usual_bimod_another}) we see that
$\tilde{\Upsilon}^0$ contains a Weil generic element in $\tilde{\Upsilon}$ hence it is
nonempty. Pick  $\lambda'\in \tilde{\Upsilon}^0\cap \tilde{C}'$.
Applying the analog of Lemma \ref{Lem:abloc_Morita} again we see that
abelian localization holds for $\A_{\lambda'}^{\theta'}$. This proves ($\natural$).


Note that, thanks to Corollary \ref{Cor:abelian_localization_connection}, ($\natural$)
implies that abelian localization holds for $\Hcal_{\lambda'}^{\theta'}$ with
Zariski generic $\lambda'\in \tilde{C}'\cap \tilde{\Upsilon}$.


{\it Step 5}. We prove ($T_1$).  Note that we can choose $\lambda_m,\lambda_1$ and $\theta_0$ such that  $\lambda_1\in \lambda+C_{0,1}, \lambda_m\in \lambda-C_0$
and  both $\lambda_1-\lambda,\lambda-\lambda_m$ are good for $\theta_0$, see the picture:

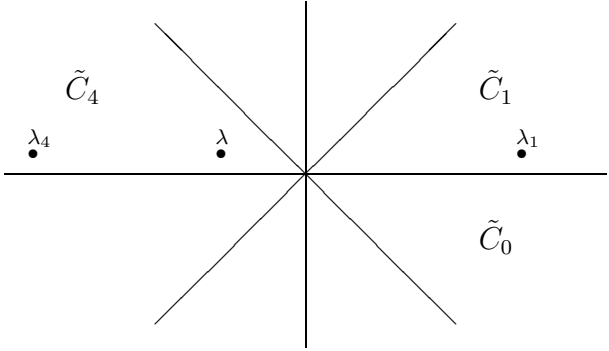
\begin{picture}(100,50)
\put(2,25){\line(1,0){80}}
\put(22,5){\line(1,1){40}}
\put(42,2){\line(0,1){46}}
\put(62,5){\line(-1,1){40}}
\put(10,35){$\tilde{C}_4$}
\put(65,35){$\tilde{C}_1$}
\put(65,15){$\tilde{C}_0$}
\put(5,27){\tiny $\bullet$}
\put(5,29){\tiny $\lambda_4$}
\put(30,27){\tiny $\bullet$}
\put(30,29){\tiny $\lambda$}
\put(70,27){\tiny $\bullet$}
\put(70,29){\tiny $\lambda_1$}
\end{picture}

Thanks to the conclusion of Step 4, abelian localization holds
for $\Hcal^{\theta_1}_{\lambda_1},\Hcal^{\theta_m}_{\lambda_m}$ as long as
$\lambda_1$ is deep enough inside $\lambda+C_{0,1}$ and $\lambda_m$
is deep enough in $\lambda-C_{0,1}$.
Now ($T_1$) follows from Lemma \ref{Lem:one_direction}.

{\it Step 6}. Now we prove that ($T_i$)$\Rightarrow$($T_{i+1}$) completing the
proof of the theorem. We can modify $\lambda_i,\lambda_{i+1}$ so that
\begin{itemize}
\item[(I)] $\lambda_{i+1}-\lambda$ and $\lambda_i-\lambda_{i+1}$
are good for $\theta_i$,
\item[(II)] abelian localization
holds for $\Hcal_{\lambda_i}^{\theta_i}, \Hcal_{\lambda_{i+1}}^{\theta_{i+1}}$.
\end{itemize}
Namely, for a Zariski generic element
$\chi_{i+1}\in \underline{\Lambda}$ in  $C_{i,i+1}$
we set $\lambda_{i+1}:=\lambda+\chi_{i+1}\in \tilde{C}_{i+1}$. Thanks to Step 4, abelian localization
holds for $\Hcal_{\lambda_{i+1}}^{\theta_{i+1}}$.
Next for a Zariski generic element $\chi_i\in \underline{\Lambda}\cap C_i$ we set
$\lambda_i:=\lambda_{i+1}+\chi_i\in \tilde{C}_i$. We also have that
abelian localization holds for $\Hcal_{\lambda_{i+1}}^{\theta_{i+1}}$.
So (II) hold. We have $\chi_{i+1}=\lambda_{i+1}-\lambda,\chi_i=\lambda_i-\lambda_{i+1}\in C_i$.
So (I) holds by Corollary \ref{Cor:sublattice}.

Thanks to (I),  from
Lemma \ref{Lem:one_direction}, it follows that
\begin{equation}\label{eq:tens_prod1}
\Hcal_{\lambda_{i+1},\lambda_i-\lambda_{i+1}}\otimes^L_{\Hcal_{\lambda_{i+1}}}
\Hcal_{\lambda, \lambda_{i+1}-\lambda}\xrightarrow{\sim}\Hcal_{\lambda,\lambda_i-\lambda}.
\end{equation}
Hence
\begin{equation}\label{eq:tens_prod2}
\begin{split}
&\Hcal_{\lambda_{i+1},\lambda_i-\lambda_{i+1}}\otimes^L_{\Hcal_{\lambda_{i+1}}}
\Hcal_{\lambda, \lambda_{i+1}-\lambda}\otimes^L_{\Hcal_{\lambda}}
\Hcal_{\lambda_m,\lambda-\lambda_m}\xrightarrow{\sim}
\Hcal_{\lambda,\lambda_i-\lambda}\otimes^L_{\Hcal_\lambda}
\Hcal_{\lambda_m,\lambda-\lambda_m}\\
&\xrightarrow{\sim}
\Hcal_{\lambda_m,\lambda_i-\lambda_m}.
\end{split}
\end{equation}
The second isomorphism follows from ($T_{i}$).
On the other hand, by Lemma \ref{Lem:iso_right_order} (and (II)), we have
\begin{equation}\label{eq:tens_prod3}
\Hcal_{\lambda_{i+1},\lambda_i-\lambda_{i+1}}\otimes^L_{\Hcal_{\lambda_{i+1}}}
\Hcal_{\lambda_m, \lambda_{i+1}-\lambda_m}\xrightarrow{\sim}\Hcal_{\lambda_m,\lambda_i-\lambda}.
\end{equation}
Combining (\ref{eq:tens_prod2}),(\ref{eq:tens_prod3}) we deduce that
\begin{equation}\label{eq:tens_prod4}
\Hcal_{\lambda_{i+1},\lambda_i-\lambda_{i+1}}\otimes^L_{\Hcal_{\lambda_i}}
\Hcal_{\lambda, \lambda_{i+1}-\lambda}\otimes^L_{\Hcal_{\lambda}}
\Hcal_{\lambda_m,\lambda-\lambda_m}\xrightarrow{\sim}
\Hcal_{\lambda_{i+1},\lambda_i-\lambda_{i+1}}\otimes^L_{\Hcal_{\lambda_i}}
\Hcal_{\lambda_m, \lambda_{i+1}-\lambda_m}.
\end{equation}
But $\Hcal_{\lambda_{i+1},\lambda_i-\lambda_{i+1}}$ is a derived Morita
equivalence bimodule by Corollary \ref{Cor:derived_equivalence}. So we can cancel it out in
(\ref{eq:tens_prod4}). We arrive at ($T_{i+1}$). This finishes the induction step and hence
the proof of the theorem.
\end{proof}
\subsection{Remarks}
\subsubsection{Equivalence of Conjectures \ref{Conj:der_loc} and \ref{Conj:abelian}}
\label{SSS_conj_equiv}
Clearly, Conjecture \ref{Conj:abelian} implies Conjecture \ref{Conj:der_loc}.

Let $\lambda,\theta$ be as in Theorem \ref{Thm:exactness_precise} but we make no assumptions
on the tilting generator.
Assume that Conjecture \ref{Conj:fin_hom_dim} holds (or, a formally weaker assumption, that if
derived localization holds for $\A_\lambda^\theta$ with some $\theta$ then it
holds for all generic $\theta$). Then using the argument of the proof of
Theorem \ref{Thm:exactness_precise} (and replacing various technical lemmas
there with their obvious analogs for the sheaves $\A^?_?$ that have the
same proof, compare with Remark \ref{Rem:A_analogs}), one can prove the following
result.

\begin{Prop}\label{Prop:A_analog} Under the assumption of the previous paragraph,
if derived localization holds for
$\A_\lambda^\theta$ and $\lambda\in \tilde{C}$, then abelian localization holds for $\A_\lambda^\theta$.
\end{Prop}

\subsubsection{The locus where abelian localization holds for $\Hcal^\theta_\lambda$}\label{SSS_enhanced_localization}
We expect that one can strengthen Theorem \ref{Thm:exactness_precise} to cover all
parameters in $\paramq$, not just those lying outside of singular hyperplanes.
For every classical wall $\Upsilon$ relevant for $\lambda\in \paramq$ we pick a  shift
$\hat{\Upsilon}$ so that $\hat{\Upsilon}\cap (\lambda^\circ+\param_{\Z})=\varnothing$. Note that the collection of $\hat{\Upsilon}$ splits $\lambda+\param_{\Z}$
into the disjoint union of  chambers that are shifts of classical chambers for $\lambda$. For the
classical chamber $C$, let $\tilde{C}$ denote the corresponding shifted chamber.

\begin{Conj}\label{Conj:upgraded_localization}
With a suitable choice of the shifts $\hat{\Upsilon}$ that depends on the choice of $E$
the following holds:
for every generic $\theta\in C$ and $\lambda\in \tilde{C}$, abelian localization holds for
$\Hcal^\theta_\lambda$.
\end{Conj}

As a consequence, for each $\lambda\in \paramq$, there is $\theta$ such that $\Gamma^\theta_\lambda$
is exact.

To prove this claim one can use an argument of the proof of Theorem \ref{Thm:exactness_precise}
once one knows the claim is true for $\lambda$ Weil generic in an essential hyperplane. It should not
be difficult to reduce the latter claim to the case when $\dim \param=1$.

In the special case when $X=\M^\theta_0(n\delta,\epsilon_0)$ for an affine type A quiver
(i.e., $X$ is a resolution of $(\C^2)^n/\Gamma_n$ for cyclic $\Gamma$) Conjecture
\ref{Conj:upgraded_localization} should follow from the main result of
\cite{cycl_ab_loc} and some easy combinatorics.

A more subtle question is how to determine the hyperplanes $\hat{\Upsilon}$. They depend on
the choice of the tilting bundles $E^\theta$ and to determine them one should have some
explicit information on the algebras $\Hcal_\lambda$ that is only available in very few
cases (resolutions of $(\C^2)^n/\Gamma_n$ for arbitrary $\Gamma$, where we get symplectic reflection
algebras, and, perhaps, hypertoric varieties, where tilting bundles are direct sums of line bundles).
One could also expect that an explicit description of the algebras $\Hcal_\lambda$
should be available for smooth Coulomb branches.

We also would like to point out that Conjecture \ref{Conj:upgraded_localization} does not require
that ($\heartsuit$) holds. We expect that it holds even without ($\heartsuit$).

\section{$\mathcal{O}$-regular parameters}
Let $T$ denote a Hamiltonian torus acting on $X$.
Throughout the section we assume that $X^T$ is finite. So it makes sense to speak about
the category $\mathcal{O}$, Section \ref{SS_cat_O}. The category $\mathcal{O}$ for $\A_\lambda$
associated to a generic one-parameter subgroup $\nu:\C^\times\rightarrow T$ will be denoted
by $\Ocat_\nu(\A_\lambda)$. Recall that $\Ca_\nu(\A_\lambda),\Ca_{\nu}(\A_\lambda^\theta)$
stand for the Cartan subquotients of $\A_\lambda$ and $\A_\lambda^\theta$,
see Section \ref{SSS_Cartan_subquotient}.
The goal of this section is to understand the conditions on $\lambda$ for
the natural homomorphism $\Ca_\nu(\A_\lambda)\rightarrow \Ca_\nu(\A_\lambda^\theta)$
to be an isomorphism.

\subsection{Deformed category $\mathcal{O}$}
Consider the completion $\ring:=\C[\paramq]^{\wedge_\lambda}$.
In this section we will discuss a deformation of the category $\Ocat_\nu(\A_\lambda)$
over $\operatorname{Spec}(\ring)$.

\subsubsection{Construction}
We can consider the $\ring$-algebra
$\A_{\ring}:=\A_{\paramq}\otimes_{\C[\paramq]}\ring$.

\begin{Lem}\label{Lem:partial_completion_Noetherian}
$\A_{\ring}$ is a Noetherian algebra.
\end{Lem}
\begin{proof}
We have a $\Z_{\geqslant 0}$-filtration on $\A_{\paramq}$ with $\paramq^*$ in
degree $0$ such that $\gr\A_{\paramq}=\C[Y]\otimes \C[\paramq]$, see the proof
of \cite[Lemma 3.5]{BL}. This induces a $\Z_{\geqslant 0}$-filtration on $\A_{\ring}$
whose associated graded is $\C[Y]\otimes \ring$, a Noetherian algebra. The claim of the lemma
follows.
\end{proof}

We consider the category $\Ocat_\nu(\A_{\ring})$ that consists
of all finitely generated $\A_\ring$-modules   where $\A_\ring^{>0}$
acts locally nilpotently. For example, the universal Verma module
$$\Delta_\ring(\Ca_\nu(\A_\ring)):=\A_\ring/\A_\ring\A_{\ring}^{>0}$$
is an object in $\Ocat_\nu(\A_\ring)$.

First, we are going to establish a weight decomposition for a module in
$\Ocat_\nu(\A_\ring)$. Note that we still have the grading element $h\in \A_{\ring}$,
compare to Section \ref{SSS_O_basics}.

\begin{Lem}\label{Lem:deformed_weights}
There are  elements $F_1,\ldots,F_k\in \ring[t]$ such that
the image of every $F_i$ in $\C[t]$ have only one root
and, for every $M\in \Ocat_\nu(\A_\ring)$ and every $m\in M$,
the element $m$ is annihilated by the product of some elements of the form
$(F_i(h)-j)$ for $j\in \Z_{\geqslant 0}$.
\end{Lem}
\begin{proof}
Every object in $\Ocat_{\nu}(\A_{\ring})$ admits a nonzero homomorphism from
$\Delta_\ring(\Ca_\nu(\A_{\ring}))$.
Since  $\A_\ring$ is Noetherian, the object $M$ is filtered by quotients of $\Delta_\ring(\Ca_\nu(\A_{\ring}))$. So it is enough
to prove the lemma for $\Delta_\ring(\Ca_\nu(\A_{\ring}))$.
Since the adjoint action of $h$ on $\A_\ring/\A_\ring\A_\ring^{>0}$ has eigenvalues in
$\Z_{\leqslant 0}$,  we need to show that
$\Ca_\nu(\A_\ring)$ is annihilated by the product of elements
of the form $F_i(h)$. This algebra is a finitely generated module over $\ring$. The claim of the lemma easily follows from the Hensel lemma.
\end{proof}

For $M\in \Ocat_\nu(\A_\ring)$ and $\alpha\in \C$ we set
$M^\alpha$ for the subset of all $m$ that is annihilated by $\prod_{i=1}^k
(F_{i}(h)-j_i)$ with $j_i\in \Z$ and  $\alpha+j_i$ being the only root of the image of
$F_i(t)$ in $\C[t]$.
Clearly, $M^\alpha$ is an $\ring$-submodule of $M$.

\begin{Cor}\label{Cor:deformed_wt_decomp}
We have $M=\bigoplus_\alpha M^\alpha$.
\end{Cor}

In particular, any object in $\Ocat_\nu(\A_\ring)$
is gradable in a functorial way.

Now we proceed to  the main result of this section. By a {\it finite} $\ring$-algebra
we mean an associated unital $\ring$-algebra that is  finitely generated as a module over $\ring$.

\begin{Prop}\label{Prop:deformed_category_O}
The category $\Ocat_\nu(\A_\ring)$ is equivalent to
the category of modules over a finite $\ring$-algebra.
\end{Prop}
\begin{proof}
We just need to show that
\begin{enumerate}
\item The Hom modules between two objects in $\Ocat_\nu(\A_\ring)$
are finitely generated over $\ring$,
\item
and there is a projective object $P$ in $\Ocat_\nu(\A_\ring)$
such that every other object is a quotient of the direct sum of finitely many copies of $P$.
\end{enumerate}
(1) is easy and is left as an exercise for the reader.

We prove (2). We follow a similar argument in \cite[Section 2.4]{GGOR}. Let $\Ocat_\nu(\A_\lambda)^{gr}, \Ocat_\nu(\A_\ring)^{gr}$ denote the category of
graded objects in $\Ocat_\nu(\A_\lambda), \Ocat_\nu(\A_\ring)$. Note that every projective in
$\Ocat_\nu(\A_\lambda)^{gr}, \Ocat_\nu(\A_\ring)^{gr}$ is projective also in
$\Ocat_\nu(\A_\lambda), \Ocat_\nu(\A_\ring)$.

Every object $\bar{M}\in\Ocat_\nu(\A_\lambda)^{gr}$ decomposes as $\bigoplus_{z\in \C}\bar{M}[z]$,
where $\bar{M}[z]=\bigoplus_{i\in\Z}\bar{M}_{i}^{i+z}$. Note that $\bar{M}[z]$ is an $\A_\lambda$-submodule.
Similarly, every object
$\bar{M}_\ring\in\Ocat_\nu(\A_\lambda)^{gr}$
decomposes as the direct sum of $\A_\ring$-submodules $\bigoplus_{z\in \C}\bar{M}_\ring[z]$,
where $\bar{M}_\ring[z]=\bigoplus_{i\in\Z}\bar{M}_{\ring,i}^{i+z}$.

We can find a finite collection $(z_i,n_i)\in \C\times \Z_{>0}$
such that $M^\alpha\neq 0$ implies that $\alpha\in z_i+\Z$ for some $i$ and
\begin{itemize}
\item[(i)]
$(\A_\lambda/ \A_\lambda \A_{\lambda}^{>n_i})[z_i]\in \Ocat_\nu(\A_\lambda)^{gr}$ (with its grading inherited from
$\A_\lambda$) is projective that represents
the functor $\bar{M}\mapsto \bar{M}[z_i]^{z_i}$,
\item[(ii)] and
every simple object in $\Ocat_\nu(\A_\lambda)^{gr}$ is a quotient of a graded shift of one of these projective
objects.
\end{itemize}
 From (i) we deduce that, for every $\bar{M}_\ring\in \Ocat_\nu(\A_\ring)^{gr}$, the direct summand
 $\bar{M}_\ring[z_i]^{z_i}$
 is annihilated by $\A_\ring^{>n_i}$. It follows that
$(\A_\ring/ \A_\ring \A_\ring^{>n_i})[z_i]$ represents the functor
$\bar{M}\mapsto \bar{M}_\ring[z_i]^{z_i}$.
So $(\A_\ring/ \A_\ring \A_\ring^{>n_i})[z_i]$
is projective. Similarly, (ii) implies that  every object in
$\Ocat_\nu(\A_\ring)^{gr}$ is the quotient of a finite direct sum
of the projectives that are shifts of $(\A_\ring/ \A_\ring \A_\ring^{>m_i})[z_i]$.
This implies (2).
\end{proof}

\subsubsection{Deformed highest weight structure}\label{SSS_hw_deformed}

We can also consider the completion $\A_{\paramq}^{\wedge_\lambda}$ of $\A_{\paramq}$
with respect to the two-sided ideal generated by the maximal ideal of $\lambda$
in $\C[\paramq]$. We have an algebra homomorphism $\A_\ring
\rightarrow \A_{\paramq}^{\wedge_\lambda}$.

Now suppose that $\lambda$ is in the intersection of the loci described in
Lemmas \ref{Lem:Verma_flatness}, \ref{Lem:Tor_locus}
Corollary \ref{Cor:full_embedding}, i.e.,
\begin{itemize}
\item[(i)]
$\Ca_\nu(\A_\lambda)\xrightarrow{\sim}\Ca_\nu(\A_\lambda^\theta)$,
and $\operatorname{Tor}_i^{\A_\lambda}(\Delta_{\lambda}^{opp}(p), \Delta_{\lambda}(p'))=0$
for all $p,p'\in X^T$ and $i>0$. In particular, $\mathcal{O}_\nu(\A_\lambda)$ is
a highest weight category with standard objects $\Delta_\lambda(p), p\in X^T$.
\item[(ii)] Abelian localization holds for $\A_\lambda^\theta$, and hence
the natural functor $D^b(\mathcal{O}_\nu(\A_\lambda))\rightarrow
D^b_{\mathcal{O}}(\A_\lambda)$ is an equivalence.
\item[(iii)] For all $p\in X^T$, the $\C[\param]$-module $\Delta_{\paramq}(p)$ is flat in a Zariski neighborhood of
$\lambda$.
\end{itemize}

Thanks to (ii), every object $M$ in $\Ocat_\nu(\A_\lambda)$ that has no higher self-extensions
in this category has no higher self-extensions in $\A_\lambda\operatorname{-mod}$
either. So such an object admits a unique deformation to an $\ring$-flat $\A_{\paramq}^{\wedge_\lambda}$-module, to be denoted by
$\hat{M}$. Note that every object in $\Ocat_\nu(\A_\lambda)$ admits a weakly $\nu(\C^\times)$-equivariant
structure. So $\hat{M}$ also has one. We can consider the $\nu(\C^\times)$-finite part
of $\hat{M}$ to be denoted by $\hat{M}_{fin}$.

\begin{Lem}\label{Lem:deformed_object}
$\hat{M}_{fin}$ is an object in $\Ocat_\nu(\A_\ring)$
that is flat over $\ring$.
\end{Lem}
\begin{proof}
Consider the quotient $\C[\paramq]/\mathfrak{m}^k$,
where $\mathfrak{m}$ is the maximal ideal of $\lambda$. Let $\hat{M}_k$
denote the unique flat deformation of $M$ over $\A_{\paramq}/(\mathfrak{m}^k)$
so that $\hat{M}=\varprojlim_k \hat{M}_k$. Let $M^i$ denote the degree $i$
graded component of $M$. The action of $\nu(\C^\times)$ on each $\hat{M}_k$
is rational. Note that $M^i_{k+1}\twoheadrightarrow M^i_k$.
So each $\hat{M}_k^i$ is a flat deformation of $M^i$ over
$\C[\paramq]/\mathfrak{m}^k$. We have $\hat{M}_{fin}=\bigoplus_i \hat{M}^i$, where
$\hat{M}^i=\varprojlim_k \hat{M}^i_k$. In particular, $\hat{M}_{fin}$
is flat over $\ring$.

Now we need to show that $\hat{M}_{fin}$ lies in $\Ocat_\nu(\A_\ring)$.
The weakly  $\nu(\C^\times)$-equivariant structure is manifest from the construction.
The submodule $\A_\ring^{>0}$ acts locally nilpotently  because the degrees are bounded from the above. It remains to show that $\hat{M}_{fin}$ is finitely generated.
All graded components of $\hat{M}_{fin}$ are finitely generated $\ring$-modules.
We can find $j$ such that the finite dimensional subspace $\bigoplus_{i>j}M^i$ generates
$M$. It is easy to see that the finitely generated $\ring$-module $\bigoplus_{i>j}\hat{M}^i$ generates $\hat{M}_{fin}$.
\end{proof}

\begin{Lem}\label{Lem:deformed_hw}
\begin{enumerate}
\item If $M=\Delta_\lambda(p)$, then $\hat{M}_{fin}=\Delta_\ring(p)$.
\item If $M$ is projective in $\Ocat_\nu(\A_\lambda)$, then $\hat{M}_{fin}$ is projective
in $\Ocat_\nu(\A_\ring)$.
\item
$\Ocat_\nu(\A_\ring)$ is a highest weight category over
$\ring$ in the sense of Rouquier,
\cite{Rouquier}, with standard objects $\Delta_\ring(p)$.
\end{enumerate}
\end{Lem}
\begin{proof}
We prove (1). By Lemma \ref{Lem:deformed_object},  $\hat{M}_{fin}$ is a flat deformation of
$M$. Thanks to (iii), $\Delta_\ring(p)$ is a flat deformation of $\Delta_\lambda(p)$.
Thanks to (ii), $\Delta_\lambda(p)$ has no higher self-extensions, which implies
that a flat deformation is unique. Hence, for $M=\Delta_\lambda(p)$, we get
$\hat{M}_{fin}\cong \Delta_\ring(p)$.

(2) is easy. To prove (3) we observe that, thanks to (1) and (2), the projectives
in $\Ocat_\nu(\A_\ring)$ are filtered by Verma modules. The required upper-triangularity
properties follow from those of $\Ocat_\nu(\A_\lambda)$.
\end{proof}

\subsubsection{Deformed quotient functor}
Now we investigate the structure of $\Ocat_\nu(\A_\ring)$
under the assumption that the functor $\Gamma_\lambda^\theta$ is an exact functor
(and without the assumptions (i)-(iii) from Section \ref{SSS_hw_deformed}).
Note that in this case $\Gamma_\lambda^\theta\circ \Loc^\theta_\lambda\cong \operatorname{id}$.
In particular, $\Gamma_\lambda^\theta$ is a Serre quotient functor. Recall also
that we can identify $\Ocat_\nu(\A_\lambda^\theta)$ with $\Ocat_\nu(\A_{\lambda'})$
via $\Gamma(\A_{\lambda,\chi}^\theta\otimes_{\A_\lambda^\theta}\bullet)$,
where $\chi:=\lambda'-\lambda$ is in $\underline{\Lambda}$ and is sufficiently deep in the chamber of $\theta$,
see Corollary \ref{Cor:abel_loc}.
Then $\Gamma_\lambda^\theta:\Ocat_\nu(\A_\lambda^\theta)\rightarrow
\Ocat_\nu(\A_\lambda)$ is identified with $\Hom_{\A_{\lambda'}}(\A_{\lambda,\chi},\bullet):
\Ocat_\nu(\A_{\lambda'})\rightarrow \Ocat_\nu(\A_\lambda)$ by
Corollary \ref{Cor:abel_loc_bimod}.
Note that, by the choice of $\chi$, the bimodule $\A_{\paramq,\chi}$ is flat over $\C[\paramq]$,
compare with Lemma \ref{Lem:translation_coincidence}.
Set $\A_{\ring,\chi}:=\A_{\paramq,\chi}\otimes_{\C[\paramq]}\ring$.

We need analogs of these statements for $\Ocat_\nu(\A_\ring)$.
We write $\ring'$ for the completion of $\C[\paramq]$ at $\lambda'$
so that $\ring'$ is identified with $\ring$ via the automorphism of $\paramq$
given by $\lambda\mapsto \lambda+\chi$.
We can form the category  $\Ocat_\nu(\A_{\ring'})$.
As we have argued in Section \ref{SSS_hw_deformed}, the latter is a highest
weight category over $\ring'$. Let $\mathsf{A}_{\ring}$
denote the opposite algebra of the endomorphism algebra of a projective generator
of $\Ocat_\nu(\A_{\ring'})$. So $\mathsf{A}_{\ring}$
is flat over $\ring$ and
$\mathsf{A}_{\ring}\operatorname{-mod}\cong \Ocat_\nu(\A_{\ring'})$.

\begin{Prop}\label{Prop:deformed_quotient}
Under the assumptions above, there is an idempotent $\epsilon\in \mathsf{A}_{\ring}$ with the following properties:
\begin{enumerate}
\item $\Ocat_\nu(\A_{\ring})\cong \epsilon \mathsf{A}_{\ring}
\epsilon\operatorname{-mod}$,
\item and the functor $\epsilon\bullet: \mathsf{A}_{\ring}\rightarrow
\epsilon \mathsf{A}_{\ring}
\epsilon\operatorname{-mod}$ is identified with $\Hom_{\A_{\ring'}}(\A_{\ring,\chi},\bullet)$.
\item The quotient $\A_{\ring}/\A_{\ring}\epsilon \A_{\ring}$ is $\ring$-torsion.
\end{enumerate}
\end{Prop}
\begin{proof}
The proof is in several steps.

{\it Step 1}. It follows from Corollary \ref{Cor:abel_loc_bimod}
that  $\A_{\lambda,\chi}$ is a projective $\A_{\lambda'}$-module.
For $k>0$, set  $\ring_k:=\C[\paramq]/\mathfrak{m}^k$
and $\ring'_k:=\C[\paramq]/\left(\mathfrak{m}'\right)^k$, where $\mathfrak{m},
\mathfrak{m}'$ are the maximal ideals of $\lambda,\lambda'$, respectively.
Since $\A_{\paramq,\chi}$
is flat over $\paramq$, we conclude that the base change $\A_{\ring_k,\chi}$
is flat over $\ring_k$. Hence $\A_{\ring_k,\chi}$
is a projective $\A_{\ring'_k}$-module. It follows that, for an $\A_{\ring'_k}$-module
$M$, we have $\Ext^i_{\A_{\ring'}}(\A_{\ring,\chi},M)\cong
\Ext^i_{\A_{\ring'_k}}(\A_{\ring_k,\chi},M)=0$ for $i>0$.

{\it Step 2}. We claim that the functor $\operatorname{Ext}^i_{\A_{\ring'}}(\A_{\ring,\chi},\bullet)$
restricts to $\Ocat_\nu(\A_{\ring'})\rightarrow \Ocat_\nu(\A_{\ring})$.
We equip $\A_{\paramq}$ with the filtration whose associated graded is
$\C[Y][\param]$, see the proof of Lemma \ref{Lem:partial_completion_Noetherian}.
We have induced filtrations on $\A_{\ring'},
\A_{\ring}$. The bimodule $\A_{\ring,\chi}$ also carries
a natural filtration, whose associated graded is the finitely generated
$\C[Y]\otimes \C[\param]^{\wedge_0}$-module $\Gamma(X^\theta,\Str(\chi))\otimes
\C[\param]^{\wedge_0}$. For $M\in \Ocat_\nu(\A_{\ring'})$ we have
a good filtration compatible with the filtration on $\A_{\ring'}$ whose associated graded is
set theoretically supported on $X_+\times \param^{\wedge_0}$, where
$X_+$ denote the contracting locus of $\nu$ in $X$. The module
$N:=\operatorname{Ext}^i_{\A_{\ring'}}(\A_{\ring,\chi},M)$
inherits a filtration whose associated graded is supported on $X_+\times \param^{\wedge_0}$.
It follows that the grading on $N$ is bounded from above. And $N$ is finitely
generated. So it lies in $\Ocat_\nu$.

{\it Step 3}. A short exact sequence of $0\rightarrow M_1\rightarrow M_2\rightarrow M_3\rightarrow 0$
in $\Ocat_\nu(\A_{\ring'})\cong \mathsf{A}_{\ring}\operatorname{-mod}$ gives rise to
a long exact sequence for $R\Hom_{\A_{\ring'}}(\A_{\ring,\chi},\bullet)$.
Note that if $M_3$ is torsion over $\ring$, then
$\operatorname{Ext}^i_{\A_{\ring'}}(\A_{\ring,\chi},M_3)=0$ for $i>0$
by Step 1. Since $\Ocat_\nu(\A_{\ring})$ is equivalent to the category of
modules over a $\ring$-algebra that is a finitely generated module,
we easily deduce that
$\Hom_{\A_{\ring'}}(\A_{\ring,\chi},\bullet)$
is an exact functor $\Ocat_\nu(\A_{\ring'})
\rightarrow \Ocat_\nu(\A_{\ring})$.

{\it Step 4}. We now produce a left inverse functor to
$F:=\Hom_{\A_{\ring'}}(\A_{\ring,\chi},\bullet)$. This functor
is $G:=\A_{\ring,\chi}\otimes_{\A_{\ring}}\bullet$.
To check that it maps $\Ocat_\nu(\A_{\ring})$ to $\Ocat_\nu(\A_{\ring'})$
one argues as in Step 2.

Now we show that the adjunction unit $\operatorname{id}\rightarrow F\circ G$ is an isomorphism.
If $M\in \Ocat_\nu(\A_{\ring})$ is killed by $\mathfrak{m}$,
then $F\circ G=\Hom_{\A_{\lambda'}}(\A_{\lambda,\chi}, \A_{\lambda,\chi}\otimes_{\A_\lambda}\bullet)
\cong \Gamma^\theta_\lambda\circ \Loc^\theta_\lambda$. The morphism $M\rightarrow
F\circ G(M)$ is an isomorphism. Thanks to the five lemma, we see that
this generalizes to the case when $M$ is killed by $\mathfrak{m}^k$
for arbitrary $k$. And since $\Ocat_\nu(\A_{\ring})$
is equivalent to the category of modules over a finite $\ring$-algebra (Proposition
\ref{Prop:deformed_category_O}), we are done.

{\it Step 5}. The functor $G$ sends a projective generator for $\Ocat_\nu(\A_{\ring'})$
to a projective in $\Ocat_\nu(\A_{\ring})$. The latter is a direct sum of indecomposable
projectives. For $\epsilon$ we take an idempotent corresponding to these indecomposable projectives.
Properties (1) and (2) are now standard.

{\it Step 6}. We now show (3). It is equivalent to the condition that $F$ is an equivalence
after base change to $\operatorname{Frac}\ring$. This is a consequence of Lemma
\ref{Lem:inverse_Morita}.
%
\end{proof}

\begin{Cor}\label{Cor:flat_algebra}
Under the assumptions of Proposition \ref{Prop:deformed_quotient},
$\Ocat_\nu(\A_{\ring})$ is equivalent to the category of
modules over an $\ring$-algebra that is a free finite rank
$\ring$-module. In particular, the Hom module between two
projective objects in $\Ocat_\nu(\A_{\ring})$ is a free
finite rank $\ring$-module.
\end{Cor}

\begin{Cor}\label{Cor:flat_image}
For every $M\in \Ocat_\nu(\A_{\ring'})$ that is flat over
$\ring$ (e.g., $M$ is standardly filtered), we have that
$\Hom_{\A_{\ring'}}(\A_{\ring,\chi},M)$ is a flat $\ring$-module.
\end{Cor}

\subsection{Main result}
\subsubsection{$\Ocat$-regular parameters}
Recall, Definition \ref{defi:O_regular}, that $\lambda$ is $\Ocat$-regular
if $\Ca_\nu(\A_\lambda)$ is a commutative semisimple algebra of dimension $|X^T|$.

Below we will use the following notation:
$\Ca_\nu(\A^\theta_{\ring}):=\ring\otimes_{\C[\paramq]}\Ca_\nu(\A^\theta_{\paramq})$

\begin{Lem}\label{Lem:O_reg_equivalent}
The following two conditions are equivalent:
\begin{enumerate}
\item $\lambda$ is $\Ocat$-regular,
\item the homomorphism $\Ca_\nu(\A_\lambda)\rightarrow \Ca_\nu(\A_\lambda^\theta)$ is an isomorphism.
\end{enumerate}
\end{Lem}
\begin{proof}
(2)$\Rightarrow$(1) follows because $\Ca_\nu(\A_\lambda^\theta)$ is a commutative
semisimple algebra of dimension $|X^T|$.

(1)$\Rightarrow$(2): assume (1) holds. Consider
the homomorphism $\Ca_\nu(\A_{\ring})\rightarrow
\Ca_\nu(\A^\theta_{\ring})$. This homomorphism
is an isomorphism over the generic point in $\operatorname{Spec}(\ring)$ by Lemma
\ref{Lem:cartan_iso}. It follows that
$\Ca_\nu(\A_\ring)$ is a free  $\ring$-module of rank $|X^T|$.
Therefore $\Ca_{\nu}(\A_{\ring})\cong \ring^{\oplus X^T}$ as an
algebra. Since the homomorphism $\Ca_\nu(\A_{\ring})\rightarrow
\Ca_\nu(\A^\theta_{\ring})$ is an isomorphism generically, it gives a
bijection between the direct summands in the two algebras. This homomorphism restricts to
nonzero $\ring$-algebra homomorphisms (that a priori do not preserve units)
between the corresponding summands. But every nonzero (not necessarily unital) $\ring$-algebra homomorphism
$\ring\rightarrow \ring$ is an isomorphism.
It follows that $\Ca_\nu(\A_\ring)\rightarrow \Ca_\nu(\A^\theta_\ring)$
is an isomorphism and hence we get (2).
\end{proof}

We let $\paramq^{\Ocat-reg}$ denote the locus of regular parameters and $\paramq^{O-sing}$
denote its complement. Note  that the latter is Zariski closed.
The following is a slightly more precise version of Theorem \ref{Thm:O_regular}
that contains one more statement that will be used below when we deal with type A quiver varieties.

\begin{Thm}\label{Thm:O_regular_precise}
Recall that we assume that $X^T$ is finite.
Let $\lambda$ and $X$ be such that $\Gamma_\lambda^\theta: \Coh(\A_\lambda^\theta)
\rightarrow \A_\lambda\operatorname{-mod}$ is exact. Let $\ring=\C[\param]^{\wedge_\lambda}$.
Suppose that one of the
following two conditions hold:
\begin{enumerate}
\item  $\paramq^{O-sing}\cap \operatorname{Spec}(\ring)$ has codimension at least $2$ in $\operatorname{Spec}(\ring)$,
\item or $\Ca_\nu(\A_\lambda)\twoheadrightarrow \Ca_\nu(\A_\lambda^\theta)$.
\end{enumerate}
Then $\lambda$ is $\Ocat$-regular.
\end{Thm}

\subsubsection{Notation}
Here we introduce the notation we need in our proof of Theorem \ref{Thm:O_regular_precise}.
We consider the quotient functor $F:\Ocat_\nu(\A_{\ring'})
\rightarrow \Ocat_\nu(\A_{\ring})$ given by $\Hom_{\A_{\ring'}}(
\A_{\ring,\chi},\bullet)$. Recall, Proposition
\ref{Prop:deformed_quotient}, that under the identification
of  $\Ocat_\nu(\A_{\ring'})$ with $\mathsf{A}_{\ring}\operatorname{-mod}$
the functor $F$ becomes $\epsilon\bullet$ for an idempotent $\epsilon\in \mathsf{A}_{\ring}$. Let $G$ denote the left adjoint
functor so that $F\circ G=\operatorname{id}$. By (3) of Proposition \ref{Prop:deformed_quotient},
the functor $F$ is an equivalence after based change to $\operatorname{Frac}\ring$.

Recall the element $h\in \A_{\ring}$. Let $\alpha_1,\ldots,\alpha_k$ be its
eigenvalues in $\Ca_{\nu}(\A_\lambda)$. We order these elements in such a way that
$\alpha_i-\alpha_j\in \Z_{>0}\Rightarrow i<j$. As argued in the proof of
Lemma \ref{Lem:deformed_weights},  we can decompose
$\Ca_{\nu}(\A_\ring)$ into the direct sum
$\Ca_{\nu}(\A_\ring)=\bigoplus_{i=1}^k \Ca_{\nu}(\A_\ring)^{\alpha_i}$.
Similarly, $\Ca_{\nu}(\A^\theta_\ring)=\bigoplus_{i=1}^k \Ca_{\nu}(\A^\theta_\ring)^{\alpha_i}$. The homomorphism
$\Ca_\nu(\A_\ring)\rightarrow
\Ca_\nu(\A^\theta_\ring)$ decomposes into the direct sum
of homomorphisms $\Ca_{\nu}(\A_\ring)^{\alpha_i}
\rightarrow \Ca_{\nu}(\A^\theta_\ring)^{\alpha_i}$ for $i=1,\ldots,k$.

Recall, Corollary \ref{Cor:deformed_wt_decomp},  that we have a direct sum decomposition $M=\bigoplus_{\beta\in \C}M^\beta$
for all modules $M\in \Ocat_\nu(\A_\ring)$. For $i=1,\ldots,k$, we consider the subcategory
$\Ocat_\nu(\A_\ring)_{\geqslant i}$ consisting of all objects
$M$ such that $M^\beta\neq \{0\}$ implies $\beta-\alpha_j\in \Z_{\leqslant 0}$ for
some $j\geqslant i$. For notational convenience, we also set
$\Ocat_\nu(\A_\ring)_{\geqslant 0}:=
\Ocat_\nu(\A_\ring)_{\geqslant 1}(=\Ocat_\nu(\A_\ring))$.

Recall also that the simple objects in $\Ocat_{\nu}(\A_{\lambda'})$ are labelled by $X^T$
via the equivalence with $\Ocat_\nu(\A_\lambda^\theta)$. The category $\Ocat_\nu(\A_{\ring'})$
is highest weight, and $F: \Ocat_\nu(\A_{\ring'})\twoheadrightarrow \Ocat_\nu(\A_\ring)$ is an equivalence generically. Let $(X^T)_{\geqslant i}$ denote
set of labels that correspond to the simples in the base change of
$\Ocat_\nu(\A_\ring)_{\geqslant i}$ to $\operatorname{Frac}(\ring)$.

\subsubsection{Proof}
\begin{proof}[Proof of Theorem \ref{Thm:O_regular_precise}]
In the proof we will establish the following three claims by ascending induction on $i=0,\ldots,k$.
\begin{itemize}
\item[($A_i$)] The homomorphism $\Ca_\nu(\A_\ring)^{\alpha_i}\rightarrow
\Ca_\nu(\A^\theta_\ring)^{\alpha_i}$ is an isomorphism (for $i=0$, we assume
that both algebras are zero, so the claim is tautologically true).
\item[($B_i$)] $(X^T)_{\geqslant i+1}$ is a poset in the coarsest highest weight poset
for $\Ocat_\nu(\A_\lambda^\theta)\xrightarrow{\sim}\Ocat_\nu(\A_{\lambda'})$ (the coarsest highest weight posets were described in Section
\ref{SSS_HW_struct}).
\end{itemize}
Thanks to ($B_{i-1}$) it makes sense to define the highest weight subcategory
$\Ocat_\nu(\A_{\ring'})_{\geqslant i}$ spanned by the standard
objects labelled by $(X^T)_{\geqslant i}$.
\begin{itemize}
\item[($C_i$)] The functors  $F,G$ map between $\Ocat_\nu(\A_{\ring'})_{\geqslant i+1}$
and $\Ocat_\nu(\A_\ring)_{\geqslant i+1}$.
In particular, $F$ gives a quotient functor $\Ocat_\nu(\A_{\ring'})_{\geqslant i+1}
\rightarrow \Ocat_\nu(\A_\ring)_{\geqslant i+1}$.
\end{itemize}
Note that ($A_0$) and ($B_0$) are tautological and ($C_0$) is Proposition
\ref{Prop:deformed_quotient}. On the other hand, ($A_1$),\ldots,($A_k$) is what we need
to prove.

Below we will prove that
\begin{itemize}
\item
($C_{i-1}$) implies ($A_i$),
\item ($C_{i-1}$) and  ($A_i$) imply ($B_i$),
\item and  ($C_{i-1}$),($A_i$) and ($B_i$)
imply ($C_i$).
\end{itemize}

{\it Proof of $(C_{i-1})\Rightarrow (A_i)$}. Set $\alpha:=\alpha_i$.
We will write $\Ocat_{\geqslant i}$ for $\Ocat_{\nu}(\A_\ring)_{\geqslant i}$.
First,  we claim that
the module $\Delta_\ring(\Ca_\nu(\A_\ring)^\alpha)$, which is
clearly an object of $\Ocat_{\geqslant i}$,
is projective in that category. More precisely, we claim that
\begin{equation}\label{eq:Hom_description}
\Hom_{\Ocat_{\geqslant i}}(\Delta_\ring(\Ca_\nu(\A_\ring)^\alpha),M)
\cong M^\alpha.
\end{equation}
Indeed,
$$\Hom_{\Ocat_{\geqslant i}}(\Delta_\ring(\Ca_\nu(\A_{\ring})_\alpha),M)
\cong \Hom_{\Ca_\nu(\A_\ring)}(\Ca_\nu(\A_\ring)_\alpha,
M^{\A_{\ring}^{>0}}). $$
The image of $1\in \Ca_\nu(\A_\ring)_\alpha$ in $M$ must lie in
$M^\alpha$. Since $M$ has no weights in $\alpha+\Z_{>0}$,
the whole module $M^\alpha$ is annihilated by $\A_\ring^{>0}$, and
(\ref{eq:Hom_description}) follows.

Thanks to ($C_{i-1}$), compare to  Corollary \ref{Cor:flat_algebra}, we see
that the endomorphism algebra
of any projective object in $\Ocat_{\geqslant i}$ is free over $\ring$.
(\ref{eq:Hom_description}) implies that  $\End(\Delta_\ring(\Ca_\nu(\A_\ring)^\alpha))=\Ca_{\nu}(\A_\ring)^\alpha$. So $\Ca_\nu(\A_\ring)^\alpha$ is free over $\ring$.

Now we are ready to establish ($A_i$). The $\ring$-modules $\Ca_{\nu}(\A_\ring)^\alpha$
and $\Ca_\nu(\A^\theta_\ring)^\alpha$ are free. Over $\operatorname{Frac}\ring$, the algebras
$\Ca_{\nu}(\A_\ring)^\alpha$
and $\Ca_\nu(\A^\theta_\ring)^\alpha$ are isomorphic.  So they are of
the same rank.
The locus where a homomorphism of such modules fails to be an isomorphism is a divisor given
by the  determinant of the homomorphism. If the locus has codimension at least two, then it is empty.
We also easily see that (2)  implies ($A_i$) as well.

{\it Proof of ($C_{i-1}$)\&($A_i$)$\Rightarrow$($B_i$)}.  Take an indecomposable idempotent
$\varepsilon\in \Ca_\nu(\A_\lambda)^\alpha\cong \Ca_\nu(\A^\theta_{\lambda})^\alpha$ and
deform it to an idempotent in $\Ca_\nu(\A_\ring)^\alpha$ to be also
denoted by $\varepsilon$.  Note that $\varepsilon$ corresponds to a point
in $(X^T)_{\geqslant i}\setminus (X^T)_{\geqslant i+1}$, denote it by $p=p(\epsilon)$.

Set $\Ocat'_{\geqslant i}:=\Ocat_\nu(\A_{\ring'})_{\geqslant i}$, this is
a highest weight category over $\ring$ by Lemma \ref{Lem:deformed_hw}.
The Verma module $\Delta_\ring(\Ca_\nu(\A_\ring)\varepsilon)$
is projective in $\Ocat_{\geqslant i}$. By ($C_{i-1}$),  there is an indecomposable projective in
$\Ocat'_{\geqslant i}$, say $P_\epsilon$, such that $F(P_\epsilon)=\Delta_\nu(\Ca_\nu(\A_\ring)\varepsilon)$.
Recall that $F$ is an equivalence over $\operatorname{Frac}\ring$.
Moreover, thanks to (2) of \cite[Proposition 4.11]{catO_charp}, we have
$$G(\Delta_{\operatorname{Frac}\ring}(p'))=\A_{\operatorname{Frac}\ring,\chi}
\otimes_{\A_{\operatorname{Frac}}\ring}\Delta_{\operatorname{Frac}\ring}(p')\cong
\Delta_{\operatorname{Frac}\ring'}(p'), \forall p'\in X^T.$$
It follows that $P_\epsilon=\Delta_{\ring'}(p)$. This is true for all
$p\in (X^T)_i:=(X^T)_{\geqslant i}\setminus (X^T)_{\geqslant i+1}$.

Now we prove ($B_i$), equivalently, that $(X^T)_i$
is a poset coideal in $(X^T)_{\geqslant i}$. Since the classes of standard and costandard
objects in $\Ocat_\nu(\A_{\lambda'})$ coincide, see Lemma \ref{Lem:stand_costand}, our condition
follows from the fact that each $\Delta_{\lambda'}(p)$ with $p\in (X^T)_i$
is projective. ($B_i$) is proved.

{\it Proof of ($C_{i-1}$)\&($A_i$)\&($B_i$)$\Rightarrow$ ($C_i$)}. We need to prove
that $F:\Ocat'_{\geqslant i}\rightarrow \Ocat_{\geqslant i}$ sends
$\Ocat'_{\geqslant i+1}$ to $\Ocat_{\geqslant i+1}$, while
$G: \Ocat_{\geqslant i}\rightarrow \Ocat'_{\geqslant i}$ sends
$\Ocat_{\geqslant i+1}$ to $\Ocat'_{\geqslant i+1}$.

First, we prove that $F(\Ocat'_{\geqslant i+1})\subset \Ocat_{\geqslant i+1}$.
Note that $\Ocat'_{\geqslant i+1}$ consists of all objects in $\Ocat'_{\geqslant i}$
that admit no nonzero
Hom's from the projectives $\Delta_{\ring'}(p)$ for $p\in (X^T)_i$, and
similarly, $\Ocat'_{\geqslant i+1}$ consists of all objects that admit no
nonzero Hom's from the projectives $\Delta_{\ring'}(p)$. Note that
$G$ maps indecomposable projectives in $\Ocat_{\geqslant i}$ to indecomposable
projectives in $\Ocat'_{\geqslant i}$. As was noted in the previous part of
the proof, over $\operatorname{Frac}\ring$, the object $G(\Delta_\ring(p))$
coincides with $\Delta_{\ring'}(p)$. It follows that
$G(\Delta_\ring(p))=\Delta_{\ring'}(p)$. By adjointness,
$F$ sends  $\Ocat'_{\geqslant i+1}$ to $\Ocat_{\geqslant i+1}$.

Now we prove that $G(\Ocat_{\geqslant i+1})\subset\Ocat'_{\geqslant i+1}$.
Let $G_\lambda,F_\lambda$ denote the specializations of $G$ to the closed point in
$\operatorname{Spec}\ring$ and let $\Ocat_{\geqslant i,\lambda},\Ocat'_{\geqslant i,\lambda}$
denote the specializations of the categories. So $F_\lambda: \Ocat'_{\geqslant i,\lambda}
\rightarrow \Ocat_{\geqslant i,\lambda}$ is a Serre quotient functor and
$G_\lambda$ is its left adjoint (and right inverse). Note that
$F_\lambda(L_{\lambda'}(p))\neq 0$ for all $p\in (X^T)_i$. This is because $\Delta_{\lambda'}(p)$
is the projective cover of $L_{\lambda'}(p)$ and, by the previous part of the proof,
we have that
 $F_\lambda(\Delta_{\lambda'}(p))$  is the
indecomposable projective $\Delta_\lambda(p)$. This also implies that
$F_\lambda(L_{\lambda'}(p))\not\in \Ocat_{\geqslant i+1,\lambda}$.

So the simples in $\Ocat_{\geqslant i+1,\lambda}$
are nonzero objects of the form $F_\lambda(L)$, where $L$ is a simple in
$\Ocat'_{\geqslant i+1,\lambda}$. For such $L$, the object $G_\lambda(F_\lambda(L))$
surjects onto $L$ and the kernel is annihilated by $F_\lambda$. Therefore the
kernel lies in   $\Ocat'_{\geqslant i+1,\lambda}$ and hence
$G_\lambda$ sends every simple in $\Ocat_{\geqslant i+1,\lambda}$ to
$\Ocat'_{\geqslant i+1,\lambda}$. Hence $G_\lambda(\Ocat_{\geqslant i+1,\lambda})
\rightarrow \Ocat'_{\geqslant i+1,\lambda}$.

We deduce that $G$ maps every object in $\Ocat_{\geqslant i+1}$ that is annihilated
by $\mathfrak{m}^k$ to $\Ocat'_{\geqslant i+1}$. Since $\Ocat_{\geqslant i+1},\Ocat'_{\geqslant i+1}$
are equivalent to  categories of modules over finite $\ring$-algebras, this implies that $G(\Ocat_{\geqslant i+1})
\subset \Ocat'_{\geqslant i+1}$ and completes the proof.
\end{proof}

\subsubsection{Comments}
Suppose that $\Gamma^\theta_\lambda$ is exact and $\lambda\in \paramq^{\Ocat-reg}$.
In particular,  $\Gamma_\lambda^\theta:\Ocat_\nu(\A_\lambda^\theta)\rightarrow
\Ocat_\nu(\A_\lambda)$ is a quotient functor between categories with the same
number of simples. So it is an equivalence.  This equips $\Ocat_\nu(\A_\lambda)$
with a highest weight structure and we would like to describe its order and
standard objects. To $p\in X^T$ we assign the image of $h$ in the quotient
of $\Ca_\nu(\A_\lambda)=\Ca_\nu(\A_\lambda^\theta)$ associated to $p$, to be denoted
by $\alpha(p)$. Set $p<_\lambda p'$  if $h_{p'}(\lambda)-h_p(\lambda)\in \Z_{>0}$.

The following claim has been established in the proof of Theorem \ref{Thm:O_regular_precise}.

\begin{Thm}\label{Thm:hw}
Suppose that $\Gamma^\theta_\lambda$ is exact and $\lambda\in \paramq^{\Ocat-reg}$.
Then the order $<_\lambda$ is a highest weight order for the highest weight category
$\Ocat_\nu(\A_\lambda)$. Moreover, the standard objects are $\Delta_\lambda(p)$,
where $p$ runs over $X^T$.
\end{Thm}

Before, this was known for $\lambda$ outside an unspecified finite union of essential
hyperplanes.

Also let us point out that it looks likely that for each $\lambda\in \paramq$, there is $\theta$ such that
$\Gamma^\theta_\lambda$ is exact. This, together with Theorem \ref{Thm:O_regular_precise},
that $\paramq^{O-sing}\subset \paramq$ has pure codimension $1$ and, therefore,
is the union of some essential hyperplanes.

\subsection{Reduction to rank $1$}
We are interested in determining essential but non-singular hyperplanes whose Weil generic points
are $\Ocat$-singular. Fix such an essential hyperplane, $\tilde{\Upsilon}$. Let $\Upsilon$ be the
classical wall parallel to $\tilde{\Upsilon}$.

We use the notation of Section \ref{SSS_conical_slices}. Pick an element $\zeta\in \Upsilon$
as in Proposition \ref{Prop:rank1}. The torus $T$
acts on $X_{\zeta}$, the fixed points are in a natural bijection with $X^T$.
So the number of fixed points in $Y_{\zeta}$ is also finite: let $y_1,\ldots,
y_k$ denote these fixed points.  We can form the corresponding slice resolutions,
$\underline{X}_1,\ldots,\underline{X}_k$. Set $\underline{\param}_i:=H^2(\underline{X}_i,\C)$
and let $\eta_i$ denote the pullback map $\param\rightarrow \underline{\param}_i$.

\begin{Prop}\label{Prop:O_reg_rk1}
For an essential non-singular
hyperplane $\tilde{\Upsilon}\subset \paramq$ parallel to $\Upsilon$,  the following two conditions are equivalent:
\begin{enumerate}
\item $\tilde{\Upsilon}\subset \paramq^{O-sing}$,
\item For some $i=1,\ldots,k$, the point $\eta_i(\tilde{\Upsilon})\in \underline{\paramq}_i$
is $\Ocat$-singular.
\end{enumerate}
\end{Prop}
\begin{proof}
Let $\tilde{\Upsilon}_\hbar$ denote the linear subspace in $\param\oplus \C$ spanned by
$\tilde{\Upsilon}$ (embedded into the affine hyperplane $\hbar=1$).

Consider the specialization $\A_{\tilde{\Upsilon}}$, its Cartan subquotient $\Ca_\nu(\A_{\tilde{\Upsilon}})$
and its Rees algebra $R_\hbar(\Ca_\nu(\A_{\tilde{\Upsilon}}))=\Ca_\nu(R_\hbar(\A_{\tilde{\Upsilon}}))$.
We have $R_\hbar(\Ca_\nu(\A_{\tilde{\Upsilon}}))/(\hbar)=\Ca_\nu(\C[Y_\Upsilon])$, where we write
$Y_\Upsilon$ for $\Upsilon\times_\param Y_\param$.

Let $y$ be one of the points $y_1,\ldots,y_k$ and $\eta$
for the corresponding pullback map. Consider the algebra $R_\hbar(\Ca_\nu(\A_{\tilde{\Upsilon}}))^{\wedge_y}$. It is naturally isomorphic to
$\Ca_\nu(R_\hbar(\A_{\tilde{\Upsilon}})^{\wedge_y})$. From Section \ref{SSS_slice_quant} it follows that we have a $T$-equivariant
isomorphism $$R_\hbar(\A_{\tilde{\Upsilon}})^{\wedge_y}\cong
\C[[\tilde{\Upsilon}_\hbar]]\widehat{\otimes}_{\C[[\hbar]]}\Weyl_\hbar(T_y\mathcal{L})^{\wedge_0}
\widehat{\otimes}_{\C[[\hbar]]}R_\hbar(\underline{\A}_{\eta(\tilde{\Upsilon})})^{\wedge_0}.$$
The Cartan subquotient of the right hand side is easily seen to coincide
with $$\C[[\tilde{\Upsilon}_\hbar]]\widehat{\otimes}_{\C[[\hbar]]} R_\hbar(\Ca_\nu(\underline{\A}_{\eta(\tilde{\Upsilon})}))^{\wedge_0}.$$
So we conclude that $$R_\hbar(\Ca_\nu(\A_{\tilde{\Upsilon}}))^{\wedge_y}\cong \C[[\tilde{\Upsilon}_\hbar]]\widehat{\otimes}_{\C[[\hbar]]} R_\hbar(\Ca_\nu(\underline{\A}_{\eta(\tilde{\Upsilon})}))^{\wedge_0}.$$
When we need to indicate the dependence of the right hand side on $i$ we will write
$\underline{\A}^i_{\eta(\tilde{\Upsilon})}$ instead of $\underline{\A}_{\eta(\tilde{\Upsilon})}$.

Note that the variety defined by the finite $\C[\Upsilon]$-algebra $\Ca_\nu(\C[Y_\Upsilon])$ is $Y_\Upsilon^T$.
Since $Y_\Upsilon^T$ is finite over $\Upsilon$, we have that
$$R_\hbar(\Ca_\nu(\A_{\tilde{\Upsilon}}))^{\wedge_{\zeta}}\cong \bigoplus_{i=1}^k
R_\hbar(\Ca_\nu(\A_{\tilde{\Upsilon}}))^{\wedge_{y_i}}.$$
We conclude that
\begin{equation}\label{eq:reduct_main}
R_\hbar(\Ca_\nu(\A_{\tilde{\Upsilon}}))^{\wedge_{\zeta}}
\cong \bigoplus _i\C[[\tilde{\Upsilon}_\hbar]]
\widehat{\otimes}_{\C[[\hbar]]}\Ca_\nu(R_\hbar(\underline{\A}^i_{\eta_i(\tilde{\Upsilon})}))^{\wedge_0}.
\end{equation}

Now we are ready to show that (1)$\Rightarrow$(2). Assume first that all $\eta_i(\tilde{\Upsilon})$ are $\Ocat$-regular. Then, after inverting $\hbar$, the right
hand side is isomorphic to $\C[[\tilde{\Upsilon}_\hbar]][h^{-1}]^{\oplus X^T}$. So the same
holds for the left hand side. In particular, the generic rank of the finitely generated
$\C[\tilde{\Upsilon}_\hbar]$-module $\C[\tilde{\Upsilon}_\hbar]\otimes_{\C[\param,\hbar]} R_\hbar(\Ca_\nu(\A_{\paramq}))$ equals $|X^T|$.
On the other hand, note that for a Zariski generic parameter $\lambda\in \tilde{\Upsilon}$,
$\Gamma_\lambda^\theta$ is an equivalence because $\tilde{\Upsilon}$ is not singular. It follows
that the number of irreducible representations of $\Ca_\nu(\A_\lambda)$ for a Zariski
generic parameter $\lambda\in \tilde{\Upsilon}$ equals to $|X^T|$. But $\dim \Ca_\nu(\A_\lambda)$
equals to the generic rank of $\C[\tilde{\Upsilon}_\hbar]\otimes_{\C[\param,\hbar]} R_\hbar(\Ca_\nu(\A_{\paramq}))$, which is $|X^T|$. It follows that $\lambda$ is $\Ocat$-regular. This proves (1)$\Rightarrow$(2).

The proof of (2)$\Rightarrow$(1) is similar.
\end{proof}

\subsection{Connection to Hikita  conjecture}\label{SS_Hikita_connection}
Suppose that $X$ is a conical symplectic resolution with a Hamiltonian action of
a torus $T$ satisfying $|X^T|<\infty$. Choose a generic one-parameter
subgroup $\nu:\C^\times\rightarrow T$. It was conjectured in \cite[Section 10]{BLPW} that there
is a {\it dual} conical symplectic resolution $X^!$ depending on $\nu$ that also comes equipped with
a Hamiltonian action of a torus, $T^!$, and $X^T$ is in bijection with $(X^!)^{T^!}$.

Now we recall Hikita's conjecture. We can still consider the algebra
$$\Ca_\nu(\C[Y]):=\C[Y]^0/(\sum_{i>0}\C[Y]^{-i}\C[Y]^i).$$

\begin{Conj}[Hikita]\label{Conj:Hikita}
We have an algebra isomorphism $\Ca_\nu(\C[Y])\cong H^*(X^!)$.
\end{Conj}

The most important consequence for us is the following result.

\begin{Cor}\label{Cor:Hikita_flatness}
Suppose that Conjecture \ref{Conj:Hikita} holds. Then $\Ca_\nu(\A_\lambda)$ is a commutative algebra
of dimension $|X^T|$.
\end{Cor}
\begin{proof}
Consider the $\C[\paramq,\hbar]$-algebra $R_\hbar(\Ca_\nu(\A_{\paramq}))=
\Ca_\nu(R_\hbar(\A_{\paramq}))$. Its specialization to $0\in \param\oplus \C$
is $\Ca_\nu(\C[Y])$ hence has dimension $|(X^!)^{T^!}|=|X^T|$, thanks to
Conjecture \ref{Conj:Hikita}. Its generic specialization
has dimension $|X^T|$ and is commutative. It follows that
$R_\hbar(\Ca_\nu(\A_{\paramq}))$ is commutative and is a free
$\C[\paramq,\hbar]$-module of rank $|X^T|$.
\end{proof}

\begin{Cor}\label{Cor:localization_Hikita}
Under the assumption of Corollary \ref{Cor:Hikita_flatness}, if $\Gamma^\theta_\lambda:
\Ocat_\nu(\A_\lambda^\theta)\rightarrow \Ocat_\nu(\A_\lambda)$ is an equivalence,
then $\lambda$ is $\Ocat$-regular.
\end{Cor}
\begin{proof}
Since $\Gamma^\theta_\lambda$ is an equivalence, then the number of irreducibles in
$\Ocat_\nu(\A_\lambda)$ equals $|X^T|$. So the same is true for the number of
irreducible representations of
$\Ca_\nu(\A_\lambda)$. Since $\Ca_\nu(\A_\lambda)$ is a commutative algebra of dimension
$|X^T|$, Corollary \ref{Cor:Hikita_flatness}, shows that $\Ca_\nu(\A_\lambda)\cong \C^{\oplus X^T}$,
i.e., $\lambda$ is $\Ocat$-regular.
\end{proof}

Hikita's conjecture is known in a number of cases, \cite{Hikita}. For our purposes,
the following is important.

\begin{Lem}\label{Lem:Hikita}
Hikita's conjecture holds when $X=T^*\operatorname{Gr}(k,n)$ and when
$X=\operatorname{Hilb}_n(\C^2)$.
\end{Lem}

This is proved in \cite[Theorems 1.2,1.3]{Hikita}.

Note that we will only need to know that $\dim \Ca_\nu(\C[Y])=|X^T|$.

\section{Quiver varieties of type $A$}
\subsection{Main result}\label{SS_quiver_A_main}
Here we consider the case when $X$ is a quiver variety of finite or affine type $A$.
Here is the main result.

\begin{Thm}\label{Thm:type_A}
Conjectures \ref{Conj:fin_hom_dim}, \ref{Conj:der_loc}, \ref{Conj:abelian}
hold for finite or affine type $A$ quiver varieties.
\end{Thm}

Note that we know all singular hyperplanes, see Section \ref{SSS_quant_quiver_walls}.
So Theorem \ref{Thm:type_A}
completely describes the locus where abelian and derived localization hold.

The varieties $X$ in question satisfy both ($\heartsuit$) and $|X^T|<\infty$. Thanks
to Theorems \ref{Thm:exactness} and \ref{Thm:O_regular}, what remains to prove
is the following two claims:

\begin{enumerate}
\item For a suitable choice of a generic one-parameter subgroup $\nu:\C^\times\rightarrow T$
every $\Ocat$-singular hyperplane is singular,
\item Once $\Gamma_\lambda^\theta$ is an equivalence between categories $\Ocat$ for $X$ and all
of its slices, it is an equivalence $\Coh(\A_\lambda^\theta)\xrightarrow{\sim}\A_\lambda\operatorname{-mod}$.
\end{enumerate}

We note that all slices to $X$ are again affine or finite type $A$ quiver varieties by
Lemma \ref{Lem:type_A_slices} so (1) and (2) indeed prove Theorem \ref{Thm:type_A}.

(2) is quite standard, while (1) is more complicated. There are three key ingredients
of proving (1).

First, using Proposition \ref{Prop:O_reg_rk1}, we can reduce to the case when
the quiver has a single vertex. Using results of Section \ref{SS_Hikita_connection},
we then handle the cases of $X=T^*\operatorname{Gr}(k,n)$ and $\operatorname{Hilb}_n(\C^2)$.
This leaves the case of higher rank Gieseker moduli spaces. It is handled by reduction
to the case of Hilbert schemes. This reduction is morally similar to the approach to abelian localization
for these varieties from \cite[Section 5]{Gies}.

\subsection{Gieseker case}\label{SS_Gies_O_reg}
\subsubsection{The main result}
Consider the variety $X=\M^\theta(n,r)$. Set $V:=\C^n$ and $W:=\C^r$. Let $T_0$ denote a maximal torus in
$\GL(W)$ and $\C^\times$ denote the one-dimensional torus acting on $V$.
Then $T:=T_0\times \C^\times$ acts on $X$ with finitely many fixed points.
We choose $\nu:\C^\times\rightarrow T_0\times \C^\times$ of the form
$(\operatorname{diag}(t^{d_1},\ldots, t^{d_r}),t)$ with $d_i-d_{i+1}>n$
for all $i$. Consider also the one-parameter subgroup $\nu_0:\C^\times\rightarrow
T_0\times \C^\times$ given by $t\mapsto (\operatorname{diag}(t^{d_1},\ldots,t^{d_r}),1)$.

We have $\paramq=\C$.
Recall, \cite[Theorem 1.1]{Gies}, that the singular values of $\lambda$ are characterized
as follows: $\lambda-r/2$ is a rational number in $(-r,0)$ with denominator
$\leqslant n$.

\begin{Prop}\label{Prop:Gies_O_sing}
For $\nu$ as above, we have $\paramq^{O-sing}=\paramq^{sing}$.
\end{Prop}

This proposition will be proved in Section
\ref{SSS_Gies_sing_proof}.

\subsubsection{Surjectivity of $\Ca_{\nu_0}(\A_\lambda(n,r))
\rightarrow \Gamma(\Ca_{\nu_0}(\A^\theta_{\lambda}(n,r)))$}
Set $X:=\M^\theta(n,r), Y:=\M^0(n,r)$.
We have the quantization $\Ca_{\nu_0}(\A^\theta_{\lambda}(n,r))$
of $X^{\nu_0(\C^\times)}$, \cite[Section 5.2]{catO_R}. The locus $X^{\nu_0(\C^\times)}$
is the disjoint union $\bigsqcup \prod_{i=1}^r \operatorname{Hilb}_{n_i}(\C^2)$,
where the union is taken over all compositions $(n_1,\ldots,n_r)$ of $n$.

The following result was obtained in \cite[Proposition 3.5]{Gies}.

\begin{Lem}\label{Lem:global_section_Cartan}
We have
$$\Gamma(\Ca_{\nu_0}(\A^\theta_{\lambda}(n,r)))\cong
\bigoplus \bigotimes_{i=1}^r \A_{\lambda+(i-1)}(n_i,1).$$
\end{Lem}

By the construction of     $\Ca_{\nu_0}(\A^\theta_{\lambda}(n,r))$
we have a homomorphism
\begin{equation}\label{eq:Cartan_homom} \Ca_{\nu_0}(\A_\lambda(n,r))
\rightarrow \Gamma(\Ca_{\nu_0}(\A^\theta_{\lambda}(n,r)))\end{equation}

The following lemma will allow us to reduce the proof of Proposition
\ref{Prop:Gies_O_sing} to the case of Hilbert schemes.

\begin{Lem}\label{Lem:Cartan_surjectivity}
If $\lambda$ is regular, then (\ref{eq:Cartan_homom}) is surjective.
\end{Lem}
\begin{proof}
The proof is in several steps.

{\it Step 1}. Let $Y_0:=\operatorname{Spec}(\C[X^{\nu_0(\C^\times)}])$.
So $Y_0$ is the disjoint union of  $(\mathfrak{h}\oplus \mathfrak{h}^*)/(\prod_{i=1}^r S_{n_i})$,
where $\mathfrak{h}$ stands for the Cartan subalgebra in $\g$. Next, set $Y_0':=
\operatorname{Spec}(\Ca_{\nu}(\C[Y]))^{red}$, where the superscript ``red''
means that we consider the reduced scheme structure.

The resolution of singularities morphism $X\rightarrow Y$
gives rise to a finite dominant morphism $\pi_1:
Y_0\rightarrow Y'_0$. Set $Z:=\g\quo G$. Also consider the morphism
$\pi_2:Y_0'\rightarrow Z$ restricted from
$Y\rightarrow Z$. Set $\pi:=\pi_2\circ\pi_1$.
The variety $Z$ is naturally stratified (via the stabilizer stratification). We note that the image of the closure of every leaf in $Y_0$ is the closure of a single
stratum and the dimension of the leaf is twice bigger than the dimension
of its image in $Z$. Note also that, according to \cite[Lemma 6.7]{catO_R}, $Y_0$ has finitely
many leaves.

Below we will write $\A_\lambda$ for $\A_\lambda(n,r)$ and
$\A_\lambda^\theta$ for $\A^\theta_\lambda(n,r)$.

{\it Step 2}. Consider the cokernel, $C$, of (\ref{eq:Cartan_homom}). This is a HC
$\Ca_{\nu_0}(\A_\lambda)$-bimodule. So it makes sense to speak about its
associated variety, $\operatorname{V}(C)$, in $Y_0'$. This variety must be a union of symplectic leaves. We have  $\operatorname{V}(C)\subset \operatorname{V}(\Gamma(\Ca_{\nu_0}(\A_\lambda^\theta)))$ and
$\operatorname{V}(\Gamma(\Ca_{\nu_0}(\A_\lambda^\theta)))=\pi_1(Y_0)$. Note that every symplectic
leaf in $\pi_1(Y_0)$ is the image of a symplectic leaf in $Y_0$. Also the image of every symplectic
leaf in $Y_0$ under $\pi$ is a stratum in $Z$. It follows that
$\pi_2(\operatorname{V}(C))$ is a union of strata.  Choose
a point $z$ in an open stratum of $\pi_2(\operatorname{V}(C))$.

{\it Step 3}. We will arrive at a contradiction with \cite[Proposition 5.4]{Gies} by
comparing  $\Ca_{\nu_0}(\A_\lambda)\rightarrow \Gamma(\Ca_{\nu_0}(\A_\lambda^\theta))$
with the similarly defined homomorphism for the slice algebras.

Let $\tau=(\tau_1,\ldots,\tau_\ell)$ be the partition corresponding to the stratum of  $z$.
We write $\A_\lambda(\tau)$ for $\bigotimes_{i=1}^\ell \A_\lambda(\tau_i,r)$
and $\A_\lambda^\theta(\tau)$ for the sheaf version.


{\it Step 4}. Consider the algebra
$\A_\lambda^{\wedge_z}:=\C[Z]^{\wedge_z}\otimes_{\C[Z]}
\A_\lambda$.

Our next goal is to describe the structure of $\A_\lambda^{\wedge_z}$.
Now note that $\A_\lambda^{\wedge_z}$ is the quantum Hamiltonian reduction
for the action of $G$ on
$$\C[Z]^{\wedge_z}\otimes_{\C[Z]} D(R)
=D(Z^{\wedge_z}\times_Z R).$$
Let $\xi$ be a semisimple element in $\g$ mapping to $z$. We can assume that $\xi$
is diagonal and that the centralizer of $\xi$ in $G$ is a standard Levi subgroup,
to be denoted by $G_\tau$.  Note
$\g_\tau$ embeds into $\g$ via $x\mapsto \xi+x$, this realizes $\g_\tau$ as a transverse slice to $G\xi$ in $\g$. We write $Z_\tau$ for $\g_\tau\quo G_\tau$
and $R_\tau$ for $\g_\tau\oplus \operatorname{Hom}(V,W)$.
The embedding of $\g_\tau$ into $\g_\tau$ gives rise to an isomorphism
$$G\times^L (Z_\tau^\wedge\times_{Z_\tau}R_\tau)\xrightarrow{\sim} Z^{\wedge_z}\times_Z R.$$
Hence we get an identification of quantum Hamiltonian reductions:
$$
\A_\lambda^{\wedge_z}\cong \C[Z_\tau]^{\wedge_0}\otimes_{\C[Z_\tau]}D(R_\tau)/\!/\!/_\lambda G_\tau.$$
Note that the right hand side is naturally identified with
$$\A_\lambda(\tau)^{\wedge_0}:=\C[Z_\tau]^{\wedge_0}\otimes_{\C[Z_\tau]}\A_\lambda(\tau).$$
So we arrive at the identification
\begin{equation}\label{eq:Ham_red_ident}
\A_\lambda^{\wedge_z}\cong \A_\lambda(\tau)^{\wedge_0}.
\end{equation}

{\it Step 5}.
Set
\begin{align*}&\Ca_{\nu_0}(\A_\lambda)^{\wedge_z}:=\C[Z]^{\wedge_z}\otimes_{\C[Z]}\Ca_{\nu_0}(\A_\lambda)=
\Ca_{\nu_0}(\A_\lambda^{\wedge_z}),\\
&\Gamma(\Ca_{\nu_0}(\A_\lambda^\theta))^{\wedge_z}:=\C[Z]^{\wedge_z}\otimes_{\C[Z]}
\Gamma(\Ca_{\nu_0}(\A_\lambda^\theta)),\\
&\Ca_{\nu_0}(\A_\lambda(\tau))^{\wedge_0}:=\C[Z_\tau]^{\wedge_0}\otimes_{\C[Z_\tau]}\Ca_{\nu_0}(\A_\lambda(\tau))=
\Ca_{\nu_0}(\A_\lambda(\tau)^{\wedge_0}),\\
&\Gamma(\Ca_{\nu_0}(\A_\lambda^\theta(\tau)))^{\wedge_0}:=\C[Z_\tau]^{\wedge_0}\otimes_{\C[Z_\tau]}
\Gamma(\Ca_{\nu_0}(\A_\lambda^\theta(\tau))).
\end{align*}
We note that (\ref{eq:Ham_red_ident}) gives rise to an isomorphism
\begin{equation}\label{eq:Cartan_iso1}
\Ca_{\nu_0}(\A_\lambda)^{\wedge_z}\xrightarrow{\sim} \Ca_{\nu_0}(\A_\lambda(\tau))^{\wedge_0}.
\end{equation}

On the other hand, after we identify $\Gamma(\Ca_{\nu_0}(\A_\lambda^\theta)),\Gamma(\Ca_{\nu_0}(\A_\lambda^\theta(\tau)))$
with the direct sums of spherical rational Cherednik algebras and use the isomorphism
(or, more precisely, the direct sum of isomorphisms) similar to
(\ref{eq:Ham_red_ident}), we get an isomorphism
\begin{equation}\label{eq:Cartan_iso2}
\Gamma(\Ca_{\nu_0}(\A_\lambda^\theta))^{\wedge_z}\xrightarrow{\sim}
\Gamma(\Ca_{\nu_0}(\A_\lambda^\theta(\tau)))^{\wedge_0}.
\end{equation}
In the subsequent steps we will prove that
\begin{itemize}
\item[(*)]
the isomorphisms
(\ref{eq:Cartan_iso1}) and (\ref{eq:Cartan_iso2}) intertwine
\begin{equation}\label{eq:Cartan_homom1} \Ca_{\nu_0}(\A_\lambda(n,r))^{\wedge_z}
\rightarrow \Gamma(\Ca_{\nu_0}(\A^\theta_{\lambda}(n,r)))^{\wedge_z}
\end{equation}
with
\begin{equation}\label{eq:Cartan_homom2}\Ca_{\nu_0}(\A_\lambda(\tau))^{\wedge_0}\xrightarrow{\sim}
\Gamma(\Ca_{\nu_0}(\A_\lambda^\theta(\tau)))^{\wedge_0}.
\end{equation}
\end{itemize}
We will then use \cite[Proposition 5.4]{Gies} to arrive at a contradiction with the choice of $z$.

{\it Step 6}. To prove (*) we will need a characterization of (\ref{eq:Cartan_iso1}).
For this, we will use a construction from the proof of \cite[Proposition 5.5]{Gies}.
Let $Z^{Reg}$ denote the open stratum in $Z$. The notations $\g^{Reg}$ and $\h^{Reg}$
have the similar meaning.

Set
$\A_\lambda^{Reg}:=\C[Z^{Reg}]\otimes_{\C[Z]}\A_\lambda$. So $\A_\lambda^{Reg}$
is identified with the quantum Hamiltonian reduction of
$D(\g^{Reg}\times \operatorname{Hom}(V,W))$. This gives rise to an identification
\begin{equation}\label{eq:reg_locus_ident}
\A_\lambda^{Reg}\cong \left(\C[\mathfrak{h}^{Reg}]\otimes_{\C[\mathfrak{h}]} \A_\lambda(1,r)^{\otimes n}\right)^{S_n}.
\end{equation}
We note that (\ref{eq:reg_locus_ident}) is compatible with
(\ref{eq:Ham_red_ident}) in the following way. Set $Z^{\wedge_z,Reg}:=Z^{\wedge_z}\cap Z^{Reg}$
intersection of the subschemes in $Z$ so that $Z^{\wedge_z,Reg}$ is a principal open
affine subscheme in $Z^{\wedge_z}$. The pullbacks of $\A_{\lambda}^{\wedge_z}$ and $\A_\lambda^{Reg}$
to $Z^{\wedge_z,Reg}$ are naturally identified. The pullbacks of the right hand sides of
(\ref{eq:reg_locus_ident}) and (\ref{eq:Ham_red_ident})  to $Z^{\wedge_z,Reg}$ are also
naturally identified. These identifications intertwine the pullbacks of isomorphisms
(\ref{eq:Ham_red_ident}) and (\ref{eq:reg_locus_ident}).

{\it Step 7}. Consider the natural homomorphism $\A_\lambda^{\wedge_z}\rightarrow \A_\lambda^{\wedge_z,Reg}$. Both algebras are filtered (with filtrations coming
from the filtration by the order of differential operator) and the homomorphism
preserves the filtration. It is easy to see that the associated graded homomorphism
is injective. Hence  $\A_\lambda^{\wedge_z}\rightarrow \A_\lambda^{\wedge_z,Reg}$
is injective. It follows that the induced homomorphism
$\left(\A_\lambda^{\wedge_z}\right)^{\nu_0(\C^\times)}\rightarrow
\left(\A_\lambda^{\wedge_z,Reg}\right)^{\nu_0(\C^\times)}$ is also injective.
Therefore, (\ref{eq:Cartan_iso1}),(\ref{eq:Cartan_homom1}) are uniquely recovered
from their localization to $\C[Z^{\wedge_z,Reg}]$. A similar claim holds
for (\ref{eq:Cartan_iso2}),(\ref{eq:Cartan_homom2}). It follows that it is sufficient to
verify  our claim --
that (\ref{eq:Cartan_iso1}), (\ref{eq:Cartan_iso2}) intertwine
(\ref{eq:Cartan_homom1}) with (\ref{eq:Cartan_homom2}) -- after the
localization to the regular locus.

For this we note that
$$\Ca_{\nu_0}\left(\left(\C[\mathfrak{h}^{Reg}]\otimes_{\C[\mathfrak{h}]} \A_\lambda(1,r)^{\otimes n}\right)^{S_n}\right)\cong
\left( D(\mathfrak{h}^{Reg})\otimes \Ca_{\nu_0}\left(\bar{\A}_\lambda(1,r)^{\otimes n}\right)\right)^{S_n}.$$
As was explained in the proof of \cite[Proposition 5.5]{Gies},
$\Ca_{\nu_0}\left(\bar{\A}_\lambda(1,r)^{\otimes n}\right)$ is naturally identified
with $(\C^{\oplus r})^{\otimes n}$, where the copies of $\C$ in the brackets are
indexed by the $\nu_0$-eigen-basis elements of $W$.
So we have an identification
\begin{equation}\label{eq:localiz_ident}
\Ca_{\nu_0}(\A_\lambda^{Reg})\cong
\left( D(\mathfrak{h}^{Reg})\otimes (\C^{\oplus r})^{\otimes n}\right)^{S_n}
\end{equation}
The algebra
$$\Gamma(\Ca_{\nu_0}(\A^\theta_{\lambda}(n,r)))^{Reg}=\C[Z^{Reg}]\otimes_{\C[Z]}
\Gamma(\Ca_{\nu_0}(\A^\theta_{\lambda}(n,r)))$$
naturally identifies with the right hand side of (\ref{eq:localiz_ident})
so that the localization of (\ref{eq:Cartan_homom}) is the identity -- in fact,
this is a part of the argument in \cite[Proposition 5.5]{Gies}. (\ref{eq:localiz_ident})
gives rise to an identification
\begin{equation}\label{eq:localiz_ident1}
\Ca_{\nu_0}(\A_\lambda^{\wedge_z,Reg})\cong
\left( D(\mathfrak{h}^{Reg})^{\wedge_z}\otimes (\C^{\oplus r})^{\otimes n}\right)^{S_n}
\end{equation}
We can write similar identifications for the other three algebras in
(\ref{eq:Cartan_iso1}) and (\ref{eq:Cartan_iso2}). Under these identifications,
all four homomorphisms (\ref{eq:Cartan_iso1}),(\ref{eq:Cartan_iso2}),
(\ref{eq:Cartan_homom1}),(\ref{eq:Cartan_homom2}) become the identity homomorphisms.
So the localized version of (*) is verified. Hence (*) is verified.

{\it Step 8}. Now we can finish the proof. Set
$\bar{\A}_\lambda(\tau)=\bigotimes_{i=1}^\ell \bar{\A}_\lambda(\tau_i)$
and let $\bar{\A}^\theta_\lambda(\tau)$ have the similar meaning.
Thanks to (*), the claim that the cokernel of
(\ref{eq:Cartan_homom}) is nonzero and is supported on the stratum of $z$
in $Z$ translates to saying that the cokernel of
$$\Ca_{\nu_0}(\bar{\A}_\lambda(\tau))\xrightarrow{\sim}
\Gamma(\Ca_{\nu_0}(\bar{\A}_\lambda^\theta(\tau)))$$
is nonzero and finite dimensional. This contradicts
\cite[Proposition 5.4]{Gies}. The contradiction finishes the proof.
\end{proof}

\subsubsection{Proof of Proposition \ref{Prop:Gies_O_sing}}\label{SSS_Gies_sing_proof}
\begin{proof}
By Lemma \ref{Lem:Cartan_surjectivity}, the homomorphism
\begin{equation}\label{eq:surjection_Cartan} \Ca_{\nu_0}(\bar{\A}_\lambda(n,r))\rightarrow
\Gamma(\Ca_{\nu_0}(\bar{\A}^\theta_\lambda(n,r)))
\end{equation}
is surjective.

Pick $m\gg 0$ and consider the one-parameter subgroup $\nu=(m\nu_0,\nu'):
\C^\times\rightarrow T=T_0\times \C^\times$, where $\nu':\C^\times
\rightarrow \C^\times$ is the identity homomorphism. From (\ref{eq:surjection_Cartan}),
it follows that
\begin{equation}\label{eq:surjection_Cartan1}
\Ca_{\nu'}\left(\Ca_{\nu_0}(\bar{\A}_\lambda(n,r))\right)\rightarrow
\Ca_{\nu'}\left(\Gamma(\Ca_{\nu_0}(\bar{\A}^\theta_\lambda(n,r)))\right)
\end{equation}
is surjective.

Since
$m\gg 0$, from \cite[Section 5.5]{catO_R}, we see that
\begin{equation}\label{eq:Cartan_iso3}
\Ca_{\nu'}(\Ca_{\nu_0}(\bar{\A}_\lambda(n,r)))\xrightarrow{\sim}
\Ca_{\nu}(\bar{\A}_\lambda(n,r)),
\end{equation}
moreover, both sides are the same quotients of $\bar{\A}_\lambda(n,r)^T$.

On the other hand, again using \cite[Section 5.5]{catO_R}, we see that,
for $\nu$ as in the previous paragraph,
the homomorphism $\Ca_\nu(\bar{\A}_\lambda(n,r))
\rightarrow \Ca_\nu(\bar{\A}_\lambda^\theta(n,r))$ factors through
$$\Ca_{\nu'}(\Ca_{\nu_0}(\bar{\A}_\lambda(n,r)))\rightarrow
\Ca_{\nu'}(\Gamma(\Ca_{\nu_0}(\bar{\A}^\theta_\lambda(n,r)))).$$
The direct summands of $\Gamma(\Ca_{\nu_0}(\bar{\A}^\theta_\lambda(n,r)))$
are tensor products of rational Cherednik algebras. Their parameters
are regular provided $\lambda$ is. We use Corollary \ref{Cor:localization_Hikita}
and Lemma \ref{Lem:Hikita} to see that
\begin{equation}\label{eq:Cartan_iso4}
\Ca_{\nu'}(\Gamma(\Ca_{\nu_0}(\bar{\A}^\theta_\lambda(n,r))))
\xrightarrow{\sim} \Ca_{\nu}(\bar{\A}_\lambda^\theta(n,r)).
\end{equation}

Using (\ref{eq:surjection_Cartan1}),(\ref{eq:Cartan_iso3}), (\ref{eq:Cartan_iso4}),
we conclude that $\Ca_\nu(\bar{\A}_\lambda(n,r))\twoheadrightarrow
\Ca_\nu(\bar{\A}_\lambda^\theta(n,r))$. Using Theorem \ref{Thm:O_regular_precise},
we see that $\lambda$ is $\Ocat$-regular.
\end{proof}

\subsection{Reduction to rank 1}
In this section we will prove (1) from Section \ref{SS_quiver_A_main}.
Let $\Upsilon$ be a classical wall and $\tilde{\Upsilon}$ be a non-singular essential hyperplane
parallel to $\Upsilon$.

\subsubsection{Real root case}\label{SSS_real_rk1}
Suppose, first, that $\Upsilon$ is defined by a real root. Let $\zeta\in\Upsilon$
be as in Proposition \ref{Prop:rank1}. Then the corresponding slice varieties $\underline{X}$
are all of the form $T^*\operatorname{Gr}(k,n)$. Using Corollary \ref{Cor:localization_Hikita}
and Lemma \ref{Lem:Hikita}, we see that $\underline{\paramq}^{O-sing}=
\underline{\paramq}^{sing}$. Now we can use Proposition \ref{Prop:O_reg_rk1}
to show that a Zariski generic point in $\tilde{\Upsilon}$ is $\Ocat$-regular.

\subsubsection{Imaginary root case}\label{SSS_imaginary_rk1}
Note that $\underline{X}$ is the product of Gieseker varieties
whose framing spaces are all identified with $\bigoplus_{i\in Q_0}W_i$.
This follows from Section \ref{SSS_slices_quiver}.

We can argue as in the case of real roots by using Proposition
\ref{Prop:Gies_O_sing} instead of Lemma \ref{Lem:Hikita}, if we
know that the torus action on $\underline{X}$ induced by $\nu$
is the product of the actions considered in Section \ref{SS_Gies_O_reg}.

Take $\nu_0: \C^\times\rightarrow \prod_{i\in Q_0}\operatorname{GL}(W_i)$
with weights far away from one another.

Now pick a generic $\zeta\in \Upsilon$ and  a $T_0$-stable point $y\in Y_\zeta$.
Let $r\in \mu^{-1}(\zeta)$ be an element from the closed $G$-orbit over $y$,
i.e. a semisimple representation. Its stabilizer in $G\times T_0$ projects
surjectively onto $T_0$. We can pick a section of the projection. This gives
an action of $T_0$ on $T^*V$ fixing $r$.

Recall the construction of the slice representation $T^*\underline{V}$.
We have $T^*\underline{V}\oplus (\g/\g_r)^{\oplus 2}\cong T^*V$
up to summands that are trivial over $G_r$.
Recall that $G_r\cong \prod_{i=1}^k \GL_{n_i}$, see
Section \ref{SSS_sympl_leaves}. From here we see that
the slice module, as a module over $(G\times T_0)_r$,
has the form $\bigoplus_{i=1}^k \mathfrak{sl}_{n_i}^{\oplus 2}\oplus
\bigoplus_{i\in Q_0}[\Hom(V_i,W_i)\oplus \Hom(W_i,V_i)]$.
Up to trivial over $G_r$ summands, the second direct sum is
$\bigoplus_{i=1}^k (\Hom(\C^{n_i},W)\oplus \Hom(W,\C^{n_i}))$.
The action of $T_0$ on this $\Hom(\C^{n_i},W)$ is induced by a character on
$\C^{n_i}$ and the initial action on $W$. The variety $\underline{X}$
is the product of Gieseker moduli spaces and the action of $\nu_0$
coincides with the product action of the one-parameter subgroups
denoted by $\nu_0$ in Section \ref{SSS_Gies_sing_proof}.

From here we see that the one-parameter subgroup $\nu=m\nu_0+\nu'$, where $m\gg 0$ and $\nu'$ acts on the cycle, gives rise to the one-parameter subgroup acting on $\underline{X}$ as $m\nu_0+\nu'$
from the proof of Proposition \ref{Prop:Gies_O_sing}.
Using Proposition \ref{Prop:O_reg_rk1} combined with
Proposition \ref{Prop:Gies_O_sing}, we conclude that a Zariski generic parameter on $\tilde{\Upsilon}$ is $\Ocat$-regular.

\subsubsection{$\Gamma^\theta_\lambda$ is an equivalence between the categories $\mathcal{O}$}
So at this point we know that a Zariski generic parameter on every non-singular
essential hyperplane is non-singular, see Sections \ref{SSS_real_rk1}
and \ref{SSS_imaginary_rk1}. For every regular $\lambda$, there is
$\theta$ such that $\Gamma^\theta_\lambda$ is exact, this follows
from Theorem \ref{Thm:exactness_precise}. So we see that $\paramq^{\Ocat-sing}
\subset \paramq^{sing}$. And if $\Gamma^\theta_\lambda$ is exact and $\lambda\in \paramq^{\Ocat-reg}$,
then $\Gamma^\theta_\lambda:\Ocat_\nu(\A^\theta_\lambda)\xrightarrow{\sim}
\Ocat_\nu(\A_\lambda)$. This finishes the proof of (1) from Section
\ref{SS_quiver_A_main}.

\subsection{From categories $\Ocat$ to all modules}
Note that every slice variety $\underline{X}$ over $0\in \param$ is again a finite or
affine type $A$ quiver variety, Lemma \ref{Lem:type_A_slices}.
Let $\underline{\Gamma}^\theta_\lambda$ denote the
global section functor $\Coh(\underline{\A}^\theta_\lambda)\rightarrow
\underline{\A}\operatorname{-mod}$. Let $\eta$ denote the natural map
$\paramq\rightarrow \underline{\paramq}$.

\begin{Lem}\label{Lem:singular_hyperplane_slice}
The $\underline{\tilde{\Upsilon}}$ be a singular hyperplane for $\underline{X}$. Then
$\eta^{-1}(\underline{\tilde{\Upsilon}})$ is a singular hyperplane for $X$.
\end{Lem}
\begin{proof}
Assume the contrary: $\tilde{\Upsilon}:=\eta^{-1}(\underline{\tilde{\Upsilon}})$ is non-singular.
Equivalently, by Lemma \ref{Lem:ab_loc_Weil_generic}, for a Weil generic point $\lambda\in \tilde{\Upsilon}$, there is
a regular element $\theta\in \param_{\Q}$ such that abelian localization
holds for $\A_\lambda^\theta$. By Corollary \ref{Cor:loc_slice},
abelian localization holds for $\underline{\A}^{\eta(\theta)}_{\eta(\lambda)}$.
We can pick an ample line bundle $\chi$ on $X$ such that abelian localization
holds for $\underline{\A}^{\eta(\theta)}_{\lambda'}$, where $\lambda'$ is either a Weil generic point
in $\underline{\tilde{\Upsilon}}+\eta(\chi)$ or $\lambda'=\eta(\lambda+\chi)$.
The locus of $\underline{\lambda}\in \underline{\tilde{\Upsilon}}$ such that
$\A_{\underline{\lambda},\eta(\chi)}, \A_{\underline{\lambda}+\eta(\chi),-\eta(\chi)}$
are mutually inverse Morita equivalence bimodules is Zariski open, by Lemma
\ref{Lem:inverse_Morita}.
For such $\underline{\lambda}$, if abelian localization holds for
$\underline{\A}^{\eta(\theta)}_{\underline{\lambda}+\eta(\chi)}$,
then it also holds for  $\underline{\A}^{\eta(\theta)}_{\underline{\lambda}}$, compare
to Lemma \ref{Lem:abloc_Morita}.
We conclude that abelian localization holds for $\underline{\A}^{\eta(\theta)}_{\underline{\lambda}}$
with a Weil generic $\underline{\lambda}\in \underline{\tilde{\Upsilon}}$. Hence $\underline{\tilde{\Upsilon}}$
is not singular. A contradiction.
\end{proof}

Now we prove the following proposition that finally implies
Theorem \ref{Thm:type_A}.

\begin{Prop}\label{Prop:O_to_mod}
Suppose that, for each leaf in $Y$, there is a one-parameter subgroup
$\underline{\nu}$ such that $\underline{\Gamma}:
\Ocat_{\underline{\nu}}(\underline{\A}_{\underline{\lambda}}^{\underline{\theta}})
\rightarrow \Ocat_{\underline{\nu}}(\underline{\A}_{\underline{\lambda}})$
is an equivalence as long as $\underline{\lambda}$ is regular and $\underline{\theta}$
lies in the chamber determined by $\underline{\lambda}$. Then,
for any regular $\lambda$ and any $\theta$ in the chamber determined by $\lambda$,
abelian localization holds for $\A_\lambda^\theta$.
\end{Prop}
\begin{proof}
Pick $\lambda$ and $\theta$ as in the statement of the lemma. We already know
that $\Gamma^\theta_\lambda$ is exact, by Theorem \ref{Thm:exactness_precise}.

By Lemma \ref{Lem:singular_hyperplane_slice}, $\eta(\lambda)$ is regular for every
slice.
So we can  assume that $\underline{\Gamma}$ is an equivalence
$\operatorname{Coh}(\underline{\A}_{\underline{\lambda}}^{\underline{\theta}})
\rightarrow \underline{\A}_{\underline{\lambda}}\operatorname{-mod}$
for all proper slices.

Now pick $\chi$ in the chamber of $\theta$ such that abelian
localization holds for $\A_{\lambda+\chi}^\theta$.
Since abelian localization holds for $\A_{\lambda+\chi}^\theta$ and $\Gamma_{\lambda}^\theta$
is exact, we see that under the identification of $\Coh(\A_{\lambda}^\theta)$
with $\A_{\lambda+\chi}\operatorname{-mod}$ as in Corollary \ref{Cor:abel_loc_bimod}, we have $\Loc_\lambda^\theta=\A_{\lambda,\chi}\otimes_{\A_\lambda}\bullet$
and $\Gamma_\lambda^\theta=\A_{\lambda+\chi,-\chi}\otimes_{\A_{\lambda+\chi}}\bullet$.
This is a consequence of Remark \ref{Rem:translation_coincidence}
and Lemma \ref{Lem:fun_iso1}. In particular, it follows that
$\A_{\lambda+\chi,-\chi}\otimes_{\A_{\lambda+\chi}}\A_{\lambda,\chi}
\xrightarrow{\sim} \A_\lambda$. We need to show that
\begin{equation}\label{eq:another_bimod_homom}\A_{\lambda,\chi}\otimes_{\A_{\lambda}}\A_{\lambda+\chi,-\chi}
\rightarrow \A_{\lambda+\chi}\end{equation}
is an isomorphism. Note that, since $\Gamma_\lambda^\theta:\Ocat_\nu(\A_\lambda^\theta)
\rightarrow \Ocat_\nu(\A_\lambda)$ is an equivalence, we have
\begin{equation}\label{eq:another_homom}\A_{\lambda,\chi}\otimes_{\A_{\lambda}}\A_{\lambda+\chi,-\chi}
\otimes_{\A_{\lambda+\chi}} M
\xrightarrow{\sim} M\end{equation}
for all $M\in \Ocat_\nu(\A_{\lambda+\chi})$, in particular, for all finite dimensional modules.


Now we claim that the kernel and the cokernel of (\ref{eq:another_bimod_homom})
are finite dimensional.
The restrictions (to a proper slice $\underline{Y}$) of
$\A_{\lambda,\chi}$ and $\A_{\lambda+\chi,-\chi}$ are $\underline{\A}_{\eta(\lambda),\eta(\chi)},
\underline{\A}_{\eta(\lambda+\chi),-\eta(\chi)}$ by Section
\ref{SSS_res_fun}. Since abelian localization holds for
$\underline{\A}^{\eta(\theta)}_{\eta(\lambda)},
\underline{\A}^{\eta(\theta)}_{\eta(\lambda+\chi)}$, tensoring with the bimodules $\underline{\A}_{\eta(\lambda),\eta(\chi)},
\underline{\A}_{\eta(\lambda+\chi),-\eta(\chi)}$ give mutually quasi-inverse equivalences
between the categories of all modules at parameters $\eta(\lambda+\chi),\eta(\lambda)$,
see Lemma \ref{Lem:abloc_Morita} and Remark \ref{Rem:usual_bimod_another}.
In particular, the analog of (\ref{eq:another_bimod_homom}) is an isomorphism.
Since this holds for any proper slice, we conclude that the claim in the beginning
of the paragraph holds.

Now we can prove that (\ref{eq:another_bimod_homom}) is an isomorphism. As we have remarked
above in this proof,
the functor $[\A_{\lambda,\chi}\otimes_{\A_{\lambda}}\A_{\lambda+\chi,-\chi}]\otimes_{\A_{\lambda+\chi}}\bullet$
is the identity endofunctor of $\A_\lambda\operatorname{-mod}_{fd}$. Since the cokernel of
(\ref{eq:another_bimod_homom}) is finite dimensional,
it follows that (\ref{eq:another_bimod_homom}) is surjective.

Let $K$ denote the kernel of (\ref{eq:another_bimod_homom}). It is finite dimensional. Restricting
the bimodule $\A_{\lambda,\chi}\otimes_{\A_{\lambda}}\A_{\lambda+\chi,-\chi}$
 to $Y^{reg}$
and taking the global sections of the restriction, we see that $K$ must split as a direct summand.
Since $[\A_{\lambda,\chi}\otimes_{\A_{\lambda}}\A_{\lambda+\chi,-\chi}]\otimes_{\A_{\lambda+\chi}}\bullet$
is the identity endofunctor of $\A_\lambda\operatorname{-mod}_{fd}$, we see that
the tensor product of $K$ with any finite dimensional $\A_\lambda$-module is zero.
Since $K$ itself is finite dimensional, we conclude that $K=0$.

So $\A_{\lambda,\chi},\A_{\lambda+\chi,-\chi}$ are mutually inverse Morita equivalence bimodules.
By Lemma \ref{Lem:abloc_Morita} and Remark \ref{Rem:usual_bimod_another}, abelian localization holds for $\A_\lambda^\theta$.
\end{proof}

\end{document}